\documentclass[reqno,11pt,letterpaper]{amsart}

\usepackage{amsmath}
\usepackage{amssymb}
\usepackage{amsthm}
\usepackage{mathrsfs}
\usepackage{accents}
\usepackage{calc}
\usepackage{arydshln}
\usepackage{upgreek}
\usepackage{slashed}
\usepackage{xifthen}
\usepackage{graphicx}
\usepackage{longtable}
\usepackage[inline]{enumitem}

\usepackage{xcolor}
\definecolor{winered}{rgb}{0.6,0,0}
\definecolor{lessblue}{rgb}{0,0,0.7}

\usepackage[pdftex,colorlinks=true,linkcolor=winered,citecolor=lessblue,breaklinks=true,bookmarksopen=true]{hyperref}

\makeatletter
\newcommand{\myitem}[2]{\item[\rm{(#2)}]\def\@currentlabel{#2}\label{#1}}
\makeatother

\usepackage{titletoc}

\makeatletter

\def\@tocline#1#2#3#4#5#6#7{
\begingroup
  \par
    \parindent\z@ \leftskip#3 \relax \advance\leftskip\@tempdima\relax
                  \rightskip\@pnumwidth plus 4em \parfillskip-\@pnumwidth
    \ifcase #1 
       \vskip 0.6em \hskip 0em 
       \or
       \or \hskip 0em 
       \or \hskip 1em 
    \fi%
    #6
    \nobreak\relax{\leavevmode\leaders\hbox{\,.}\hfill}
    \hbox to\@pnumwidth {\@tocpagenum{#7}}
  \par
\endgroup
}

 \def\l@section{\@tocline{0}{0pt}{0pc}{}{}}

\renewcommand{\tocsection}[3]{%
  \indentlabel{\@ifnotempty{#2}{ 
    \ignorespaces\bfseries{#2. #3}}}
  \indentlabel{\@ifempty{#2}{\ignorespaces\bfseries{#3}}{}} 
    \vspace{1.5pt}}

\renewcommand{\tocsubsection}[3]{%
  \indentlabel{\@ifnotempty{#2}{
    \ignorespaces#2. #3}}
  \indentlabel{\@ifempty{#2}{\ignorespaces #3}{}}
    \vspace{1.5pt}}

\renewcommand{\tocsubsubsection}[3]{%
  \indentlabel{\@ifnotempty{#2}{
    \ignorespaces#2. #3}}
  \indentlabel{\@ifempty{#2}{\ignorespaces #3}{}}
    \vspace{1.5pt}}

\makeatother

\makeatletter
\def\@nomenstarted{0}
\newlength{\@nomenoldtabcolsep}

\newcommand{\nomenstart}
  {%
    \def\@nomenstarted{1}%
    \setlength{\@nomenoldtabcolsep}{\tabcolsep}%
    \setlength{\tabcolsep}{3.5pt}%
    \begin{longtable}{p{0.11\textwidth} p{0.86\textwidth}}
  }

\newcommand{\nomenitem}[2]{%
    \ifcase\@nomenstarted%
      \or 
      \or \\ 
    \fi%
    #1\,{\leavevmode\leaders\hbox{\,.}\hfill} & #2%
    \def\@nomenstarted{2}%
  }%
\newcommand{\nomenend}
  {\\%
      \end{longtable}%
      \setlength{\tabcolsep}{\@nomenoldtabcolsep}%
      \def\@nomenstarted{0}%
  }
\makeatother

\makeatletter
\newcommand{\vast}{\bBigg@{4}}
\newcommand{\Vast}{\bBigg@{5}}
\newcommand{\VAST}[1]{\bBigg@{#1}}
\makeatother

\allowdisplaybreaks

\numberwithin{equation}{section}
\numberwithin{figure}{section}
\newtheorem{thm}{Theorem}[section]

\newtheorem{prop}[thm]{Proposition}
\newtheorem{lemma}[thm]{Lemma}
\newtheorem{cor}[thm]{Corollary}
\newtheorem*{thm*}{Theorem}
\newtheorem*{prop*}{Proposition}
\newtheorem*{cor*}{Corollary}
\newtheorem*{conj*}{Conjecture}

\theoremstyle{definition}
\newtheorem{definition}[thm]{Definition}

\theoremstyle{remark}
\newtheorem{rmk}[thm]{Remark}

\makeatletter
\newcommand{\fakephantomsection}{%
  \Hy@MakeCurrentHref{\@currenvir.\the\Hy@linkcounter}
  \Hy@raisedlink{\hyper@anchorstart{\@currentHref}\hyper@anchorend}%
  \Hy@GlobalStepCount\Hy@linkcounter%
}
\makeatother

\newcommand{\mc}{\mathcal}
\newcommand{\cA}{\mc A}

\newcommand{\cC}{\mc C}
\newcommand{\cD}{\mc D}
\newcommand{\cE}{\mc E}
\newcommand{\cF}{\mc F}

\newcommand{\cI}{\mc I}

\newcommand{\cM}{\mc M}

\newcommand{\cO}{\mc O}
\newcommand{\cP}{\mc P}

\newcommand{\cR}{\mc R}
\newcommand{\cS}{\mc S}

\newcommand{\cU}{\mc U}
\newcommand{\cV}{\mc V}

\newcommand{\cX}{\mc X}
\newcommand{\cY}{\mc Y}

\newcommand{\ms}{\mathscr}

\newcommand{\sC}{\ms C}

\newcommand{\C}{\mathbb{C}}
\newcommand{\N}{\mathbb{N}}
\newcommand{\R}{\mathbb{R}}
\newcommand{\Z}{\mathbb{Z}}

\newcommand{\Sph}{\mathbb{S}}

\newcommand{\sfb}{\mathsf{b}}

\newcommand{\sfr}{\mathsf{r}}

\newcommand{\sfw}{\mathsf{w}}

\newcommand{\sfH}{\mathsf{H}}

\newcommand{\sfZ}{\mathsf{Z}}

\newcommand{\bfA}{\mathbf{A}}

\newcommand{\bfR}{\mathbf{R}}

\newcommand{\slDelta}{\slashed{\Delta}{}}

\newcommand{\slpa}{\slashed{\partial}{}}

\newcommand{\End}{\operatorname{End}}

\renewcommand{\Re}{\operatorname{Re}}
\renewcommand{\Im}{\operatorname{Im}}
\newcommand{\Id}{\operatorname{Id}}

\newcommand{\supp}{\operatorname{supp}}
\newcommand{\sgn}{\operatorname{sgn}}

\newcommand{\Ups}{\Upsilon}

\newcommand{\eps}{\epsilon}
\newcommand{\ff}{\mathrm{ff}}

\newcommand{\la}{\langle}

\newcommand{\ol}{\overline}
\newcommand{\pa}{\partial}
\newcommand{\dd}{{\mathrm d}}
\newcommand{\ra}{\rangle}

\newcommand{\xra}{\xrightarrow}
\newcommand{\pfstep}[1]{$\bullet$\ \underline{\textit{#1}}}

\newcommand{\bop}{{\mathrm{b}}}
\newcommand{\cop}{{\mathrm{c}}}
\newcommand{\chop}{{\mathrm{c}\semi}}

\newcommand{\scop}{{\mathrm{sc}}}

\newcommand{\cl}{{\mathrm{cl}}}

\newcommand{\eop}{{\mathrm{e}}}

\newcommand{\semi}{\hbar}

\newcommand{\lb}{{\mathrm{lb}}}
\newcommand{\rb}{{\mathrm{rb}}}

\newcommand{\fbface}{{\mathrm{fbf}}}

\newcommand{\cface}{{\mathrm{cf}}}
\newcommand{\dface}{{\mathrm{df}}}

\newcommand{\sface}{{\mathrm{sf}}}

\newcommand{\tface}{{\mathrm{tf}}}

\newcommand{\cp}{{\mathrm{c}}}

\newcommand{\Diff}{\mathrm{Diff}}

\DeclareMathOperator{\Op}{Op}

\newcommand{\Vb}{\cV_\bop}
\DeclareMathOperator{\specb}{\mathrm{spec}_\bop}
\newcommand{\Diffb}{\Diff_\bop}

\newcommand{\Psib}{\Psi_\bop}
\newcommand{\Psisc}{\Psi_\scop}

\newcommand{\Psich}{\Psi_\chop}
\newcommand{\Vch}{\cV_\chop}
\newcommand{\Diffch}{\Diff_\chop}

\newcommand{\Tch}{{}^\chop T}
\newcommand{\Sch}{{}^\chop S}
\newcommand{\Ellch}{\Ell_\chop}
\newcommand{\WFch}{\WF_\chop}
\newcommand{\Hch}{H_{\cop,h}}

\newcommand{\Vsc}{\cV_\scop}

\newcommand{\Psih}{\Psi_\semi}

\newcommand{\WF}{\mathrm{WF}}
\newcommand{\Ell}{\mathrm{Ell}}

\newcommand{\Omegab}{{}^{\bop}\Omega}

\newcommand{\WFsc}{\WF_{\scop}}

\newcommand{\Tb}{{}^{\bop}T}

\newcommand{\Tsc}{{}^{\scop}T}

\newcommand{\Sb}{{}^{\bop}S}
\newcommand{\Ssc}{{}^{\scop}S}

\newcommand{\Ellh}{\Ell_{\semi}}

\newcommand{\half}{{\tfrac{1}{2}}}

\newcommand{\sigmab}{{}^\bop\sigma}
\newcommand{\sigmasc}{{}^\scop\sigma}

\newcommand{\sigmach}{{}^\chop\sigma}

\newcommand{\loc}{{\mathrm{loc}}}
\newcommand{\CI}{\cC^\infty}
\newcommand{\CIdot}{\dot\cC^\infty}

\newcommand{\CIc}{\cC^\infty_\cp}

\newcommand{\CmI}{\cC^{-\infty}}

\newcommand{\Hb}{H_{\bop}}

\newcommand{\Hbh}{H_{\bop,h}}

\newcommand{\Hsc}{H_{\scop}}
\newcommand{\Hsch}{H_{\scop,h}}

\newcommand{\openbigpmatrix}[1]
  {%
    \def\@bigpmatrixsize{#1}%
    \addtolength{\arraycolsep}{-#1}%
    \begin{pmatrix}%
  }
\newcommand{\closebigpmatrix}
  {%
    \end{pmatrix}%
    \addtolength{\arraycolsep}{\@bigpmatrixsize}%
  }

\newlength{\enummargin}

\newcommand{\usref}[1]{{\upshape\ref{#1}}}

\DeclareGraphicsExtensions{.mps}

\makeatletter
\newcommand*{\fwbw}[1]{\expandafter\@fwbw\csname c@#1\endcsname}
\newcommand*{\@fwbw}[1]{\ifcase #1 \or {\rm fw}\or {\rm bw}\fi}
\AddEnumerateCounter{\fwbw}{\@fwbw}
\makeatother

\begin{document}

\title[Semiclassical propagation through cone points]{Semiclassical propagation through cone points}

\date{\today}

\author{Peter Hintz}
\address{Department of Mathematics, Massachusetts Institute of Technology, Cambridge, Massachusetts 02139-4307, USA}
\email{phintz@mit.edu}

\begin{abstract}
  We introduce a general framework for the study of the diffraction of waves by cone points at high frequencies. We prove that semiclassical regularity propagates through cone points with an almost sharp loss even when the underlying operator has leading order terms at the conic singularity which fail to be symmetric. We moreover show improved regularity along strictly diffractive geodesics. Applications include high energy resolvent estimates for complex- or matrix-valued inverse square potentials and for the Dirac--Coulomb equation. We also prove a sharp propagation estimate for the semiclassical conic Laplacian.

  The proofs use the semiclassical cone calculus, introduced recently by the author, and combine radial point estimates with estimates for a scattering problem on an exact cone. A second microlocal refinement of the calculus captures semiclassical conormal regularity at the cone point and thus facilitates a unified treatment of semiclassical cone and b-regularity.
\end{abstract}

\maketitle

\section{Introduction}
\label{SI}

We present a systematic analysis of the propagation of semiclassical regularity through points which are geometrically singular (cone points), analytically singular (e.g.\ including inverse square potentials), or both. The novel aspect of our approach is that it handles leading order singular terms with ease, \emph{regardless of symmetry or sign conditions}.

As a simple application of our main microlocal propagation result, we consider high energy scattering by complex-valued potentials on $\R^n$ with an inverse square singularity. Denote by $H^2_0(\R^n\setminus\{0\})$ the closure of $\CIc(\R^n\setminus\{0\})$ in the topology of $H^2(\R^n)$; denote further by $\Delta=\sum_{j=1}^n D_{x^j}^2$ (where $D=\frac{1}{i}\pa$) the nonnegative Laplacian, and denote polar coordinates on $\R^n$ by $(r,\omega)\in(0,\infty)\times\Sph^{n-1}$.

\begin{thm}[High energy estimates for potential scattering]
\label{ThmIV}
  Let $V(x)=\frac{V_0(x)}{|x|^2}$, where $V_0=V_0(r,\omega)\in\CIc([0,\infty)_r\times\Sph^{n-1};\C)$ and $V_0(0,\omega)\equiv\sfZ\in\C$. (Thus $V(x)=\frac{\sfZ}{|x|^2}+\cO(|x|^{-1})$ near the origin.) If $n\geq 5$ and $\Re\sqrt{(\frac{n-2}{2})^2+\sfZ}>1$, there exists $\lambda_0>0$ so that for all $\lambda\in\C$ with $\Re\lambda>\lambda_0$ and $0<\Im\lambda<1$, the operator
  \begin{equation}
  \label{EqIV}
    \Delta+V-\lambda \colon H^2_0(\R^n\setminus\{0\}) \to L^2(\R^n)
  \end{equation}
  is invertible, and its inverse obeys the operator norm bound
  \begin{equation}
  \label{EqIVEst}
    \|\chi(\Delta+V-\lambda)^{-1}\chi\|_{L^2\to L^2} \leq C_{\chi,\eps}|\lambda|^{-\frac12+\eps}
  \end{equation}
  for all $\chi\in\CIc(X)$ and $\eps>0$. More generally, $\Delta+V-\lambda\colon\cD\to L^2(\R^n)$ is invertible for $n\geq 2$ and $\sfZ\in\C\setminus(-\infty,-(\frac{n-2}{2})^2]$ for a suitable domain $\cD$ (see~\eqref{EqHVPropMap} with $l=1$), and the estimate~\eqref{EqIVEst} holds in this generality as well.
\end{thm}

The point is that we can allow for $\sfZ$ to be nonreal, in which case $\Delta+V$ is not a symmetric operator on $\CIc(\R^n\setminus\{0\})$. For a general result for matrix-valued inverse square potentials without symmetry conditions, see Theorem~\ref{ThmHVProp}; Lemma~\ref{LemmaHVScalar} verifies the assumptions of Theorem~\ref{ThmHVProp} for the case considered in Theorem~\ref{ThmIV}. Typical applications of high energy resolvent estimates include decay and local smoothing estimates for solutions to wave and Schr\"odinger equations; since such applications are orthogonal to the focus of the present paper, we shall not discuss them here.

Burq and Planchon--Stalker--Tahvildar-Zadeh proved Strichartz estimates for exact inverse square potentials in the case of real $\sfZ>-(\frac{n-2}{2})^2$ \cite{PlanchonStalkerTahvildarZadehInvSq,BurqPlanchonStalkerTahvildarZadehInvSq}. Duyckaerts \cite{DuyckaertsInvSq} obtained, by means of estimates for semiclassical defect measures, high energy resolvent estimates (without the $\eps$-loss) in the more general setting of inverse square potentials at a finite collection of points $p_j$ in $\R^n$, at each of which the coefficient $\sfZ_j$ satisfies $\sfZ_j>-(\frac{n-2}{2})^2$. We also mention the work by Baskin--Wunsch \cite{BaskinWunschConicDecay} on lossless resolvent estimates in a \emph{geometric} setting, namely in the presence of finitely many conic singularities, and the work by Hillairet--Wunsch \cite{HillairetWunschConic} on resonances in this setting (see also \cite{GalkowskiVainberg}).

\begin{rmk}[More natural settings]
  The setting of Theorem~\ref{ThmIV} is chosen here for its simplicity. More natural examples in which leading order terms without signs or symmetry properties are present arise in particular in the study of PDEs on vector bundles. As an example, motivated by the recent work of Baskin--Wunsch \cite{BaskinWunschDiracCoulomb}, we prove high energy resolvent estimates for the Dirac--Coulomb equation in~\S\ref{SsHDC}, see Theorem~\ref{ThmHDC}.
\end{rmk}

The heart of the proof of Theorem~\ref{ThmIV} is the propagation of semiclassical regularity through $r=0$,\footnote{In particular, the choice of the large end of the space, here $\R^n$, is only made for convenience and allows for simple control of the global structure of the geodesic flow. Thus, we do not discuss the large literature on limiting absorption principles here.} which we prove in this paper for a general class of \emph{admissible operators}, see Definition~\ref{DefPOp} and Theorem~\ref{ThmP}. Thus, in addition to inverse square singularities (which may be anisotropic), we allow for the underlying metric $g$ to have a conic singularity at $r=0$, so $g=\dd r^2+r^2 k(r,y,\dd r,\dd y)$ for some smooth $r$-dependent tensor $k$, with $k|_{r=0}$ a Riemannian metric on a closed manifold $Y$. We moreover allow for further first order differential operators of the schematic form $r^{-1}D_r$, $r^{-2}D_y$ to be present. All these singular terms are allowed to be of the same strength at $r=0$: they are, to leading order at $r=0$, homogeneous of degree $-2$ with respect to dilations.

In order to appreciate Theorem~\ref{ThmIV}, note that the degree $-2$ homogeneity of the Laplacian and of the potential $r^{-2}$ is reflected also in the Hardy inequality, which demonstrates that any factor of $r^{-1}$ should be regarded as a derivative as far as analysis near the cone point $r=0$ is concerned. Therefore, when $\sfZ$ in Theorem~\ref{ThmIV} is nonreal, the operator $\Delta+V-\lambda$ is, \emph{even to leading order at the cone point}, not symmetric. Therefore, techniques rooted in the spectral theory of self-adjoint operators do not apply. Furthermore, recall that even for solutions of smooth coefficient PDEs $P u=f$ where the principal symbol of $P$ is complex-valued, microlocal regularity of $u$ propagates along the null-bicharacteristics of $\Re P$ \emph{only under a sign condition on $\Im P$} near the boundary of the support of $\Im P$ \cite[\S4.5]{VasyMinicourse}; on a technical level, the term $\Im P$ contributes the leading term in a positive commutator argument for proving the propagation of regularity along null-bicharacteristics of $\Re P$. The absence of sign conditions on $\Im V$ in Theorem~\ref{ThmIV} is thus a significant obstacle for the applicability of existing methods.

In general geometric or analytic settings where one cannot separate variables, propagation estimates through cone points and other types of singularities have so far largely been restricted to self-adjoint settings. Melrose--Wunsch \cite{MelroseWunschConic} studied the diffraction of waves by conic singularities by combining microlocal propagation estimates in Mazzeo's edge calculus \cite{MazzeoEdge} with the inversion of a suitable model operator on an exact cone. This point of view is closely related to that adopted in the present paper, see Remark~\ref{RmkIEdge}, though by contrast to the present work, \cite{MelroseWunschConic} takes full advantage of the self-adjointness of the underlying Laplace operator.

Later works on wave propagation in singular geometries have been based on positive commutator arguments relative to a quadratic form domain (thus still in self-adjoint settings), following the blueprint of Vasy's work \cite{VasyPropagationCorners} on the propagation of singularities on smooth manifolds with corners (see Lebeau \cite{LebeauPropagation} for the analytic setting). Vasy's work was extended to the setting of manifolds with edge singularities by Melrose--Vasy--Wunsch \cite{MelroseVasyWunschEdge}, and the same authors established improved regularity of the strictly diffracted front on manifolds with corners \cite{MelroseVasyWunschDiffraction}. See Qian \cite{QianDiffractionInvSq} for the case of inverse square potentials. We remark that in these works, the underlying geometry near the singularity is \emph{not} reflected in the type of singularities which propagate or diffract---for instance, in the case of \cite{VasyPropagationCorners}, the geometry is that of a manifold with corners equipped with a smooth (incomplete!) Riemannian metric, but the correct notion of regularity is conormality at the boundary; thus, these works introduce mixed differential-pseudodifferential calculi which are compatible with both structures.

Baskin--Marzuola \cite{BaskinMarzuolaCone} combined the techniques of \cite{VasyPropagationCorners} with \cite{BaskinVasyWunschRadMink} to study the long-time behavior of waves on manifolds with conic singularities. An important ingredient in their work is a high energy estimate for propagation through the conic singularity. In the present paper we give an alternative proof which in particular avoids the use of a mixed calculus; see also Remark~\ref{RmkI2nd}. We also mention that Gannot--Wunsch \cite{GannotWunschPotential} analyzed the diffraction by conormal potentials in the semiclassical setting using direct commutator methods involving paired Lagrangian distributions, inspired by \cite{DeHoopUhlmannVasyDiffraction}.

The recent work by Baskin--Wunsch \cite{BaskinWunschDiracCoulomb} on diffraction for the Dirac--Coulomb equation is also rooted in \cite{MelroseWunschConic,MelroseVasyWunschEdge}. While the (first order) Dirac--Coulomb operator is self-adjoint for the range of Coulomb charges considered in \cite{BaskinWunschDiracCoulomb}, the wave type operator obtained by taking an appropriate square has nonsymmetric leading order terms at the central singularity; thus, the authors work directly with the first order operator in their proofs of propagation results. We are able to give a direct proof of high energy estimates for the resolvent associated with the wave type operator arising in \cite{BaskinWunschDiracCoulomb}, see \S\ref{SsHDC}.

In the high energy regime under study in the present paper, the strategy for overcoming the issues caused by the absence of symmetry or self-adjointness properties is the following. We distill the contribution of $V$ to the high frequency propagation of regularity (i.e.\ in Theorem~\ref{ThmIV}: the inverse powers of $|\lambda|$ appearing in uniform estimates of $L^2$ norms) into a model problem right at the cone point, thus decoupling it from the real principal type propagation away from the cone point (where $V$ plays no role due to its subprincipal nature). More precisely, in the setting of Theorem~\ref{ThmIV}, set $h:=|\lambda|^{-\frac12}$ and $z=h^2\lambda=1+\cO(h)$, and define the semiclassical rescaling
\begin{equation}
\label{EqIP}
\begin{split}
  P_{h,z} = h^2(\Delta+V-\lambda) &= h^2\Delta - z + \tfrac{h^2}{|x|^2} V_0 \\
    &= (h D_r)^2-i(n-1)\tfrac{h}{r}h D_r + h^2 r^{-2}\Delta_{\Sph^{n-1}} - z + \tfrac{h^2}{r^2}V_0.
\end{split}
\end{equation}
This is a semiclassical differential operator in $r>0$. Its uniform analysis as $h\to 0$, as far as the novel bit near $r=0$ is concerned, is based on two ingredients, discussed in more detail in~\S\ref{SsIS}.

\begin{enumerate}
\item \textit{Symbolic propagation estimates:} real principal type propagation in $r>0$ in the spirit of \cite{DuistermaatHormanderFIO2}, and radial point estimates down to $r=0$ in the spirit of \cite{MelroseEuclideanSpectralTheory,VasyMicroKerrdS} but taking place in the semiclassical cone algebra introduced by the author in \cite{HintzConicPowers}. The advantage of this algebra in the present setting is that $P_{h,z}$ has a smooth and nondegenerate principal symbol in this algebra down to $r=0$; in this algebra, the proofs of the relevant symbolic estimates are then essentially standard.
\item \textit{Inversion of a model problem.} Passing to the rescaled variable $\hat r=\frac{r}{h}$ and letting $h\to 0$ for fixed $\hat r$ in the resulting expression of $P_{h,z}$ gives
  \begin{equation}
  \label{EqINormOp}
    N(P) = D_{\hat r}^2 - i(n-1)\hat r^{-1}D_{\hat r}+\hat r^{-2}\Delta_{\Sph^{n-1}} - 1 + \tfrac{\sfZ}{\hat r^2}.
  \end{equation}
  The inversion of $N(P)$ is a scattering problem on an exact cone at unit frequency, and requires the existence of the limiting (outgoing) resolvent. Its analysis is based on b-analysis near the small end of the cone \cite{MelroseAPS} and on the microlocal approach to scattering theory on spaces with conic infinite ends pioneered by Melrose \cite{MelroseEuclideanSpectralTheory}.
\end{enumerate}

The $\eps$-loss in the estimate~\eqref{EqIVEst} is then due to the analogous loss in the limiting absorption principle for the scattering problem, as one needs to exclude incoming but allow outgoing spherical waves, cf.\ Remark~\ref{RmkHLLAP}. (We shall in fact deduce the lossy estimates stated in Theorem~\ref{ThmIV} from sharp results---as far as the relationship of domain and codomain of $P_{h,z}$ is concerned---on spaces with variable semiclassical orders.) For general admissible operators, the decay rates of incoming and outgoing solutions of the model problem are typically different, and the semiclassical loss upon propagation through the cone point is equal to their difference (up to an additional $\eps$-loss); we give explicit examples in which this loss indeed occurs in Appendix~\ref{SLoss}, demonstrating that our analysis is sharp up to an $\eps$-loss.

The close connection between diffraction by conic singularities and scattering on large ends of cones was recently studied for exact (or `product') cones (i.e.\ the metric is $g=\dd r^2+r^2 k(y,\dd y)$) by Yang \cite{YangDiffraction}, resulting in a partial improvement of the classical analysis by Cheeger--Taylor \cite{CheegerTaylorConicalI,CheegerTaylorConicalII} which was based on separation of variables and Bessel function analysis. Recently, Chen Xi \cite{XiConeParametrix} constructed a detailed parametrix for high frequency diffraction by (nonexact) conic singularities, i.e.\ for the operator $(h^2\Delta_g-(1\pm i 0))^{-1}$, with applications to short time Strichartz estimates for the Schr\"odinger equation; an important ingredient in his work is the precise resolvent construction by Guillarmou--Hassell--Sikora \cite{GuillarmouHassellSikoraResIII}, applied on an exact cone which arises similarly to~\eqref{EqINormOp}. (The history of the study of propagation and diffraction phenomena for solutions of wave type equations on manifolds with singularities is long, starting with Sommerfeld's example \cite{SommerfeldDiffraction} and early developments by Friedlander \cite{FriedlanderSound} and Keller \cite{KellerDiffraction}. The use of geometric and microlocal techniques for the analysis of singularities goes back to work on manifolds with boundary by Melrose--Sj\"ostrand \cite{MelroseSjostrandSingBVPI,MelroseSjostrandSingBVPII} using commutator techniques, and Melrose and Taylor \cite{MelroseDiffractiveParametrix,TaylorGrazing,MelroseTaylorParametrix} using parametrix constructions.)

For operators $P_{h,z}=h^2\Delta_g-z$, $z=1+\cO(h)$, on (nonexact) conic manifolds, we are able to obtain a \emph{lossless} propagation estimate by means of a positive commutator argument which is global on the level of the normal operator $N(P)$, i.e.\ which involves the construction of a commutator which is positive \emph{as an operator} on an exact cone, in the spirit of Mourre's construction \cite{MourreSingular} and Vasy's approach to many-body scattering \cite{VasyThreeBody,VasyManyBody}; see Theorem~\ref{ThmHL}, in particular the estimate~\eqref{EqHLSharpCAP}.

Finally, we prove a \emph{diffractive improvement} which gives finer control on the strength of singularities as they propagate through the cone point. Combining our framework with the arguments in \cite{MelroseWunschConic,MelroseVasyWunschEdge} for the propagation of coisotropic regularity, we show that, under a nonfocusing condition, the strongest singularities propagating towards the cone point only continue along geometric geodesics (limits of geodesics barely missing the cone point), whereas away from those, the diffracted front is smoother; see~\S\ref{SsPD}.

\subsection{Sketch of the proof}
\label{SsIS}

Consider again the operator $P_{h,z}$ from~\eqref{EqIP}; we work locally near $r=0$, thus on $X=[0,1)_r\times\Sph^{n-1}$. In order to achieve a clean separation of the regimes $h\to 0$, $r>0$ (corresponding to semiclassical analysis away from the cone point) and $h\sim r\to 0$ (where the normal operator $N(P)$ in~\eqref{EqINormOp} enters and semiclassical tools cease to be applicable), we work on a resolution of the total space $[0,1)_h\times X$ obtained by a real blow-up of $h=r=0$,\footnote{Recall here that the real blow-up gives an invariant way of introducing polar coordinates around $\{0\}\times\pa X$. Thus, a neighborhood of $h=r=0$ in $X_\chop$ is diffeomorphic to $[0,1)_\rho\times[0,\frac{\pi}{2}]_\theta\times\pa X$ and equipped with a smooth map (the blow-down map) to $[0,1)\times X=[0,1)\times([0,1)\times\Sph^{n-1})$ given by $(\rho,\theta,y)\mapsto(\rho\sin\theta,(\rho\cos\theta,y))$ which is a diffeomorphism away from the front face $\rho^{-1}(0)$. In practice, it is more convenient to work with the smooth functions $h+r$, $\frac{h}{h+r}$ on $X_\chop$ instead of $\rho,\theta$.}
\[
  X_\chop := \bigl[ [0,1)_h \times X; \{0\}\times\pa X \bigr].
\]
See Figure~\ref{FigISXch}. We wish to regard $\frac{h}{h+r}$ as the `true' semiclassical parameter; we proceed to make this more precise.

\begin{figure}[!ht]
\centering
\includegraphics{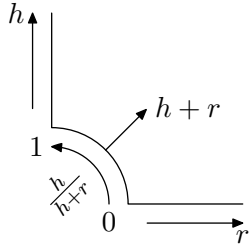}
\caption{The semiclassical cone single space $X_\chop$.}
\label{FigISXch}
\end{figure}

Note first that for $h=1$, the rescaling $r^2 P_{1,z}$ is a Fuchs-type operator, or b-differential operator in the terminology of Melrose \cite{MelroseAPS}, namely a differential operator built out of the vector fields $r\pa_r$ and $\pa_y$ (which span the space of \emph{b-vector fields}), where $y\in\R^{n-1}$ denotes local coordinates on $\pa X$. In this sense, the rescaled operator $r^2 P_{1,z}$ has elliptic principal part given by $(r D_r)^2+k^{i j}D_{y^i}D_{y^j}$, where $k^{i j}$ is the inverse metric on $\Sph^{n-1}$. As $h$ tends to $0$, the operator $r^2 P_{h,z}$ is built out of the semiclassical vector fields $h r\pa_r$ and $h\pa_y$ (which span the space of \emph{semiclassical b-vector fields}). In this semiclassical sense (i.e.\ ignoring terms with extra powers of $h$), its principal part is
\[
  r^2 P_{h,z} \sim (h r D_r)^2+k^{i j}h D_{y^i}h D_{y^j}-r^2 z.
\]
The characteristic set, i.e.\ the zero set of its principal symbol $\xi_{\bop\semi}^2+|\eta_{\bop\semi}|^2-r^2$, becomes singular at $r=0$, which is indicative of the inadequacy of the semiclassical b-setting to capture the behavior of $P_{h,z}$ microlocally near $h=r=0$ (cf.\ the above discussion regarding the tension between the geometry and the notion of regularity in \cite{MelroseWunschConic,VasyPropagationCorners} and subsequent works). The way out is to divide by $(h+r)^2$ and thus consider
\[
  \bigl(\tfrac{r}{h+r}\bigr)^2 P_{h,z} =: p_{h,z}(r,y,\tfrac{h}{h+r}r D_r,\tfrac{h}{h+r}D_y)
\]
as a differential operator built out of $\frac{h}{h+r}r\pa_r$ and $\frac{h}{h+r}\pa_{y^i}$, which are the prototypical \emph{semiclassical cone vector fields} introduced in \cite{HintzConicPowers}, see~\S\ref{SsCV}. In this sense, the principal part of $(\frac{r}{h+r})^2 P_{h,z}$ (i.e.\ ignoring terms of size $\cO(\frac{h}{h+r})$) is
\[
  \bigl(\tfrac{r}{h+r}\bigr)^2 P_{h,z} \sim \bigl(\tfrac{h}{h+r}r D_r\bigr)^2 + k^{i j}\bigl(\tfrac{h}{h+r}D_{y^i}\bigr)\bigl(\tfrac{h}{h+r}D_{y^j}\bigr) - \bigl(\tfrac{r}{h+r}\bigr)^2 z.
\]
Note that in the regime $\tfrac{h}{h+r}\ll 1$, where we are aiming to use semiclassical methods, this is now nondegenerate in the sense that its principal symbol
\[
  p_{0,1}(r,y,\xi,\eta)=\xi^2+|\eta|^2-1
\]
(recall $z=1+\cO(h)$) has a smooth zero set on which $p_{0,1}$ vanishes simply. (The microlocal analysis of \emph{semiclassical cone operators} in the semiclassical regime is thus concerned with tracking amplitudes of oscillations $r^{\frac{i}{h/(h+r)}\xi}e^{\frac{i}{h/(h+r)}\eta\cdot y}$ through the phase space over $X_\chop$---more precisely: over the `semiclassical face' $\frac{h}{h+r}=0$---whose fiber variables are $(\xi,\eta)$.)

The semiclassical cone calculus $\Psich(X)$, introduced in \cite{HintzConicPowers} and developed further in~\S\ref{SC}, makes this rigorous. It allows for the symbolic analysis of pseudodifferential operators of the form
\[
  \Op_{\cop,h}(p) = \text{``}p(\tfrac{h}{h+r},h+r,y,\tfrac{h}{h+r}r D_r,\tfrac{h}{h+r}D_y)\text{''}
\]
using standard methods from microlocal analysis: there is a semiclassical principal symbol $p(0,r,y,\xi,\eta)$, which is a symbol on the aforementioned phase space (defined rigorously after Lemma~\ref{LemmaCVSpan}). Moreover, as usual, the commutator $i[\Op_{\cop,h}(p),\Op_{\cop,h}(q)]$ is given by the quantization of the Poisson bracket of $p$ and $q$ up to operators with an extra factor of $\frac{h}{h+r}$. For the operator $P_{h,z}$ in~\eqref{EqIP}, the Hamilton vector field of its principal symbol is nondegenerate except at two submanifolds of critical points over $r=0$; these critical sets are saddle points for the Hamilton flow, and are end or starting points of geodesics hitting the cone point or emanating from it. (See Figure~\ref{FigPFlow}.) One can thus prove quantitative microlocal propagation and radial point estimates on the associated scale of semiclassical cone Sobolev spaces, which measure $L^2$ norms of derivatives along $\frac{h}{h+r}r D_r$, $\frac{h}{h+r}D_y$.

\begin{rmk}[Semiclassical cone ps.d.o.s as tools]
  The (large) calculus $\Psich(X)$ was introduced in~\cite{HintzConicPowers} as the space in which inverses and complex powers of elliptic semiclassical cone operators, such as $h^2\Delta+1$, live; the goal there was a precise description of their Schwartz kernels. Here, by contrast, we use semiclassical cone ps.d.o.s as \emph{tools} to understand propagation phenomena. Correspondingly, we only need to consider the \emph{small} semiclassical cone calculus, as our analysis will be based on proving estimates, rather than on the construction and usage of parametrices. (Parametrices are typically significantly more challenging to construct \cite{XiConeParametrix} and are very precise tools; on the flipside, they tend to be less convenient when the need for generalizations or for proofs of sharp mapping properties on various function spaces arises.) Thus, in~\S\ref{SC}, we provide a perspective on $\Psich(X)$ which makes it easy to work with in nonelliptic settings.
\end{rmk}

At this point, we control semiclassical singularities of solutions of $P_{h,z}u=f$ at $\frac{h}{h+r}=0$. In order to control $u$ globally, including at $h=r=0$, one needs to invert the \emph{normal operator} $N(P)$ of $P_{h,z}$, which is the restriction of $P_{h,z}$ to the front face of $X_\chop$; see~\eqref{EqIP}--\eqref{EqINormOp} for a concrete example. The function spaces on which one inverts $N(P)$ need to match the function spaces in which the symbolic propagation estimates are obtained. As already observed in \cite{HintzConicPowers} (see also the earlier paper \cite{LoyaConicResolvent}) and demonstrated in detail on the level of function spaces in~\S\ref{SsCPN}, the correct function spaces for $N(P)$ are standard Sobolev spaces when $\hat r=\frac{r}{h}\gtrsim 1$ (i.e.\ measuring regularity with respect to $D_{\hat r}$ and $\hat r^{-1}D_{y^i}$) and b-Sobolev spaces in $\hat r\lesssim 1$ (i.e.\ measuring regularity with respect to $\hat r D_{\hat r}$ and $D_{y^i}$). Following \cite{MelroseEuclideanSpectralTheory}, we show in~\S\ref{SsPP} that the analysis of $N(P)$ on spaces with variable orders of decay as $\hat r\to\infty$ \emph{precisely} matches the above symbolic analysis which involves variable semiclassical orders (powers of $\frac{h}{h+r}$) to accommodate the threshold requirements for propagation into/out of the radial sets, cf.\ \cite[Appendix~E.4]{DyatlovZworskiBook}.

We stress the \emph{global} (rather than microlocal or symbolic) nature of the requirement that the normal operator $N(P)$ be invertible; while verifying this in concrete situations is nontrivial, one has many standard techniques at one's disposal (such as boundary pairing arguments, unique continuation, separation of variables, etc).

In combination, the symbolic estimates and the normal operator invertibility provide control of $u$ at both hypersurfaces $\frac{h}{h+r}=0$ and $h+r=0$ of $X_\chop$. Thus, we have uniform control as $h\to 0$. (One can package this into an invertibility statement for a modification of $P_{h,z}$ by placing complex absorbing potentials away from $r=0$ in the spirit of \cite{NonnenmacherZworskiQuantumDecay,WunschZworskiNormHypResolvent,DatchevVasyGluing,VasyMicroKerrdS}, see~\S\ref{SsPA}.)

\begin{rmk}[Relation to edge propagation]
\label{RmkIEdge}
  The proof of symbolic propagation estimates for wave equations on conic or edge manifolds using the edge calculus \cite{MazzeoEdge}, as done in \cite[\S8]{MelroseWunschConic} and \cite[\S11]{MelroseVasyWunschEdge}, is closely related, via the Fourier transform in time, to the semiclassical cone Sobolev spaces associated with $\Psich(X)$; see Remark~\ref{RmkCPEdge}. Going one step further in the comparison, we note that the fine analysis of diffraction by Melrose and Wunsch \cite{MelroseWunschConic} for waves on a conic manifolds uses a normal operator at the cone point which is defined via a rescaled FBI (Fourier--Bros--Iagolnitzer) transform in time; this normal operator is thus equivalent to the operator $N(P)$ considered here, but used in a different manner.
\end{rmk}

\begin{rmk}[Second microlocalization]
\label{RmkI2nd}
  Writing $h r D_r=(h+r)\frac{h}{h+r}r D_r$ and $h D_y=(h+r)\frac{h}{h+r}D_y$ suggests that semiclassical conormal regularity at the cone point (regularity under application of $h r D_r$ and $h D_y$) can be captured on the scale of semiclassical cone Sobolev spaces as well. We present a systematic second microlocal perspective on this in~\S\ref{SsCb}, inspired by recent work of Vasy on the limiting absorption principle on asymptotically conic manifolds \cite{VasyLowEnergy,VasyLAPLag} (with the conic nature referring to the \emph{large} end of the manifold). In view of the characterization of the quadratic form domain of $h^2\Delta_g+1$ as a semiclassical cone Sobolev space in \cite[Theorem~6.1]{HintzConicPowers}, we can thus eliminate the need of working with a mixed differential-pseudodifferential calculus as in \cite{BaskinMarzuolaCone}, and instead work in a single microlocal framework.
\end{rmk}

\subsection{Outline of the paper}

In~\S\ref{Sbsc}, we review basic notions from b- and scattering analysis, with an eye towards the relationship with semiclassical cone analysis. In~\S\ref{SC}, we describe a hands-on perspective on the semiclassical cone algebra $\Psich(X)$ with a focus on its use for symbolic computations. The heart of the paper is~\S\ref{SP}: we define the general class of operators to which the analysis sketched in~\S\ref{SsIS} applies (\S\ref{SsPAdm}) and analyze in detail their symbolic properties (\S\ref{SsPChar}), followed by a general analysis of $N(P)$ (\S\ref{SsPSc}). We state and prove the main microlocal result, Theorem~\ref{ThmP}, in~\S\ref{SsPP}. We prove the diffractive improvement in~\S\ref{SsPD}. Finally, \S\ref{SH} contains applications of the general theory: a sharp version of propagation estimates for $h^2\Delta_g-1$ on conic manifolds in~\S\ref{SsHL}, and high energy resolvent estimates for scattering by inverse square potentials and the Dirac--Coulomb equation in~\S\S\ref{SsHV}--\ref{SsHDC}.

\subsection*{Acknowledgments}

I am grateful to Andr\'as Vasy and Jared Wunsch for helpful conversations. I gratefully acknowledge support from a Sloan Research Fellowship and from the NSF under Grant No.\ DMS-1955614.

\section{Review of b- and scattering calculi}
\label{Sbsc}

We denote by $X$ a smooth $n$-dimensional compact manifold with nonempty, connected, and embedded boundary $\pa X$. The Lie algebra $\Vb(X)\subset\cV(X)=\CI(X;T X)$ of \emph{b-vector fields} consists of all smooth vector fields on $X$ which are tangent to $\pa X$. The Lie subalgebra $\Vsc(X)\subset\Vb(X)$ of \emph{scattering vector fields} consists of all b-vector fields which vanish, \emph{as b-vector fields}, at $\pa X$. Thus, if $x\in\CI(X)$ denotes a boundary defining function (meaning: $\pa X=x^{-1}(0)$, and $\dd x$ does not vanish on $\pa X$), then $\Vsc(X)=x\Vb(X)$. In local coordinates $(x,y)\in[0,\infty)\times\R^{n-1}$ near a point on $\pa X$, b-vector fields are of the form
\begin{equation}
\label{EqbscVF}
  a(x,y)x\pa_x + \sum_{j=1}^{n-1}b^j(x,y)\pa_{y^j},\qquad a,b^1,\ldots,b^{n-1}\in\CI,
\end{equation}
while scattering vector fields are of the form
\[
  a(x,y)x^2\pa_x + \sum_{j=1}^{n-1}b^j(x,y)x\pa_{y^j},\qquad a,b^1,\ldots,b^{n-1}\in\CI.
\]
Correspondingly, there are natural vector bundles
\begin{equation}
\label{EqbscT}
  \Tb X \to X,\qquad
  \Tsc X \to X,
\end{equation}
isomorphic to $T X^\circ$ over $X^\circ$, but with local frames (in local coordinates as above) given by $x\pa_x,\pa_{y^1},\ldots,\pa_{y^{n-1}}$ and $x^2\pa_x,x\pa_{y^1},\ldots,x\pa_{y^{n-1}}$ respectively, so that $\Vb(X)=\CI(X;\Tb X)$ and $\Vsc(X)=\CI(X;\Tsc X)$. Here, we implicitly use the bundle maps $\Tb X\to T X$ and $\Tsc X\to T X$ (which are isomorphisms over $X^\circ$ but not over $\pa X$) to identify $\CI(X;\Tb X)$ and $\CI(X;\Tsc X)$ with subspaces of $\CI(X;T X)=\cV(X)$. The dual bundles of~\eqref{EqbscT} are the b-cotangent bundle and scattering cotangent bundle, $\Tb^*X\to X$ and $\Tsc^*X\to X$, with local frames $\frac{\dd x}{x},\dd y^1,\ldots,\dd y^{n-1}$ and $\frac{\dd x}{x^2},\frac{\dd y^1}{x},\ldots,\frac{\dd y^{n-1}}{x}$, respectively. (These 1-forms are thus \emph{smooth, nonzero} sections of $\Tb^*X$, resp.\ $\Tsc^*X$, down to $\pa X$.) Writing the canonical 1-form on $T^*X^\circ$ as
\begin{equation}
\label{EqbscCan1}
  \xi_\bop\frac{\dd x}{x}+\sum_{j=1}^{n-1}(\eta_\bop)_j\,\dd y^j,\quad\text{resp.}\quad
  \xi_\scop\frac{\dd x}{x^2}+\sum_{j=1}^{n-1}(\eta_\scop)_j\frac{\dd y^j}{x},
\end{equation}
thus defines fiber-linear coordinates $(\xi_\bop,\eta_\bop)$, resp.\ $(\xi_\scop,\eta_\scop)\in\R\times\R^{n-1}$, on $\Tb^*X$, resp.\ $\Tsc^*X$. The b-density bundle is denoted $\Omegab^1 X=|\Lambda^n\,\Tb^*X|$; in local coordinates, its smooth sections are of the form $a|\frac{\dd x}{x}\dd y^1\cdots\dd y^{n-1}|$, $a\in\CI$.

The space of finite linear combinations of up to $k$-fold compositions of elements of $\cV_\bullet(X)$, $\bullet=\bop,\scop$, is denoted $\Diff_\bullet^k(X)$, and we put $\Diff_\bullet(X)=\bigoplus_{k\in\N_0}\Diff_\bullet^k(X)$. The space $\Diffb(X)$ gives rise to the notion of \emph{conormality} (relative to a fixed function space) of distributions on $X^\circ$: concretely, the space
\[
  \cA^\alpha(X)\subset x^\alpha L^\infty(X^\circ)
\]
consists of all $u$ so that $A u\in x^\alpha L^\infty(X^\circ)$ for all $A\in\Diffb(X)$. More generally, for $\delta\leq 1$, one can consider the space
\[
  \cA^\alpha_{1-\delta}(X) \subset x^\alpha L^\infty(X^\circ)
\]
of conormal distributions $u$ of type $1-\delta$, defined by the condition that for any $k\in\N_0$ and $A\in\Diffb^k(X)$, one has $A u\in x^{\alpha-k\delta}L^\infty(X^\circ)$. (Thus, $\cA^\alpha(X)=\cA_1^\alpha(X)$.) A more restrictive class than $\cA^\alpha(X)$ is the class of \emph{classical conormal distributions}, $\cA^\alpha_\cl(X)$, which is defined simply as
\[
  \cA^\alpha_\cl(X) = x^\alpha\CI(X) \subset \cA^\alpha(X).
\]
Given an element $u=x^\alpha u_0\in\cA^\alpha_\cl(X)$, the function $u_0$ is thus not merely conormal (regularity under $x\pa_x,\pa_y$), but smooth (regularity under $\pa_x,\pa_y$).

As an important example, let $E\to X$ denote a smooth real vector bundle of rank $N$, and consider the radial compactification $\bar E\to X$, i.e.\ the fiber bundle whose fiber $\bar E_x$ over $x\in X$ is equal to the radial compactification of $E_x\cong\R^N$ defined by
\[
  \ol{\R^N} := \bigl(\R^N \sqcup \bigl([0,\infty)_\rho\times\Sph^{n-1}\bigr)\bigr) / \sim,\qquad
  \R^N\setminus\{0\}\ni\rho^{-1}\omega\sim(\rho,\omega)\in[0,\infty)\times\Sph^{n-1}.
\]
Then the total space $\bar E$ is a manifold with corners which has two boundary hypersurfaces, $\bar E_{\pa X}$ (the radial compactification of $E_{\pa X}$) and $S\bar E$ (fiber infinity, locally defined by $\rho=0$). On $\bar E$, we regard only $S\bar E$ as a boundary, in the sense that we declare $\Vb(\bar E)$ to consist of all smooth vector fields on $\bar E$ which are tangent to $S\bar E$ (but not necessarily to $\bar E_{\pa X}$). For $s\in\R$, we then put
\[
  S^s(\bar E) := \cA^{-s}(\bar E).
\]
One can of course consider variants of this, e.g.\ requiring elements of $\Vb(\bar E)$ to be tangent to \emph{both} boundary hypersurfaces and defining spaces $S^{s,r}(\bar E)$ which are conormal of weight $-s,-r$ at $S\bar E$, $\bar E_{\pa X}$, respectively; or one may require classicality at one or both of the boundary hypersurfaces.

\subsection{b-pseudodifferential operators}

We denote fiber infinity of the radial compactification $\ol{\Tb^*}X$ of $\Tb^*X$ by $\Sb^*X$. Elements of $S^s(\ol{\Tb^*}X)$ will be symbols of b-pseudodifferential operators (of type $(1,0)$, in H\"ormander's $(\rho,\delta)$ terminology \cite[\S1.1]{HormanderFIO1}). Concretely, consider $a\in S^s(\ol{\Tb^*}X)$ with support contained in a local coordinate patch near a point on $\pa X$; thus, for all $i,j\in\N_0$ and $\alpha,\beta\in\N_0^{n-1}$, there exists a constant $C_{i j\alpha\beta}$ so that
\[
  \bigl|\pa_x^i\pa_y^\alpha \pa_{\xi_\bop}^j\pa_{\eta_\bop}^\beta a(x,y,\xi_\bop,\eta_\bop)\bigr| \leq C_{i j\alpha\beta}(1+|\xi_\bop|+|\eta_\bop|)^{s-(j+|\beta|)}.
\]
The (left) quantization of $a$ is then defined by
\begin{align*}
  (\Op_\bop(a)u)(x,y) &:= (2\pi)^{-n}\iiiint \exp\Bigl(i\Bigl(\frac{x-x'}{x'}\xi_\bop+(y-y')\cdot\eta_\bop\Bigr)\Bigr)\phi\Bigl(\Bigl|\log\frac{x}{x'}\Bigr|\Bigr)\phi(|y-y'|) \\
    &\quad\hspace{10em} \times a(x,y,\xi_\bop,\eta_\bop) u(x',y')\,\frac{\dd x'}{x'}\,\dd y'\,\dd\xi_\bop\,\dd\eta_\bop,
\end{align*}
where $\phi\in\CIc((-1,1))$ is identically $1$ near $0$. We define
\[
  \Psib^s(X) := \Op_\bop\bigl(S^s(\ol{\Tb^*}X)\bigr) + \Psib^{-\infty}(X).
\]
Here, writing $\pi_{L/R}\colon X^2\to X$ for the left/right projection, the space $\Psib^{-\infty}(X)$ of residual operators consists of all operators $\CIdot(X)\to\CIdot(X)$ (with $\CIdot(X)$ denoting smooth functions on $X$ vanishing to infinite order at $\pa X$) whose Schwartz kernels $\kappa\in\CmI(X^2;\pi_R^*\Omegab^1 X)$ (the dual space of $\CIdot(X^2;\pi_L^*\Omegab^1 X)$) pull back to smooth right b-densities on the b-double space\footnote{For a detailed discussion of real blow-ups such as~\eqref{EqbscX2b}, we refer the reader to \cite{MelroseDiffOnMwc}. See \cite[Appendix~A]{HintzConicPowers} for a brief summary which is sufficient for our purposes.}
\begin{equation}
\label{EqbscX2b}
  X^2_\bop := [X^2;(\pa X)^2]
\end{equation}
which vanish to infinite order at the left boundary $\lb_\bop$ (the lift of $\pa X\times X)$ and the right boundary $\rb_\bop$ (the lift of $X\times\pa X$) but are smooth down to the front face $\ff_\bop$. (See \cite[\S6]{VasyMinicourse} for more details, and also \cite{MelroseAPS,GrieserBasics}.) One often encounters weighted operators as well,
\[
  \Diffb^{k,l}(X) := x^l\Diffb^k(X),\qquad
  \Psib^{s,l}(X) := x^l\Psib^s(X).
\]
More generally still, one can consider quantizations of symbols which are conormal of order $s$ at $\Sb^*X$ and of order $l$ at $\ol{\Tb^*_{\pa X}}X$; this level of generality is occasionally useful, see e.g.\ \cite[\S5]{VasyLowEnergy} and \S\ref{SsCb}. Given an operator $A\in\Psib^{s,l}(X)$, we denote its Schwartz kernel by $K_A$.

Elements of $\Psib^{s,l}(X)$ define continuous linear operators on $\CIdot(X)$, and the composition of two b-ps.d.o.s is again a b-ps.d.o., with orders equal to the sum of the orders of the two factors. The principal symbol $\sigmab_s\colon\Psib^{s,l}(X)\to(x^l S^s/x^l S^{s-1})(\ol{\Tb^*}X)$ is a *-homomorphism, and maps commutators into Poisson brackets. In local coordinates (and omitting orders for brevity), this means that for two operators $A,B\in\Psib(X)$ with principal symbols $a,b$, we have
\begin{equation}
\label{EqbscHamb}
  \sigmab(i[A,B]) = \{a,b\} = H_a b,\qquad
  H_a = (\pa_{\xi_\bop}a)x\pa_x + (\pa_{\eta_\bop}a)\pa_y - (x\pa_x a)\pa_{\xi_\bop} - (\pa_y a)\pa_{\eta_\bop}.
\end{equation}

\subsection{Scattering pseudodifferential operators}

Turning to scattering ps.d.o.s, it is important to consider more general symbol classes than merely $S^s(\ol{\Tsc^*}X)$ or $x^{-r}S^s(\ol{\Tsc^*}X)$. Namely, for $\delta\in[0,\half)$, we shall consider the class
\[
  S_{1-\delta,\delta}^{s,r}(\ol{\Tsc^*}X)
\]
of symbols which are conormal at $\Ssc^*X$ with weight $-s$, and conormal of type $1-\delta$ with weight $-r$ at $\ol{\Tsc^*_{\pa X}}X$. This means that $S_{1-\delta,\delta}^{s,r}(\ol{\Tsc^*}X)$ consists of all smooth functions $a$ on $\Tsc^*X$ which over $X^\circ$ are symbols of type $(1,0)$ and order $s$, i.e.\ $a|_{T^*X^\circ}\in S_{1,0}^s(T^*X^\circ)=S^s(\ol{T^*}X^\circ)$, and which near $\pa X$ satisfy for all $i,j\in\N_0$ and $\alpha,\beta\in\N_0^{n-1}$ an estimate
\[
  \bigl| (x\pa_x)^i\pa_y^\alpha \pa_{\xi_\scop}^j\pa_{\eta_\scop}^\beta a(x,y,\xi_\scop,\eta_\scop) \bigr| \leq C_{i j\alpha\beta} x^{-r-(i+j+|\alpha|+|\beta|)\delta}(1+|\xi_\scop|+|\eta_\scop|)^{s-(j+|\beta|)}.
\]
In the case $\delta=0$, we omit the subscript `$1-\delta,\delta$'. We then define the (left) scattering quantization of $a$ by
\begin{align*}
  &(\Op_\scop(a)u)(x,y) \\
  &\quad := (2\pi)^{-n}\iiiint \exp\Bigl(i\Bigl[\frac{x-x'}{x^2}\xi_\scop+\frac{y-y'}{x}\cdot\eta_\scop\Bigr]\Bigr)\phi\Bigl(\Bigl|\log\frac{x}{x'}\Bigr|\Bigr)\phi(|y-y'|) \\
    &\quad\hspace{8em} \times a(x,y,\xi_\scop,\eta_\scop) u(x',y')\,\frac{\dd x'}{x'{}^2}\,\frac{\dd y'}{x'{}^{n-1}}\,\dd\xi_\scop\,\dd\eta_\scop.
\end{align*}
(In this generality, scattering ps.d.o.s were introduced by Melrose \cite{MelroseEuclideanSpectralTheory}.) If one were working with global coordinates, one could remove the cutoffs here due to the rapid decay of the partial (in the fiber variables) inverse Fourier transform of $a$ as $|\frac{1}{x}-\frac{1}{x'}|+|\frac{y}{x}-\frac{y'}{x'}|\to\infty$.\footnote{Importantly, one typically does \emph{not} want to localize more sharply to $|\frac{1}{x}-\frac{1}{x'}|+|\frac{y}{x}-\frac{y'}{x'}|\lesssim 1$ (which is a small neighborhood of the lifted diagonal in the scattering double space, see \cite[\S21]{MelroseEuclideanSpectralTheory}), as this would thus destroy the leading order commutativity of the scattering calculus at $\pa X$.} We then set
\[
  \Psi_{\scop,1-\delta,\delta}^{s,r}(X) := \Op_\scop\bigl(S_{1-\delta,\delta}^{s,r}(\ol{\Tsc^*}X)\bigr) + \Psisc^{-\infty,-\infty}(X),
\]
where $\Psisc^{-\infty,-\infty}(X)$ consists of all operators with Schwartz kernels in $\CIdot(X^2;\pi_R^*\Omega^1 X)$. We shall refer to $s$ as the (scattering) differential order, and to $r$ as the (scattering) decay order.

The principal symbol of scattering operators captures their leading order behavior for large frequencies as well as at $\pa X$:
\[
  \sigmasc_{s,r} \colon \Psi_{\scop,1-\delta,\delta}^{s,r}(X) \to (S^{s,r}_{1-\delta,\delta}/S^{s-1,r-1+2\delta}_{1-\delta,\delta})(\ol{\Tsc^*}X).
\]
This is a *-homomorphism. Thus, for $A_j\in\Psi_{\scop,1-\delta,\delta}^{s_j,r_j}(X)$, $j=1,2$, we have
\[
  [A_1,A_2] \in \Psi_{\scop,1-\delta,\delta}^{s_1+s_2-1,r_1+r_2-1+2\delta}(X);
\]
the principal symbol (capturing the commutator modulo $\Psi_{\scop,1-\delta,\delta}^{s_1+s_2-2,r_1+r_2-2+4\delta}(X)$) is given in terms of the principal symbols $a_1,a_2$ of $A_1,A_2$ by
\begin{align}
  &\sigmasc_{s_1+s_2-1,r_1+r_2-1+2\delta}(i[A_1,A_2]) = H_{a_1}a_2, \nonumber\\
\label{EqbscHamsc}
  &\qquad x^{-1}H_{a_1} = (\pa_{\xi_\scop}a_1)(x\pa_x+\eta_\scop\pa_{\eta_\scop}) + (\pa_{\eta_\scop}a_1)\pa_y - \bigl((x\pa_x+\eta_\scop\pa_{\eta_\scop})a_1\bigr)\pa_{\xi_\scop} - (\pa_y a_1)\pa_{\eta_\scop}.
\end{align}
We refer the reader to \cite[\S3]{VasyMinicourse} for more details in the special case $X=\ol{\R^n}$, in which case the scattering calculus is the same as the standard ps.d.o.\ calculus on $\R^n$ for amplitudes which are product-type symbols in the base and fiber variables.

A natural setting where one must work with $\delta>0$ arises when working with operators which have a variable scattering decay order
\[
  \sfr \in \CI(\ol{\Tsc^*}X).
\]
To wit, for $s\in\R$, we define
\[
  S^{s,\sfr}(\ol{\Tsc^*}X)
\]
to consist of all $a$ of the form $a=x^{-\sfr}a_0$, where $a_0\in\bigcap_{\delta\in(0,\frac12)}S^{s,0}_{1-\delta,\delta}(\ol{\Tsc^*}X)$. It is easy to check that $S^{s,\sfr}(\ol{\Tsc^*}X)\subset\bigcap_{\delta\in(0,\frac12)} S^{s,r_0}_{1-\delta,\delta}(\ol{\Tsc^*}X)$ for any $r_0>\sup\sfr$; in fact, differentiating variable order symbols produces only logarithmic factors in the boundary defining function $x$. Thus, we can quantize such symbols, giving rise to the space
\[
  \Psisc^{s,\sfr}(X) := \Op_\scop\bigl(S^{s,\sfr}(\ol{\Tsc^*}X)\bigr) + \Psisc^{-\infty,-\infty}(X) \subset \bigcap_{\delta\in(0,\frac12)} \Psi_{\scop,1-\delta,\delta}^{s,r_0}(X).
\]
Principal symbols of elements of $\Psisc^{s,\sfr}(X)$ are elements of $(S^{s,\sfr}/\bigcap_{\delta>0}S^{s-1,\sfr-1+2\delta})(\ol{\Tsc^*}X)$. The (variable) orders are additive under operator composition; this is a consequence of the formula for the full symbol (in local coordinates) of the composition of two ps.d.o.s.

We point out that for fixed $s\in\R$ the space $S^{s,\sfr}(\ol{\Tsc^*}X)$ (and thus $\Psisc^{s,\sfr}(X)$) only depends on the restriction of $\sfr$ to $\ol{\Tsc^*_{\pa X}}X$. Indeed, given $\sfr'\in\CI(\ol{\Tsc^*}X)$ with $\sfr'-\sfr=0$ at $\ol{\Tsc^*_{\pa X}}X$, we can write $\sfr'-\sfr = x\sfw$, $\sfw\in\CI(\ol{\Tsc^*}X)$, and therefore $x^{-\sfr'} = x^{-\sfr} \exp(-\sfw x\log x)$; by direct differentiation, one then finds that $\exp(-\sfw x\log x)\in S^{0,0}_{1-\delta,\delta}(\ol{\Tsc^*}X)$ for any $\delta>0$. In view of this, we can define $S^{s,\sfr}(\ol{\Tsc^*}X)$ and $\Psisc^{s,\sfr}(X)$, given a variable order
\[
  \sfr \in \CI(\ol{\Tsc^*_{\pa X}}X),
\]
to be equal to $S^{s,\tilde\sfr}(\ol{\Tsc^*}X)$ and $\Psisc^{s,\tilde\sfr}(X)$, respectively, where $\tilde\sfr\in\CI(\ol{\Tsc^*}X)$ is any smooth extension of $\sfr$.

\subsection{Sobolev spaces}

We next recall the corresponding scales of weighted Sobolev spaces. We have some flexibility in the choice of the underlying $L^2$-space. Thus, fix any smooth positive b-density $\mu_0\in\CI(X;\Omegab^1 X)$, and fix $a_\mu\in\R$. We then set $\mu:=x^{\alpha_\mu}\mu_0$ and
\begin{equation}
\label{EqbscL2}
  \Hb^0(X;\mu) \equiv L^2_\bop(X;\mu) \equiv \Hsc^0(X;\mu) \equiv L^2_\scop(X;\mu) := L^2(X;\mu).
\end{equation}
These spaces are independent of the choice of $\mu_0$ (but not $a_\mu$), up to equivalence of norms; the same will be true for the spaces defined in the sequel. When the density $\mu$ is fixed and clear from the context, we drop it from the notation. Let $\bullet=\bop,\scop$. For $s\geq 0$, we then let
\[
  H_\bullet^s(X) := \{ u\in H_\bullet^0(X) \colon A u\in H_\bullet^0(X) \},
\]
where $A\in\Psi_\bullet^s(X)$ denotes any fixed elliptic operator. For $s<0$, we define $H_\bullet^s(X)=(H_\bullet^{-s}(X))^*$ with respect to the $L^2_\bullet(X)$ inner product; an equivalent definition is given by $H_\bullet^s(X)=\{u_1+A u_2\colon u_1,u_2\in H_\bullet^0(X)\}$ where $A\in\Psi_\bullet^{-s}(X)$ is elliptic. Weighted spaces are defined by
\[
  \Hb^{s,l}(X)=x^l\Hb^s(X),\quad
  \Hsc^{s,r}(X)=x^r\Hsc^s(X).
\]
Finally, we define scattering Sobolev spaces with variable decay orders $\sfr\in\CI(\ol{\Tsc^*_{\pa X}}X)$ by taking $r_0<\inf\sfr$ and putting
\[
  \Hsc^{s,\sfr}(X) := \{ u\in\Hsc^{s,r_0}(X) \colon A u \in \Hsc^0(X) \},
\]
where $A\in\Psisc^{s,\sfr}(X)$ is any fixed elliptic operator.

\subsection{b-scattering operators and Sobolev spaces}

In our application, we shall encounter a compact manifold $X$ whose boundary $\pa X$ has \emph{two} connected components, say $H_1,H_2$, both of which are embedded. We can then consider the space $\cV_{\bop,\scop}(X)$ of b-scattering vector fields (which localized to a neighborhood of $H_1$, resp.\ $H_2$ lie in $\Vb$, resp.\ $\Vsc$), the corresponding b-scattering tangent bundle ${}^{\bop,\scop}T X$ and its dual ${}^{\bop,\scop}T^*X$, as well as weighted b-scattering Sobolev spaces,
\[
  H_{\bop,\scop}^{s,l,\sfr}(X),\qquad s,l\in\R,\ \sfr\in\CI(\ol{\Tsc^*_{H_2}}X).
\]
Localized to a neighborhood of $H_1$, its elements lie in $\Hb^{s,l}$, and localized to a neighborhood of $H_2$, they lie in $\Hsc^{s,\sfr}$.

Let us make this even more concrete in the setting which will arise below,
\begin{equation}
\label{Eqbsc2End}
  X = [0,\infty]_{\hat x} \times Y,\qquad H_1=\hat x^{-1}(0),\quad H_2=\hat x^{-1}(\infty),
\end{equation}
where we write $[0,\infty]$ for the closure of $[0,\infty)$ inside of $\ol\R$; here $Y$ is a compact $(n-1)$-dimensional manifold without boundary. Then $\frac{\hat x}{\hat x+1}$ and $(1+\hat x)^{-1}$ are defining functions of $H_1$ and $H_2$, respectively, hence $\cV_{\bop,\scop}(X)=(1+\hat x)^{-1}\Vb(X)$. Using local coordinates $y^1,\ldots,y^{n-1}$ on an open subset $U\subset Y$, the collection of 1-forms
\[
  (1+\hat x)\tfrac{\dd\hat x}{\hat x},\quad (1+\hat x)\dd y^1,\ \cdots,\ (1+\hat x)\dd y^{n-1}
\]
is a smooth frame of ${}^{\bop,\scop}T^*X$ over $[0,\infty]\times U$. Denoting the corresponding fiber-linear coordinates on ${}^{\bop,\scop}T^*X$ by $(\xi_{\bop,\scop},\eta_{\bop,\scop})\in\R\times\R^{n-1}$, we can then quantize a symbol\footnote{We leave the minor, largely notational, changes to accommodate symbols with variable scattering decay orders $r$ to the reader.} $a\in S^{s,l,r}(\ol{{}^{\bop,\scop}T^*}X)=\bigl(\tfrac{\hat x}{\hat x+1}\bigr)^{-l}(1+\hat x)^r S^s(\ol{{}^{\bop,\scop}T^*}X)$ by
\begin{equation}
\label{EqbscQuant}
\begin{split}
  &(\Op_{\bop,\scop}(a)u)(\hat x,y) \\
  &\qquad := (2\pi)^{-n}\iiint \exp\biggl(i\biggl( \frac{\hat x-\hat x'}{\hat x\tfrac{1}{1+\hat x}}\xi_{\bop,\scop} + \frac{y-y'}{\tfrac{1}{1+\hat x}}\cdot\eta_{\bop,\scop} \biggr)\biggr) \phi\Bigl(\Bigl|\log\frac{\hat x}{\hat x'}\Bigr|\Bigr)\phi(|y-y'|) \\
  &\qquad\hspace{10em} \times a(\hat x,y,\xi_{\bop,\scop},\eta_{\bop,\scop})u(\hat x',y')\,\frac{\dd\hat x'}{\hat x'\tfrac{1}{1+\hat x'}}\,\frac{\dd y'}{\bigl(\tfrac{1}{1+\hat x'}\bigr)^{n-1}}\,\dd\xi_{\bop,\scop}\,\dd\eta_{\bop,\scop}.
\end{split}
\end{equation}
The space $\Psi_{\bop,\scop}^{s,l,r}(X)$ of b-scattering ps.d.o.s is then the sum
\[
  \Psi_{\bop,\scop}^{s,l,r}(X) = \Op_{\bop,\scop}\bigl(S^{s,l,r}(\ol{{}^{\bop,\scop}T^*}X)\bigr) + \Psi_{\bop,\scop}^{-\infty,l,-\infty}(X).
\]
Here, $\Psi_{\bop,\scop}^{-\infty,l,-\infty}(X)=\bigl(\tfrac{\hat x}{\hat x+1}\bigr)^{-l}\Psi_{\bop,\scop}^{-\infty,0,-\infty}(X)$ is defined momentarily. First define the double space
\begin{equation}
\label{Eqbsc2EndDbl}
  X^2_{\bop,\scop} := \bigl[ \ol{[0,\infty)^2} \times Y^2; (\{0\}\times Y)^2; \Delta \cap (\{\infty\}\times Y)^2 \bigr],
\end{equation}
where $\ol{[0,\infty)^2}$ is the radial compactification (equivalently, the closure of $[0,\infty)^2$ inside of $\ol{\R^2}$), and $\Delta\subset\ol{[0,\infty)^2}\times Y^2$ is the diagonal. Then $\Psi_{\bop,\scop}^{-\infty,0,-\infty}(X)$ consists of all operators whose Schwartz kernels are smooth right b-densities on $X_{\bop,\scop}^2$ which vanish to infinite order at all boundary hypersurfaces except for the lift of $(\{0\}\times Y)^2$. See Figure~\ref{FigbscDouble}. Moreover, Schwartz kernels of elements of $\Psi_{\bop,\scop}^{s,0,r}(X)$ are conormal of order $s$ to the lifted diagonal in $X_{\bop,\scop}^2$ smoothly down to the lift of $(\{0\}\times Y)^2$, conormal with weight $-r$ down to the lift of $\Delta\cap(\{\infty\}\times Y)^2$, and vanish to infinite order at all other boundary hypersurfaces.

\begin{figure}[!ht]
\centering
\includegraphics{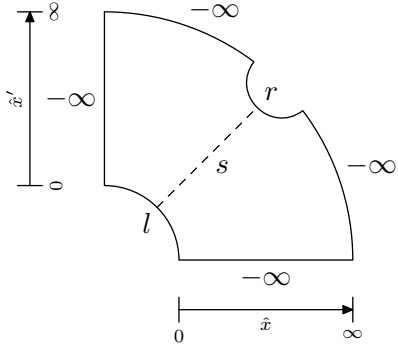}
\caption{The double space $X_{\bop,\scop}^2$ without the factor $Y^2$. The dashed line is the lifted diagonal. Indicated are the symbolic orders of Schwartz kernels of elements of $\Psi_{\bop,\scop}^{s,l,r}(X)$.}
\label{FigbscDouble}
\end{figure}

\section{Semiclassical cone calculus}
\label{SC}

We revisit and generalize the algebra $\Psich(X)$ and the associated scale of weighted Sobolev spaces from~\cite{HintzConicPowers}, give a user-friendly treatment of the symbol calculus (including Poisson brackets), and study operators and function spaces with variable (semiclassical) orders and their behavior upon restriction to the transition faces of the semiclassical cone single and double spaces (recalled later in this section). Throughout this section, we denote by $X$ a compact $n$-dimensional manifold with nonempty, connected, and embedded boundary $\pa X$. We denote by $x\in\CI(X)$ a boundary defining function.

\subsection{Vector fields, bundles, Poisson brackets}
\label{SsCV}

We recall from~\S\ref{SsIS} the \emph{semiclassical cone single space}
\[
  X_\chop := \bigl[ [0,1)_h \times X ; \{0\}\times\pa X \bigr],
\]
the boundary hypersurfaces of which we denote by $\cface$ (conic face, lift of $[0,1)\times\pa X$), $\tface$ (transition face, the front face), and $\sface$ (semiclassical face, lift of $\{0\}\times X$). See Figure~\ref{FigCVSingle}. Defining functions of these boundary hypersurfaces are $\frac{x}{x+h}$, $x+h$, and $\frac{h}{h+x}$, respectively. On $X_\chop\setminus\cface$, it is convenient to use the local defining functions $x$ of $\tface$ and $\frac{h}{x}$ of $\sface$.

\begin{figure}[!ht]
\centering
\includegraphics{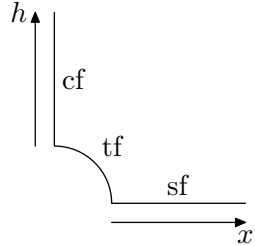}
\caption{The semiclassical cone single space $X_\chop$.}
\label{FigCVSingle}
\end{figure}

\begin{definition}[Vector fields]
\label{DefCV}
  We define the space
  \[
    \Vch(X_\chop)
  \]
  of \emph{semiclassical cone vector fields} to consist of all b-vector fields $V\in\Vb(X_\chop)$ which are horizontal, i.e.\ tangent to the fibers of $X_\chop\to[0,1)_h$, and whose restriction to $\sface$ vanishes.
\end{definition}

\begin{lemma}[Spanning set]
\label{LemmaCVSpan}
  Identifying a vector field $V\in\Vb(X)$ with its horizontal lift to $X_\chop$ along $X_\chop\to X$, the space $\Vch(X_\chop)$ is spanned over $\CI(X_\chop)$ by $\frac{h}{h+x}\Vb(X)$. Moreover, given $V,W\in\Vch(X_\chop)$, we have $[V,W]\in\frac{h}{h+x}\Vch(X_\chop)$.
\end{lemma}

This allows us to define the graded ring
\[
  \Diffch(X_\chop)=\bigoplus_{k\in\N_0}\Diffch^k(X_\chop)
\]
of differential operators in the usual manner.

\begin{proof}
  Directly from the definition, we have $\frac{h}{h+x}\Vb(X)\subset\Vch(X_\chop)$. Conversely, suppose $V\in\Vch(X_\chop)$. Let us work in local coordinates $(x,y)\in[0,\infty)\times\R^{n-1}$ near a point in $\pa X$. Near $\cface$, we use the local coordinates $(h,\hat x,y)$ with $\hat x:=\frac{x}{h}$. From the definition, we have
  \begin{equation}
  \label{EqCVSpancf}
    V = a(h,\hat x,y)\hat x\pa_{\hat x}+\sum_{j=1}^{n-1}b^j(h,\hat x,y)\pa_{y^j}
  \end{equation}
  with $a,b^j\in\CI$. Since $\hat x\pa_{\hat x}=x\pa_x$, this expresses $V$ in the desired form.
  
  Near $\sface$ on the other hand, we use $(\hat h,x,y)$ with $\hat h:=\frac{h}{x}$. Since $V\in\Vb(X_\chop)$, we can write
  \[
    V = a(\hat h,x,y)(x\pa_x-\hat h\pa_{\hat h}) + \tilde a(\hat h,x,y)\hat h\pa_{\hat h} + \sum_{j=1}^{n-1}b^j(\hat h,x,y)\pa_{y^j}.
  \]
  The horizontal nature of $V$ means $0=V h=V(x\hat h)=\tilde a x\hat h$, which implies $\tilde a\equiv 0$ by continuity from $(X_\chop)^\circ=\{x>0,\hat h>0\}$. The vanishing of $V$ at $\hat h=0$ as a b-vector field implies, in addition, that $a=\hat h a'$ and $b^j=\hat h b'_j$ with $a',b'_j\in\CI$. Since the horizontal lifts of $x\pa_x,\pa_{y^j}\in\Vb(X)$ to $X_\chop$ are equal to $x\pa_x-\hat h\pa_{\hat h},\pa_{y^j}$, the claim follows.

  Regarding the Lie algebra structure, we compute, for $V,W\in\Vb(X)$,
  \[
    \bigl[\tfrac{h}{h+x}V,\tfrac{h}{h+x}W\bigr] = \tfrac{h}{h+x}\Bigl(\tfrac{h}{h+x}[V,W] + V\bigl(\tfrac{h}{h+x}\bigr)W - W\bigl(\tfrac{h}{h+x}\bigr)V \Bigr).
  \]
  Since $V,W\in\Vb(X_\chop)$, we have $V\bigl(\tfrac{h}{h+x}\bigr),W\bigl(\tfrac{h}{h+x}\bigr)\in\tfrac{h}{h+x}\CI(X_\chop)$. The proof is complete.
\end{proof}

There exists a vector bundle
\[
  \Tch X_\chop \to X_\chop
\]
together with a smooth bundle map $\Tch X_\chop\to\Tb X_\chop$ so that the space $\Vch(X_\chop)$ is equal to the space of smooth sections of $\Tch X_\chop$. In local coordinates on $X$, a local frame of $\Tch X_\chop$ is given by (the horizontal lifts to $X_\chop$ of)
\[
  \tfrac{h}{h+x}x\pa_x,\quad \tfrac{h}{h+x}\pa_{y^1},\ \ldots,\ \tfrac{h}{h+x}\pa_{y^{n-1}}.
\]
We call $\Tch X_\chop$ the \emph{$\chop$-tangent bundle} and its dual $\Tch^*X_\chop$ the \emph{$\chop$-cotangent bundle}, with local frame
\[
  \tfrac{x+h}{h}\tfrac{\dd x}{x},\quad\tfrac{x+h}{h}\,\dd y^1,\ \ldots,\ \tfrac{x+h}{h}\,\dd y^{n-1}.
\]

A choice of local coordinates $(x,y)\in[0,\infty)\times\R^{n-1}$ on an open set $U\subset X$ induces a trivialization of $\Tch^* X_\chop$ over the preimage of $[0,1)\times U$ under $X_\chop\to X$, with fiber-linear coordinates $(\xi_\chop,\eta_\chop)\in\R\times\R^{n-1}$ defined by the requirement that the canonical 1-form on $T^*X^\circ$ be equal to
\begin{equation}
\label{EqCVCoord}
  \xi_\chop\frac{x+h}{h}\frac{\dd x}{x}+\sum_{j=1}^{n-1}(\eta_\chop)_j\frac{x+h}{h}\dd y^j.
\end{equation}
In $X_\chop\setminus\cface$, where a smooth frame of $\Tch X_\chop$ is given by $\frac{h}{x}x\pa_x,\frac{h}{x}\pa_{y^1},\ldots,\frac{h}{x}\pa_{y^{n-1}}$, it is computationally simpler to use the fiber-linear coordinates $(\xi,\eta)$ in which the canonical 1-form takes the form
\begin{equation}
\label{EqCVCoordTilde}
  \xi\frac{x}{h}\frac{\dd x}{x}+\sum_{j=1}^{n-1}\eta_j\frac{x}{h}\dd y^j.
\end{equation}
We compute the form of the Hamilton vector field $H_a$ of a smooth function $a\in\CI(\Tch^*X_\chop)$ in these fiber coordinates, and using $(\hat h,x,y)$ with $\hat h=\frac{h}{x}$ as coordinates on the base. In terms of the coordinates on $\Tb^*X$ used in~\eqref{EqbscCan1}, we have $(\xi,\eta)=\frac{h}{x}(\xi_\bop,\eta_\bop)$ and thus, by changing coordinates in the expression~\eqref{EqbscHamb},
\begin{equation}
\label{EqCVHam}
\begin{split}
  H_a &= \hat h\Bigl( (\pa_\xi a)(x\pa_x-\hat h\pa_{\hat h}-\eta\pa_\eta) + (\pa_\eta a)\pa_y \\
    &\quad\qquad - \bigl((x\pa_x-\hat h\pa_{\hat h}-\eta\pa_\eta)a\bigr)\pa_\xi a - (\pa_y a)\pa_\eta \Bigr).
\end{split}
\end{equation}

\subsection{Symbols, pseudodifferential operators, Sobolev spaces}
\label{SsCP}

The simplest symbol class for $\chop$-operators is $S^s(\Tch^*X_\chop)=\cA^{-s}(\ol{\Tch^*}X_\chop)$, where we only regard fiber infinity $\Sch^*X_\chop$ as a boundary, i.e.\ we require symbols to be smooth down to $\ol{\Tch^*_\bullet}X_\chop$ for $\bullet=\cface,\tface,\sface$. In practice, we need more general symbols: for $\delta\in[0,\half)$ and for $s,l,\alpha,b\in\R$, we define
\[
  S^{s,l,\alpha,b}_{1-\delta,\delta}(\ol{\Tch^*}X_\chop) = \bigl(\tfrac{x}{x+h}\bigr)^{-l}(x+h)^{-\alpha}\bigl(\tfrac{h}{h+x}\bigr)^{-b}S^{s,0,0,0}_{1-\delta,\delta}(\ol{\Tch^*}X_\chop)
\]
to consist of all symbols which are conormal with weight $-s$ at $\Sch^*X_\chop$, conormal with weight $-l$ at $\ol{\Tch^*_\cface}X_\chop$ and with weight $-\alpha$ at $\ol{\Tch^*_\tface}X_\chop$, and conormal of type $1-\delta$ at $\ol{\Tch^*_\sface}X_\chop$ with weight $-b$. In the coordinates~\eqref{EqCVCoordTilde}, the membership $a\in S_{1-\delta,\delta}^{s,0,0,0}(\ol{\Tch^*}X_\chop)$ is equivalent to $a=a(\hat h,x,y,\xi,\eta)$ (with $\hat h=\frac{h}{x}$) satisfying estimates
\begin{align*}
  &\bigl| (x\pa_x)^i\pa_y^\alpha(\hat h\pa_{\hat h})^j \pa_\xi^k\pa_\eta^\beta a(\hat h,x,y,\xi,\eta) \bigr| \\
  &\qquad \leq C_{i j k\alpha\beta}(1+|\xi|+|\eta|)^{s-(k+|\beta|)}\hat h^{-(i+j+k+|\alpha|+|\beta|)\delta}
\end{align*}
for all $i,j,k\in\N_0$ and $\alpha,\beta\in\N_0^{n-1}$; in coordinates $(h,\hat x,y,\xi_\bop,\eta_\bop)$ on the $\chop$-cotangent bundle over $X_\chop\setminus\sface$, with $\hat x=\frac{x}{h}$ and with the canonical 1-form given by~\eqref{EqbscCan1}, $a$ must satisfy
\[
  \bigl| (\hat x\pa_{\hat x})^i\pa_y^\alpha(h\pa_h)^j\pa_{\xi_\bop}^k\pa_{\eta_\bop}^\beta a(h,\hat x,y,\xi_\bop,\eta_\bop) \bigr| \leq C_{i j k\alpha\beta} (1+|\xi_\bop|+|\eta_\bop|)^{s-(k+|\beta|)}.
\]
See Figure~\ref{FigCPSymbol}. As usual, we omit the subscript `$1-\delta,\delta$' when $\delta=0$.

It is occasionally useful to restrict attention to symbols which are \emph{classical} conormal down to $\tface$, which amounts to replacing $x\pa_x$, $h\pa_h$ in the above symbol estimates (which are for symbols of order $0$ at $\tface$) by $\pa_x$, $\pa_h$. We denote the corresponding symbol classes with a subscript `$\cl$' as in $S^{s,l,\alpha,b}_\cl(\ol{\Tch^*}X_\chop)$.

As in~\S\ref{Sbsc}, the main use of $\delta>0$ is to accommodate symbols with variable orders. Here, we only discuss the case of variable semiclassical orders. Thus, consider $\sfb\in\CI(\ol{\Tch^*_\sface}X_\chop)$, an arbitrary extension of which to an element of $\CI(\ol{\Tch^*}X_\chop)$ we denote by the same letter; we then put
\[
  S^{s,l,\alpha,\sfb}(\ol{\Tch^*}X_\chop) := \biggr\{ \bigl(\tfrac{h}{h+x}\bigr)^\sfb a_0 \colon a_0 \in \bigcap_{\delta\in(0,\frac12)}S^{s,l,\alpha,0}_{1-\delta,\delta}(\ol{\Tch^*}X_\chop) \biggr\},
\]
which is a subset of $\bigcap_{\delta\in(0,\frac12)}S^{s,l,\alpha,b_0}_{1-\delta,\delta}(\ol{\Tch^*}X_\chop)$ for any $b_0>\sup\sfb$.

\begin{figure}[!ht]
\centering
\includegraphics{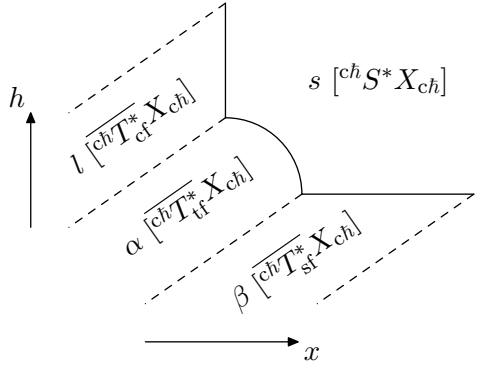}
\caption{Illustration of $\ol{\Tch^*}X_\chop$ (showing only part of the compactified fibers) and the symbol class $S^{s,l,\alpha,\beta}(\ol{\Tch^*}X_\chop)$, indicating the orders at the various boundary hypersurfaces of $\ol{\Tch^*}X_\chop$.}
\label{FigCPSymbol}
\end{figure}

We now proceed to quantize symbols $a=a(h,x,y,\xi_\chop,\eta_\chop)$, thereby giving meaning to the formal expression ``$\Op_{\cop,h}(a)=a(h,x,y,\tfrac{h}{h+x}x D_x,\tfrac{h}{h+x}D_y)$''. Thus, fixing $\phi\in\CIc((-1,1))$, identically $1$ near $0$, we define, in local coordinates $(x,y)$ on $X$,
\begin{equation}
\label{EqCPQuant}
\begin{split}
  &(\Op_{\cop,h}(a)u)(h,x,y) \\
  &\qquad := (2\pi)^{-n}\iiiint \exp\biggl(i\biggl[ \frac{x-x'}{x\tfrac{h}{h+x}}\xi_\chop + \frac{y-y'}{\tfrac{h}{h+x}}\cdot\eta_\chop \biggr]\biggr) \phi\Bigl(\Bigl|\log\frac{x}{x'}\Bigr|\Bigr)\phi(|y-y'|) \\
  &\qquad\hspace{9em} \times a(h,x,y,\xi_\chop,\eta_\chop)u(h,x',y')\,\frac{\dd x'}{x'\tfrac{h}{h+x'}}\,\frac{\dd y'}{\bigl(\tfrac{h}{h+x'}\bigr)^{n-1}}\,\dd\xi_\chop\,\dd\eta_\chop
\end{split}
\end{equation}
for $a$ and $u$ supported in the coordinate chart; for general $a,u$, one defines $\Op_{\cop,h}(a)u$ using a partition of unity.

We interpret this in terms of the $\chop$-double space
\[
  X^2_\chop := \bigl[ [0,1)_h \times X^2_\bop; \{0\}\times\ff_\bop; \{0\}\times\Delta_\bop \bigr],
\]
where we denote by $\Delta_\bop\subset X^2_\bop$ the lift of the diagonal in $X^2$ to $X^2_\bop$; see equation~\eqref{EqbscX2b} and the subsequent paragraph for the definition of $X^2_\bop$ and its boundary hypersurfaces $\lb_\bop,\ff_\bop,\rb_\bop$. First, recall from\footnote{We add subscripts `$2$' here in order to avoid confusion during the frequent changes between $X_\chop$ and $X_\chop^2$ later on.} \cite[Definition~3.1]{HintzConicPowers} that $\lb_2,\rb_2,\ff_2,\tface_2,\sface_2$, and $\dface_2$ are the lifts of $[0,1)\times\lb_\bop$, $[0,1)\times\rb_\bop$, $[0,1)\times\ff_\bop$, $\{0\}\times\ff_\bop$, $\{0\}\times X^2_\bop$, and $\{0\}\times\Delta_\bop$, respectively; moreover, $\Delta_\chop$ denotes the lift of $[0,1)\times\Delta_\bop$. See Figure~\ref{FigCPDouble}.

\begin{figure}[!ht]
\centering
\includegraphics{FigCPDouble}
\caption{The $\chop$-double space $X^2_\chop$.}
\label{FigCPDouble}
\end{figure}

Then the Schwartz kernel of $\Op_{\cop,h}(a)$ is a conormal distribution of order $s-\tfrac14$ at $\Delta_\chop$, conormal down to $\ff_2,\tface_2,\dface_2$ with weights $-l,-\alpha,-b$, and vanishes identically in a neighborhood of $\lb_2,\rb_2,\sface_2$.

The composition of two $\chop$-quantizations is almost a $\chop$-quantization itself; one merely has to allow for additional residual terms: define the space $\Psich^{-\infty}(X)$ of residual operators to consist of all operators whose Schwartz kernels are conormal sections of the right b-density bundle on $X^2_\chop$, with weight $0$ at $\ff_2$ and $\tface_2$, and with infinite order vanishing at $\lb_2,\rb_2,\dface_2,\sface_2$. We then put
\[
  \Psich^s(X) := \Op_\chop\bigl(S^s(\ol{\Tch^*}X_\chop)\bigr) + \Psich^{-\infty}(X),
\]
where $\Op_\chop=(\Op_{\cop,h})_{h\in(0,1)}$; this gives the same space as \cite[Definition~3.2]{HintzConicPowers}. More generally, we define the quantization of symbols $a\in S^{s,l,\alpha,b}_{1-\delta,\delta}(\ol{\Tch^*}X_\chop)$ by the same formula~\eqref{EqCPQuant}; the space of residual operators is now
\[
  \Psich^{-\infty,l,\alpha,-\infty}(X):=(\tfrac{x}{x+h})^{-l}(x+h)^{-\alpha}\Psich^{-\infty}(X).
\]
Thus, we can now define the spaces
\begin{align*}
  \Psi_{\chop,1-\delta,\delta}^{s,l,\alpha,b}(X) &:= \Op_\chop\bigl(S^{s,l,\alpha,b}_{1-\delta,\delta}(\ol{\Tch^*}X_\chop)\bigr) + \Psich^{-\infty,l,\alpha,-\infty}(X), \\
  \Psi_{\chop}^{s,l,\alpha,\sfb}(X) &:= \Op_\chop\bigl(S^{s,l,\alpha,\sfb}(\ol{\Tch^*}X_\chop)\bigr) + \Psich^{-\infty,l,\alpha,-\infty}(X),
\end{align*}
where in the second line $\sfb\in\CI(\ol{\Tch^*_\sface}X_\chop)$ is a variable order function. Their Schwartz kernels can be characterized as being conormal distributions (of order $s-\tfrac14$ and type $(1,0)$) at $\Delta_\chop$ which are conormal at $\ff_2$ (with weight $-l$), $\tface_2$ (with weight $-\alpha$), and conormal of type $1-\delta$ at $\dface_2$ (with weight $-b$), and which vanish to infinite order at $\lb_2,\rb_2,\sface_2$. One can also consider subalgebras which are classical at $\tface$, i.e.\ the symbols are required to be classical conormal at $\tface$, and the residual operators are required to have classical conormal Schwartz kernels at $\tface_2$; we denote these algebras by a subscript `$\cl$', such as
\[
  \Psi_{\chop,\cl}^{s,l,\alpha,\sfb}(X).
\]

All such ps.d.o.s define $h$-dependent families of bounded\footnote{though not uniformly in $h$ unless $\sfb\geq 0$} linear maps on $\CIdot(X)$; compositions of two such ps.d.o.s give a ps.d.o.\ in the same class, with orders given by the sum of the orders of the two factors. The principal symbol map is
\[
  \sigmach_{s,l,\alpha,b} \colon \Psi_{\chop,1-\delta,\delta}^{s,l,\alpha,b}(X) \to (S^{s,l,\alpha,b}/S^{s-1,l,\alpha,b-1+2\delta})(\ol{\Tch^*}X_\chop),
\]
similarly for the variable order spaces (with $\delta>0$ then arbitrary), and it is a *-homo\-morph\-ism. These facts follow from a minor variation of \cite[Proposition~3.9]{HintzConicPowers} (using weights instead of index sets), with the statements about principal symbols following by continuity from the corresponding statements for standard semiclassical operators (of type $(1-\delta,\delta)$) in $x>0$ and b-ps.d.o.s in $h>0$; we leave the details to the reader. We moreover have, for $A_j\in\Psi_{\chop,1-\delta,\delta}^{s_j,l_j,\alpha_j,b_j}(X)$, $j=1,2$, with principal symbols $a_j$,
\[
  \Op_{\cop,h}(i[A_1,A_2]) - \Op_{\cop,h}(H_{a_1}a_2) \in \Psi_{\chop,1-\delta,\delta}^{s-2,l,\alpha,b-2+4\delta}(X),
\]
analogously for variable order operators. One can evaluate $H_{a_1}a_2$ using the formula~\eqref{EqCVHam}.

Since the principal symbol captures operators to leading order at $\Sch^*X_\chop\cup\ol{\Tch^*_\sface}X_\chop$, the latter set is also the locus of the elliptic and wave front sets of an operator. Thus, for $A\in\Psich^{s,l,\alpha,b}(X)$, we define
\[
  \Ellch^{s,l,\alpha,b}(A),\ \WFch^{\prime l,\alpha}(A) \subset \Sch^*X_\chop \cup \ol{\Tch^*_\sface}X_\chop
\]
as follows: $\Ellch^{s,l,\alpha,b}(A)$ is the set of all $\zeta$ so that $\sigmach_{s,l,\alpha,b}(A)$ is elliptic in a neighborhood of $\zeta$, and $\WFch^{\prime l,\alpha}(A)$ is the complement of the set of $\zeta$ so that the full symbol of $A$ lies in $S^{-\infty,l,\alpha,-\infty}(\ol{\Tch^*}X_\chop)$ when localized to a sufficiently small neighborhood of $\zeta$. In particular, we have $\WFch^{\prime l,\alpha}(A)=\emptyset$ if and only if $A\in\Psich^{-\infty,l,\alpha,-\infty}(X)$. We omit the orders $s,l,\alpha,b$ and $l,\alpha$ when they are clear from the context. The definitions for type $(1-\delta,\delta)$ and variable order operators are analogous. See Figure~\ref{FigCPEll}.

\begin{figure}[!ht]
\centering
\includegraphics{FigCPEll}
\caption{The shaded boundary hypersurfaces are the locus of the elliptic set as well as of operator and distributional wave front sets. Cf.\ Figure~\ref{FigCPSymbol}.}
\label{FigCPEll}
\end{figure}

Finally, we define the corresponding weighted Sobolev spaces. As in~\eqref{EqbscL2}, we first fix a weighted b-density $\mu=x^{\alpha_\mu}\mu_0$, where $0<\mu_0\in\CI(X;\Omegab^1 X)$ and $\alpha_\mu\in\R$, and define
\[
  \Hch^0(X;\mu) := L^2(X;\mu),\quad
  \Hch^{0,l,\alpha,b}(X;\mu) := \bigl(\tfrac{x}{x+h}\bigr)^l(x+h)^a\bigl(\tfrac{h}{h+x}\bigr)^b\Hch^0(X;\mu).
\]
These spaces depend on $\alpha_\mu$, but are independent of $\mu_0$ (up to equivalence of norms). When the choice of $\mu$ is clear from the context, we will omit it from the notation. For $s\geq 0$, we then define $\Hch^{s,l,\alpha,b}(X)$ to consist of all $u\in\Hch^{0,l,\alpha,b}(X)$ so that $A u\in\Hch^0(X)$ for any (thus all) elliptic $A\in\Psich^{s,l,\alpha,b}(X)$. We note for $s\in\N_0$ the equivalent characterization
\[
  \Hch^{s,l,\alpha,b}(X) = \bigl\{ u\in\Hch^{0,l,\alpha,b}(X) \colon V_1\cdots V_j u\in\Hch^{0,l,\alpha,b}(X)\ \forall\,V_i\in\Vch(X_\chop),\ 0\leq i\leq j\leq s\bigr\}.
\]
For $s<0$, the space $\Hch^{s,l,\alpha,b}(X)$ can be defined either by duality as $(\Hch^{-s,-l,-\alpha,-b}(X))^*$, or as the space of all $u_1+A u_2$ where $u_1,u_2\in\Hch^{0,l,\alpha,b}(X)$ and $A\in\Psich^{-s}(X)$. Lastly, for a variable order $\sfb\in\CI(\ol{\Tch^*_\sface}X_\chop)$, we pick $b_0<\inf\sfb$ and put
\[
  \Hch^{s,l,\alpha,\sfb}(X) := \{ u\in\Hch^{s,l,\alpha,b_0}(X) \colon A u\in\Hch^0(X) \},
\]
where $A\in\Psich^{s,l,\alpha,\sfb}(X)$ is any elliptic operator; the space $\Hch^{s,l,\alpha,\sfb}(X)$ is independent of the choices of $b_0$ and $A$, up to equivalence of norms.

We can define Sobolev wave front sets in the usual manner. Let $l,\alpha\in\R$, and suppose that we are given a distribution $u\in\Hch^{-\infty,l,\alpha,-\infty}(X)$, meaning $u\in\Hch^{-N,l,\alpha,-N}(X)$ for some $N\in\R$. Let $s,b\in\R$. Then
\[
  \WFch^{s,l,\alpha,b}(u) \subset \Sch^*X_\chop \cup \ol{\Tch^*_\sface}X_\chop
\]
is the complement of all $\alpha$ so that there exists an operator $A\in\Psich^{s,l,\alpha,b}(X)$, elliptic at $\alpha$, so that $A u\in\Hch^0(X)$. (The a priori assumption on $u$ is familiar from the definition of the b-wave front set, see e.g.\ \cite[Definition~6.2]{VasyMinicourse}, and ensures that one then also has $B u\in\Hch^0(X)$ for any $B\in\Psich^{s,l,\alpha,b}(X)$ with $\WFch'(B)\subset\Ellch'(A)$.)

\begin{rmk}[Operators on vector bundles]
\label{RmkCPBundles}
  If $E,F\to X$ are smooth vector bundles, one can consider semiclassical cone ps.d.o.s acting between sections of $E,F$, giving rise to classes $\Psich^s(X;E,F)$ and function spaces $\Hch^s(X;E)$ etc. More generally, one can allow $E,F$ to be vector bundles $E,F\to X_\chop$ over the semiclassical single space, with Schwartz kernels of elements of $\Psich^s(X;E,F)$ defined by taking the tensor product of $\Psich^s(X)$ over $\CI(X^2_\chop)$ with $\CI(X^2_\chop;\pi_L^*E\boxtimes\pi_R^*F^*)$ where $\pi_L,\pi_R\colon X^2_\chop\to X_\chop$ are the stretched left and right projections. Using such ps.d.o.s, one can define Sobolev spaces $\Hch^s(X;E)$ etc.\ in this generality.
\end{rmk}

\begin{rmk}[Relationship with edge Sobolev spaces]
\label{RmkCPEdge}
  For the propagation through cone points in the spacetime setting, many authors \cite{MelroseWunschConic,MelroseVasyWunschEdge} have utilized Mazzeo's edge algebra \cite{MazzeoEdge}. A typical example is the operator $-D_t^2+\Delta_g$, where $g=g(x,y,\dd y)$ is a conic metric on a manifold $X$ with boundary (see~\eqref{EqPProdMet}); upon multiplication by $x^2$, this is a second order differential operator, the principal part of which is a Lorentzian signature quadratic form in the collection $(x D_t,x D_x,D_y)$ of edge vector fields. The membership $u\in H_\eop^1(\R_t\times X,|\dd t\,\dd g|)$---meaning that $u$, $x D_t u$, $x D_x u$, $D_y u\in L^2$---can then be characterized by taking the Fourier transform in $t$ as
  \[
    \hat u(\sigma),\ x|\sigma|\hat u(\sigma),\ x D_x\hat u(\sigma),\ D_y\hat u(\sigma) \in L^2(\R_\sigma;L^2(X;|\dd g|)).
  \]
  Introducing $h=\la\sigma\ra^{-1}$, this is equivalent to the $L^2(\R_\sigma;L^2(X))$ membership of $\tfrac{h+x}{h}\hat u$, $x D_x\hat u$, $D_y\hat u$. Upon multiplication by $\frac{h}{h+x}$, we thus find
  \[
    u\in H_\eop^1(\R_t\times X,|\dd t\,\dd g|) \quad\Longleftrightarrow\quad \hat u\in L^2\bigl(\R_\sigma;H_{\cop,\la\sigma\ra^{-1}}^{1,0,0,1}(X;|\dd g|)\bigr),
  \]
  and the respective norms of $u$ and $\hat u$ are equivalent. (One can show that similar spectral characterizations of edge Sobolev spaces remain valid also for spaces with weights and with variable differential orders; the details will be given elsewhere.)
\end{rmk}

\subsection{Restriction to \texorpdfstring{$\tface$}{the transition face}}
\label{SsCPN}

Symbolic arguments for the analysis of semiclassical cone PDEs $P u=f$ can at best control $u$ microlocally at $\Sch^*X_\chop\cup\ol{\Tch^*_\sface}X_\chop$, i.e.\ modulo errors which are trivial at infinite frequencies and at $\sface$. Crucially however, such errors may well be nontrivial at $\tface$, and thus nontrivial (meaning in particular: not small) as $h\to 0$. To obtain control at $\tface$, one needs to invert the normal operator $N(P)$, defined in~\cite[\S3.1.2]{HintzConicPowers} (denoted $N_\tface(P)$ there) and recalled below. The following result, already implicit in the definition of the normal operator in~\cite[\S3.1.2]{HintzConicPowers}, lays the groundwork for the analysis of $N(P)$.

\begin{lemma}[Restriction to $\tface$: vector fields]
\label{LemmaCPN}
  The restriction map $\Vb(X_\chop)\ni V\mapsto V|_\tface\in\Vb(\tface)$ restricts to a surjective map
  \begin{equation}
  \label{EqCPN}
    N \colon \Vch(X_\chop) \to \cV_{\bop,\scop}(\tface)
  \end{equation}
  onto the space $\cV_{\bop,\scop}(\tface)=\tfrac{h}{h+x}\Vb(\tface)$ of vector fields which are b-vector fields near $\tface\cap\cface$ and scattering vector fields near $\tface\cap\sface$. The map~\eqref{EqCPN} induces bundle isomorphisms
  \begin{equation}
  \label{EqCPNBundles}
    \Tch_\tface X_\chop \cong {}^{\bop,\scop}T\tface,\qquad
    \Tch_\tface^* X_\chop \cong {}^{\bop,\scop}T^*\tface.
  \end{equation}
\end{lemma}
\begin{proof}
  On $\tface\setminus\sface$, the conclusion follows from the expression~\eqref{EqCVSpancf}. On $\tface\setminus\cface$ on the other hand, and using coordinates $(\hat h,x,y)$ with $\hat h=\frac{h}{x}$, we can write $V\in\Vch(X_\chop)$ as
  \[
    V = \hat h a(\hat h,x,y)(x\pa_x-\hat h\pa_{\hat h}) + \sum_{j=1}^{n-1}\hat h b^j(\hat h,x,y)\pa_{y^j}
  \]
  with smooth coefficients $a,b^1,\ldots,b^{n-1}$. Restriction to $\tface$, which in these coordinates is given by $x=0$, produces $V|_\tface = -a(\hat h,0,y)\hat h^2\pa_{\hat h} + \sum_{j=1}^{n-1} b^j(\hat h,0,y)\hat h\pa_{y^j}$, which is a scattering vector field, as claimed. The surjectivity of~\eqref{EqCPN} is clear from these explicit calculations.
  
   The surjectivity of~\eqref{EqCPNBundles} follows from this, and the injectivity is also clear from these explicit calculations. More conceptually, injectivity follows from the fact that $\ker N=(x+h)\Vch(X_\chop)$; that is, $N(V)$ captures $V$ to leading order at $\tface$ as a $\chop$-vector field.
\end{proof}

The map~\eqref{EqCPN} induces a surjective map
\begin{equation}
\label{EqCPNDiff}
  N \colon \bigl(\tfrac{x}{x+h}\bigr)^{-l}\bigl(\tfrac{h}{h+x}\bigr)^{-b}\Diffch^k(X_\chop) \to \bigl(\tfrac{\hat x}{\hat x+1}\bigr)^{-l}(\hat x+1)^b\,\Diff_{\bop,\scop}^k(\tface),\qquad \hat x:=\tfrac{x}{h}
\end{equation}
into weighted b-scattering differential operators on $\tface$. More generally:

\begin{lemma}[Restriction to the transition face: ps.d.o.s]
\label{LemmaCPNpsdo}
  Let $s,l,b\in\R$. Restriction to $\tface_2\subset X^2_\chop$ induces a surjective map $N\colon\Psi_{\chop,\cl}^{s,l,0,b}(X) \to \Psi_{\bop,\scop}^{s,l,b}(\tface)$.\footnote{Recall that the subscript `$\cl$' refers to classicality at, i.e.\ here smoothness of the Schwartz kernels, down to $\tface_2$.} More generally, if $\sfb\in\CI(\ol{\Tch^*_\sface}X_\chop)$, then $\sfb':=\sfb|_{\tface\cap\sface}\in\CI(\ol{{}^{\bop,\scop}T_{\tface\cap\sface}^*}\tface)$, and restriction to $\tface_2\subset X^2_\chop$ induces a surjective map
  \begin{equation}
  \label{EqCPNpsdo}
    N\colon\Psi_{\chop,\cl}^{s,l,0,\sfb}(X) \to \Psi_{\bop,\scop}^{s,l,\sfb'}(\tface).
  \end{equation}
\end{lemma}
\begin{proof}
  This can be proved entirely on the level of Schwartz kernels, since memberships in $\Psich$ or $\Psi_{\bop,\scop}$ are characterized as conormal distributions with conormal regularity at various boundary hypersurfaces. The point then is that $\tface_2$ is naturally diffeomorphic to the double space $\tface^2_{\bop,\scop}$ in the notation of~\eqref{Eqbsc2EndDbl}, where we note that $\tface\cong[0,\infty]_{\hat x}\times\pa X$ is indeed of the form~\eqref{Eqbsc2End}. This is the route taken in \cite[\S3.1.2]{HintzConicPowers}.

  Alternatively, we can proceed explicitly for the symbolically nontrivial part using the quantization map~\eqref{EqCPQuant}, and use the Schwartz kernel perspective only to deduce the surjectivity of the restriction map for residual operators, $\Psi_{\chop,\cl}^{-\infty,l,0,-\infty}(X)\to\Psi_{\bop,\scop}^{-\infty,l,-\infty}(\tface)$. Indeed, on the level of symbols, note that with $\hat x=\frac{x}{h}$, we have
  \[
    S^{s,l,0,b}(\ol{\Tch^*}X_\chop) = \bigl(\tfrac{x}{x+h}\bigr)^{-l}\bigl(\tfrac{h}{h+x}\bigr)^{-b}S^{s,0,0,0}(\ol{\Tch^*}X_\chop) = \bigl(\tfrac{\hat x}{\hat x+1}\bigr)^{-l}(\hat x+1)^b S^{s,0,0,0}(\ol{\Tch^*}X_\chop),
  \]
  hence Lemma~\ref{LemmaCPN} implies that restriction to $\ol{\Tch^*_\tface}X_\chop$ induces a surjective map
  \[
    S_\cl^{s,l,0,b}(\ol{\Tch^*}X_\chop)\to S^{s,l,b}(\ol{{}^{\bop,\scop}T^*}\tface).
  \]
  But changing variables in the $\chop$-quantization~\eqref{EqCPQuant} to $\hat x=\frac{x}{h}$, $\hat x'=\frac{x'}{h}$ produces precisely the b-scattering quantization~\eqref{EqbscQuant}. This proves the lemma for constant orders; the proof in the variable order case is the same.
\end{proof}

As a consequence, we can relate semiclassical cone Sobolev spaces to b-scattering Sobolev spaces in the following manner:
\begin{cor}[Restriction to $\tface$: Sobolev spaces]
\label{CorCPNSob}
  Fix a collar neighborhood $\cU=[0,x_0)_x\times\pa X$ of $\pa X$. Suppose $\chi\in\cA^0([0,1)_h\times X)$ has compact support in $[0,1)\times\cU$. Define the map $\pi\colon [0,1)\times[0,\infty)_{\hat x}\times\pa X\to[0,1)\times[0,\infty)_x\times\pa X$ by $\pi(h,\hat x,y) = (h,h\hat x,y)$. Let $0<\mu_0\in\CI(X;\Omegab^1 X)$ and $0<\hat\mu_0\in\CI(\tface;\Omegab^1 X)$, let $\alpha_\mu\in\R$, and put
  \[
    \mu := x^{\alpha_\mu}\mu_0,\quad
    \hat\mu := \hat x^{\alpha_\mu}\hat\mu_0.
  \]
  \begin{enumerate}
  \item{\rm (Constant orders.)} Let $s,l,\alpha,b\in\R$. Then
    \begin{equation}
    \label{EqCPNSob}
      \| \chi u \|_{\Hch^{s,l,\alpha,b}(X;\mu)} \sim h^{\frac{\alpha_\mu}{2}-\alpha} \| \pi^*(\chi u) \|_{H_{\bop,\scop}^{s,l,b-\alpha}(\tface;\hat\mu)},\qquad u\in\Hch^{s,l,\alpha,b}(X;\mu),
    \end{equation}
    in the sense that the left hand side is bounded by a uniform constant (independent of $h$ and $u$) times the right hand side and vice versa.
  \item\label{ItCPNSobVar}{\rm (Variable orders.)} Let $\sfb\in\CI(\ol{\Tch^*_\sface}X_\chop)$ denote a variable order, and let $\sfb':=\sfb|_{\tface\cap\sface}$. If $\sfb$ is invariant under the lift of the dilation action $(x,y)\mapsto(\lambda x,y)$ in $\cU$, then $\|\chi u\|_{\Hch^{s,l,\alpha,\sfb}(X;\mu)}\sim h^{\frac{\alpha_\mu}{2}-\alpha}\|\pi^*(\chi u)\|_{H_{\bop,\scop}^{s,l,\sfb'-\alpha}(\tface;\hat\mu)}$. For general $\sfb$, and given $\delta>0$, there exists $x_0(\delta)\in(0,x_0]$ so that for $\chi\in\CIc([0,x_0(\delta))\times\pa X)$, we have
    \begin{equation}
    \label{EqCPNSobLossy}
    \begin{split}
      &C^{-1}h^{\frac{\alpha_\mu}{2}-\alpha}\|\pi^*(\chi u)\|_{H_{\bop,\scop}^{s,l,\sfb'-\alpha-\delta}(\tface;\hat\mu)} \\
      &\qquad \leq \|\chi u\|_{\Hch^{s,l,\alpha,\sfb}(X;\mu)} \leq C h^{\frac{\alpha_\mu}{2}-\alpha}\|\pi^*(\chi u)\|_{H_{\bop,\scop}^{s,l,\sfb'-\alpha+\delta}(\tface;\hat\mu)},
    \end{split}
    \end{equation}
    where $C$ does not depend on $h,u$.
  \end{enumerate}
\end{cor}
\begin{proof}
  By factoring out $h^{-\alpha}$, it suffices to consider the case $\alpha=0$. Consider first the case of constant orders. Factoring out the appropriate powers of $\tfrac{x}{x+h}=\tfrac{\hat x}{\hat x+1}$ and $\tfrac{h}{h+x}=(\hat x+1)^{-1}$, we reduce to the case $l=b=0$. For $s=0$, the equivalence of norms~\eqref{EqCPNSob} then follows from
  \[
    \iint |\chi u|^2\,x^{\alpha_\mu}\frac{\dd x}{x}\,\dd y = \iint |\pi^*(\chi u)|^2\,h^{\alpha_\mu}\hat x^{\alpha_\mu}\frac{\dd\hat x}{\hat x}\,\dd y.
  \]
  For $s\in\Z$, the conclusion follows from~\eqref{EqCPNDiff}; for general $s\in\R$, use duality and interpolation.

  For variable semiclassical orders $\sfb$ (and still with $\alpha=0$), and under the assumption of dilation-invariance near $\tface_2$, we first pick an elliptic operator $\hat A\in\Psi_{\bop,\scop}^{s,l,\sfb'}(\tface)$; we can then extend its Schwartz kernel to a neighborhood of $\tface_2$ to be constant along the orbits of $(h,x)\mapsto(\lambda h,\lambda x)$, and then extend it further to an elliptic operator $A\in\Psi_{\chop,\cl}^{s,l,0,\sfb}(X)$. In this manner, we obtain a right inverse (with special properties) of the restriction map~\eqref{EqCPNpsdo}. For any fixed $b_0<\inf\sfb$, we thus have
  \begin{align*}
    \|\chi u\|_{\Hch^{s,l,0,\sfb}(X;\mu)}^2 &\sim \|\chi u\|_{\Hch^{s,l,0,b_0}(X;\mu)}^2 + \|A(\chi u)\|_{\Hch^0(X;\mu)}^2 \\
      &\sim h^{\frac{\alpha_\mu}{2}} \Bigl( \|\pi^*(\chi u)\|_{H_{\bop,\scop}^{s,l,b_0}(\tface;\hat\mu)}^2 + \|\hat A(\pi^*(\chi u))\|_{H_{\bop,\scop}^0(X;\hat\mu)}^2 \Bigr) \\
      &\sim h^{\frac{\alpha_\mu}{2}} \|\chi u\|_{H_{\bop,\scop}^{s,l,\sfb'}(\tface;\hat\mu)}^2.
  \end{align*}
  The lossy estimate~\eqref{EqCPNSobLossy} is an immediate consequence of this, as the dilation-invariant extension of $\sfb'-\delta$, resp.\ $\sfb'+\delta$ is less, resp.\ greater than $\sfb$ in a sufficiently small (depending on $\sfb$ and $\delta$) neighborhood of $\pa X$.
\end{proof}

\subsection{Relative semiclassical b-regularity}
\label{SsCb}

We now make Remark~\ref{RmkI2nd} precise and demonstrate how to combine the notions of semiclassical cone regularity and semiclassical b- (i.e.\ conormal) regularity. Recall here that a semiclassical b-vector field is a particular type of $h$-dependent b-vector field on $X$; namely, it is a vector field on $[0,1)_h\times X$ which is horizontal and which vanishes at $h=0$. In local coordinates as in~\eqref{EqbscVF}, such a vector fields can be written as
\begin{equation}
\label{EqCbVF}
  a(h,x,y) h x\pa_x + \sum_{j=1}^{n-1} b^j(h,x,y)h\pa_{y^j}.
\end{equation}
The main insight is that the semiclassical b-algebra can be embedded into the semiclassical cone algebra via a phase space resolution, see Lemma~\ref{LemmaCbPhase} below; this can alternatively be phrased as a second microlocalization of the semiclassical b-algebra at the zero section over $\pa X$ at $h=0$, see Remark~\ref{RmkCb2nd}.

First, we explain a slightly nonstandard perspective on semiclassical (b-)phase spaces. Let $X$ be an $n$-dimensional manifold with nonempty embedded boundary $\pa X$. Thus, paralleling Definition~\ref{DefCV}, we define
\[
  X_{\bop\semi} := [0,1)\times X,\qquad
  \cV_{\bop\semi}(X_{\bop\semi}) = \{ V\in\Vb(X_{\bop\semi}) \colon V\ \text{is horizontal},\ V|_{h=0}=0 \}.
\]
It is then easy to see that $\cV_{\bop\semi}(X_{\bop\semi})$ is spanned over $\CI(X_{\bop\semi})$ by $h V$ for $V\in\Vb(X)$ (cf.\ \eqref{EqCbVF}), where we identify $V$ with an $h$-independent horizontal vector field on $X_{\bop\semi}$. We then have $\cV_{\bop\semi}(X_{\bop\semi})=\CI(X_{\bop\semi};{}^{\bop\semi}T X_{\bop\semi})$ for a rank $n$ vector bundle
\[
  {}^{\bop\semi}T X_{\bop\semi} \to X_{\bop\semi}.
\]
In local coordinates $[0,\infty)_x\times\R^{n-1}_y$, a smooth frame of this bundle is $h x\pa_x,h\pa_{y^1},\ldots,h\pa_{y^{n-1}}$. We can introduce fiber-linear coordinates on the dual bundle ${}^{\bop\semi}T^*X_{\bop\semi}$ by writing the canonical 1-form as
\[
  \xi_{\bop\semi}h^{-1}\frac{\dd x}{x} + \sum_{j=1}^{n-1} (\eta_{\bop\semi})_j h^{-1}\dd y^j.
\]
Thus, for example, the symbol of the semiclassical b-differential operator $h x D_x$ is $\xi_{\bop\semi}$.\footnote{By contrast, the standard convention is to introduce fiber-linear coordinates $(\xi_\bop,\eta_\bop)$ on $\Tb^*X$ as in~\eqref{EqbscCan1} and \emph{declare} the principal symbol of $h x D_x$ to be $\xi_\bop$; the translation to the present convention is accomplished by using (the adjoint of) the bundle isomorphism ${}^{\bop\semi}T X_{\bop\semi}\cong[0,1)_h\times\Tb X$ induced by division by $h$ (i.e.\ induced by the map $\cV_{\bop\semi}(X_{\bop\semi})\ni V\mapsto (h^{-1}V)_{h\in[0,1)}$).} Denote fiber infinity of the radial compactification $\ol{{}^{\bop\semi}T^*}X_{\bop\semi}$ by ${}^{\bop\semi}S^*X_{\bop\semi}$. Given a symbol $a\in S^{s,l,b}(\ol{{}^{\bop\semi}T^*}X_{\bop\semi})=x^{-l}h^{-b}S^{s,0,0}(\ol{{}^{\bop\semi}T^*}X_{\bop\semi})$ (conormal with weight $-s$ at ${}^{\bop\semi}S^*X_{\bop\semi}$, conormal with weight $-l$ at $\ol{{}^{\bop\semi}T^*_{[0,1)\times X}}X_{\bop\semi}$, and conormal with weight $-b$ at $h=0$), we can then define the semiclassical quantization
\begin{align*}
  (\Op_{\bop,h}(a)u)(x,y) &:= (2\pi)^{-n}\iiiint \exp\Bigl(i\Bigl[\frac{x-x'}{h x}\xi_{\bop\semi} + \frac{y-y'}{h}\cdot\eta_{\bop\semi}\Bigr]\Bigr) \phi\Bigl(\Bigl|\log\frac{x}{x'}\Bigr|\Bigr)\phi(|y-y'|) \\
    &\quad\hspace{10em} a(x,y,\xi_{\bop\semi},\eta_{\bop\semi})u(x',y')\,\frac{\dd x'}{h x'}\,\frac{\dd y'}{h^{n-1}}\,\dd\xi_{\bop\semi}\,\dd\eta_{\bop\semi}.
\end{align*}

If we make the change of variables
\begin{equation}
\label{EqCbCoV}
  (\xi_{\bop\semi},\eta_{\bop\semi})=(x+h)(\xi_\chop,\eta_\chop),
\end{equation}
cf.\ \eqref{EqCVCoord}, this exactly matches the $\chop$-quantization~\eqref{EqCPQuant}. The key point is now that this match is has a clean interpretation on the level of symbol classes on a joint resolution of the semiclassical cone and b-phase spaces:

\begin{lemma}[Relationship between semiclassical cone and b-phase spaces]
\label{LemmaCbPhase}
  Define the $\cop\bop\semi$-phase space
  \begin{equation}
  \label{EqCbTchb}
    \ol{{}^{\cop\bop\semi}T^*}X_\chop := \bigl[ \ol{\Tch^*}X_\chop; \Sch^*_\tface X_\chop \bigr].
  \end{equation}
  Denote by $\cC:=\ol{{}^{\bop\semi}T^*_{\{0\}\times\pa X}}X_{\bop\semi}$ the semiclassical b-phase space over the corner $h=x=0$, and denote by $o_\cC\subset\cC$ the zero section. Then the identity map on $(0,1)_h\times T^*X^\circ$ extends by continuity to a diffeomorphism
  \begin{equation}
  \label{EqCbPhaseDiff}
    \ol{{}^{\cop\bop\semi}T^*}X_\chop \xra{\cong} \bigl[ \ol{{}^{\bop\semi}T^*}X_{\bop\semi}; \cC; o_\cC \bigr].
  \end{equation}
\end{lemma}

We refer to the front face of~\eqref{EqCbTchb} as $\fbface$ (`finite b-frequencies'). See Figure~\ref{FigCbPhase}.

\begin{figure}[!ht]
\centering
\includegraphics{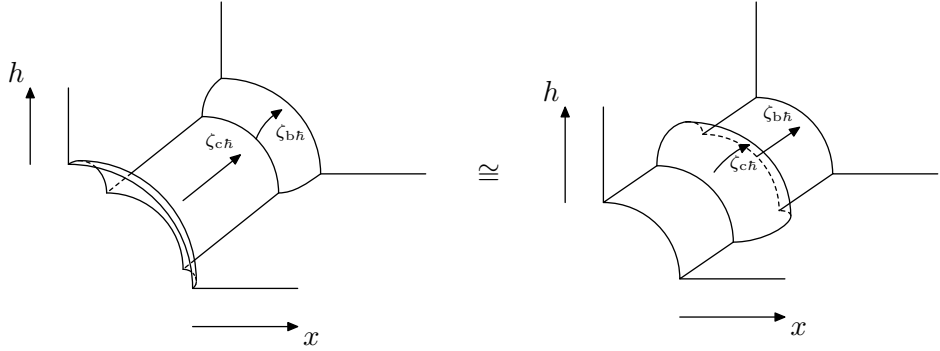}
\caption{\textit{On the left:} the resolution $\ol{{}^{\cop\bop\semi}T^*}X_\chop$ of the fiber-compactified semiclassical cone phase space at fiber infinity over $\tface$, see~\eqref{EqCbTchb}. (Unlike in Figure~\ref{FigCPSymbol}, we show the full compactified fibers here.) \textit{On the right:} the resolution of the fiber-compactified semiclassical b-phase space at $x=h=0$ and at the zero section over $x=h=0$, see~\eqref{EqCbPhaseDiff}.}
\label{FigCbPhase}
\end{figure}

\begin{proof}[Proof of Lemma~\usref{LemmaCbPhase}]
  We work in polar coordinates $\rho=x+h$, $\theta=\frac{(x,h)}{|(x,h)|}$ in the $(x,h)$ variables. Thus, local coordinates near $\Tch^*_\tface X_\chop$ are $(\rho,y,\theta,\zeta_\chop)$, $\zeta_\chop:=(\xi_\chop,\eta_\chop)$, while local coordinates near the front face of $[\ol{{}^{\bop\semi}T^*}X_{\bop\semi};\cC]$, away from fiber infinity, are $(\rho,y,\theta,\zeta_{\bop\semi})$, $\zeta_{\bop\semi}=(\xi_{\bop\semi},\eta_{\bop\semi})$. Coordinates near the interior of the front face of the final blow-up in~\eqref{EqCbPhaseDiff} are then $(\rho,y,\theta,\frac{\zeta_{\bop\semi}}{\rho})=(\rho,y,\theta,\zeta_\chop)$, see~\eqref{EqCbCoV}. Near the intersection of the lift of $o_\cC$ with that of $\cC$, smooth coordinates can be constructed by introducing polar coordinates in the fiber variables, giving $(\frac{\rho}{|\zeta_{\bop\semi}|},y,\theta,|\zeta_{\bop\semi}|,\frac{\zeta_{\bop\semi}}{|\zeta_{\bop\semi}|})$; this matches, up to a permutation, the local coordinates on $\ol{{}^{\cop\bop\semi}T^*}X_\chop$ near the lift of $\ol{\Tch^*_\tface}X_\chop$ given by $(\frac{\rho}{|\zeta_\chop|^{-1}},y,\theta,|\zeta_\chop|^{-1},\frac{\zeta_\chop}{|\zeta_\chop|})$. Lastly, near the lift of fiber infinity on the resolved b-phase space, we can use coordinates $(\rho,y,\theta,|\zeta_{\bop\semi}|^{-1},\frac{\zeta_{\bop\semi}}{|\zeta_{\bop\semi}|})$, which matches the local coordinates near the lift of $\Sch^*X_\chop$ given by $(\rho,y,\theta,\frac{|\zeta_\chop|^{-1}}{\rho},\frac{\zeta_\chop}{|\zeta_\chop|})$.
\end{proof}

The blow-up of a boundary face does not enlarge the space of conormal distributions, but allows for more precise accounting of weights. Concretely, define for $s,s',l,\alpha,b\in\R$ the symbol space
\begin{equation}
\label{EqCbSymbolsSpace}
  S^{s,s',l,\alpha,b}(\ol{{}^{\cop\bop\semi}T^*}X_\chop),
\end{equation}
where the orders refer, in this order, to fiber infinity, the front face $\fbface$ of~\eqref{EqCbTchb}, and the phase space over the lifts of $\cface$, $\tface$ and $\sface$, see Figure~\ref{FigCbSymbols}. Then we have
\begin{equation}
\label{EqCbSymbols}
\begin{split}
  S^{s,l,\alpha,b}(\ol{\Tch^*}X_\chop) &= S^{s,s+\alpha,l,\alpha,b}(\ol{{}^{\cop\bop\semi}T^*}X_\chop), \\
  S^{s,s',l,\alpha,b}(\ol{{}^{\cop\bop\semi}T^*}X_\chop) &\subset S^{\max(s,s'-\alpha),l,\alpha,b}(\ol{\Tch^*}X_\chop).
\end{split}
\end{equation}
Note that the second inclusion is \emph{false} if we use spaces of \emph{classical} symbols on both sides; after all, blow-ups do enlarge the space of smooth functions (but preserve the space of conormal functions). Since we worked with general conormal symbols and ps.d.o.s in~\S\ref{SsCP}, we can immediately quantize symbols on the $\cop\bop\semi$-phase space:

\begin{definition}[$\cop\bop\semi$-pseudodifferential operators]
\label{DefCbPsdo}
  Let $s,s',l,\alpha,b\in\R$. Then we define
  \[
    \Psi_{\cop\bop\semi}^{s,s',l,\alpha,b}(X) := \Op_\chop\bigl(S^{s,s',l,\alpha,b}(\ol{{}^{\cop\bop\semi}T^*}X_\chop)\bigr) + \Psich^{-\infty,l,\alpha,-\infty}(X),
  \]
  Operators with variable semiclassical orders $\sfb\in\ol{\Tch^*_\sface}X_\chop$ are defined similarly.
\end{definition}

\begin{figure}[!ht]
\centering
\includegraphics{FigCbSymbol}
\caption{Illustration of the orders of (symbols of) $\cop\bop\semi$-pseudodifferential operators in~\eqref{EqCbSymbolsSpace} and Definition~\ref{DefCbPsdo}.}
\label{FigCbSymbols}
\end{figure}

\begin{rmk}[Second microlocalization]
\label{RmkCb2nd}
  In view of Lemma~\ref{LemmaCbPhase}, one can view $\Psi_{\cop\bop\semi}(X)$ as a second microlocalization of the (conormal) semiclassical b-algebra $\Psi_{\bop\semi}(X)$ at the zero section over $h=x=0$. In terms of symbol classes, we have
  \begin{equation}
  \label{EqCbSymbols2}
  \begin{split}
    S^{s,l,b}(\ol{{}^{\bop\semi}T^*}X_{\bop\semi}) &= S^{s,l+b,l,l+b,b}(\ol{{}^{\cop\bop\semi}T^*}X_\chop), \\
    S^{s,s',l,\alpha,b}(\ol{{}^{\cop\bop\semi}T^*}X_\chop) &\subset S^{s,l,\max(b,s'-l,\alpha-l)}(\ol{{}^{\bop\semi}T^*}X_{\bop\semi}),
  \end{split}
  \end{equation}
  and analogous statements hold for ps.d.o.s. However, similarly to~\cite[\S5]{VasyLowEnergy} in the context of b- and scattering algebras, it is analytically advantageous to resolve $\Psi_\chop(X)$ as in Definition~\ref{DefCbPsdo}, as the second microlocal/resolved algebra involves \emph{global} (noncommutative) phenomena at $h=x=0$ (i.e.\ the lift of $\tface$, associated to which is the normal operator homomorphism into a noncommutative algebra) which are directly inherited from $\Psich(X)$, but which are not visible on the level of $\Psi_{\bop\semi}(X)$.
\end{rmk}

For two ps.d.o.s $A_j\in\Psi_{\cop\bop\semi}^{s_j,s'_j,l_j,\alpha_j,b_j}(X)$, one can compute the full symbol (i.e.\ the symbol modulo $S^{-\infty,l_1+l_2,\alpha_1+\alpha_2,-\infty}(\ol{\Tch^*}X_\chop)=S^{-\infty,-\infty,l_1+l_2,\alpha_1+\alpha_2,-\infty}(\ol{{}^{\cop\bop\semi}T^*}X_\chop)$) of the composition $A_1\circ A_2\in\Psi_{\cop\bop\semi}^{\max(s_1,s'_1-\alpha_1)+\max(s_2,s'_2-\alpha_2),l_1+l_2,\alpha_1+\alpha_2,b_1+b_2}(X)$ in local coordinates using the usual symbol expansion to be the sum of products of derivatives of the full symbols of the two factors along b-vector fields on $\ol{\Tch^*}X_\chop$ which vanish, as b-vector fields, at $\Sch^*X_\chop$ (thus vanishing as b-vector fields at the lift of $\Sch^*X_\chop$ as well as at the front face of~\eqref{EqCbTchb}) and at the lift of $\ol{\Tch^*_\sface}X_\chop$. Plugging the $\cop\bop\semi$-symbols of $A_1,A_2$ into such an expansion thus shows that, in fact,
\[
  A_1\circ A_2 \in \Psi_{\cop\bop\semi}^{s_1+s_2,s'_1+s'_2,l_1+l_2,\alpha_1+\alpha_2,b_1+b_2}(X).
\]
Similar arguments show that the principal symbol map
\[
  {}^{\cop\bop\semi}\sigma \colon \Psi_{\cop\bop\semi}^{s,s',l,\alpha,b}(X) \to (S^{s,s',l,\alpha,b}/S^{s-1,s'-1,l,\alpha,b-1})(\ol{{}^{\cop\bop\semi}T^*}X_\chop)
\]
is well-defined (and a *-homomorphism as usual). One can moreover define an associated scale of Sobolev spaces
\begin{equation}
\label{EqCbSob}
  H_{\cop\bop,h}^{s,s',l,\alpha,b}(X) = \bigl\{ u\in\Hch^{\min(s,s'-\alpha),l,\alpha,b}(X) \colon A u\in L^2(X)\ \forall\,A\in\Psi_{\cop\bop\semi}^{s,s',l,\alpha,b}(X) \bigr\},
\end{equation}
The relationships~\eqref{EqCbSymbols} and \eqref{EqCbSymbols2} imply:

\begin{prop}[Relationships between Sobolev spaces]
\label{PropCbSob}
  Let $s,s',l,\alpha,b\in\R$. Define $L^2$ using the volume density $\mu=x^{\alpha_\mu}\mu_0$, $0<\mu_0\in\CI(X;\Omegab^1 X)$ with $\alpha_\mu\in\R$. Then
  \begin{align*}
    \Hch^{s,l,\alpha,b}(X) &= H_{\cop\bop,h}^{s,s+\alpha,l,\alpha,b}(X), \\
    \Hbh^{s,l,b}(X) &= H_{\cop\bop,h}^{s,l+b,l,l+b,b}(X).
  \end{align*}
\end{prop}

One can conversely embed $H_{\cop\bop,h}^{s,s',l,\alpha,b}(X)$ into $\Hch^{\tilde s,\tilde l,\tilde\alpha,\tilde b}(X)$ and $\Hbh^{\tilde s,\tilde l,\tilde b}(X)$ under suitable inequalities (which can be read off from Proposition~\ref{PropCbSob}) between the orders. In particular, this allows us to give a direct proof of \cite[Proposition~3.18]{HintzConicPowers} on the relationship between $\Hch(X)$ and $\Hbh(X)$; for instance, for $s,l,\alpha\in\R$ (denoted $s,\alpha,\tau$ in the reference), we have
\begin{equation}
\label{EqCbProp1}
  \Hch^{s,l,\alpha,0}(X) = H_{\cop\bop,h}^{s,s+\alpha,l,\alpha,0}(X) \subset \Hbh^{s,l,\min(0,\alpha-l,\alpha-l+s)}(X),
\end{equation}
which implies (and is slightly sharper than) the first part of \cite[Proposition~3.18]{HintzConicPowers}. If one wishes to translate estimates on cone spaces to b-spaces, the advantage of the resolved $\cop\bop\semi$-Sobolev spaces, compared with $\chop$-Sobolev spaces, is that one can reduce losses in powers of $h$ (or in regularity) in the conversion; as a simple concrete example, we have
\[
  H_{\cop\bop,h}^{s,s'+\alpha,l,\alpha,0}(X) \subset \Hbh^{s,l,\min(0,\alpha-l,\alpha-l+s')}(X),
\]
which for $s'\geq -s_-$ gives an improved bound at $h=0$, and for $s'\geq 0$ a bound which is independent of the differential orders $s,s'$, unlike~\eqref{EqCbProp1} which gets lossier as $s$ decreases.

\begin{rmk}[Variable semiclassical orders]
  The above discussion applies, mutatis mutandis, to symbols and operators with variable semiclassical orders $\sfb$ as well; here $\sfb$ is a smooth function on the lift of $\ol{\Tch^*_\sface}X_\chop$ to $\ol{{}^{\cop\bop\semi}T^*}X_\chop$.
\end{rmk}

\section{Microlocal propagation estimates at cone points and generalizations}
\label{SP}

Let $n\geq 1$. We work locally near a cone point, thus on an $n$-dimensional manifold
\begin{equation}
\label{EqPMfd}
  X = [0,2 x_0)_x \times Y,\quad x_0>0,
\end{equation}
where $Y$ is a closed connected $(n-1)$-dimensional manifold, and where $X^\circ=(0,2 x_0)\times Y$ is equipped with a smooth Riemannian metric $g$ of the form
\begin{equation}
\label{EqPProdMet}
  g = \dd x^2 + x^2 k(x,y,\dd y),
\end{equation}
where $k\in\CI([0,x_0);\CI(Y,S^2 T^*Y))$ is a smooth family of smooth Riemannian metrics on the cross section $Y$. Any metric which locally near $\pa X$ is of the form $\dd\tilde x^2+\tilde x^2 k(\tilde x,y,\dd\tilde x,\dd y)$ with $k|_{\pa X}$ a Riemannian metric on $\pa X$ is of the form~\eqref{EqPProdMet} in a suitable smooth collar neighborhood of $\pa X$, as shown in~\cite[\S1]{MelroseWunschConic}.

While the above $X$ is not compact, all calculations and estimates will take place in the compact subset $[0,x_0]\times Y$ of $X$; thus, we shall commit a slight abuse of notation and write $\|u\|_{\Hch^s(X)}$ etc.\ for norms of functions $u$ on $X$ which will always have support in $x^{-1}([0,x_0])$. We fix the volume density
\begin{equation}
\label{EqPVol}
  \mu = |\dd g| = x^{n-1}|\dd x\,\dd k| \in x^n\CI(X;\Omegab^1 X)
\end{equation}
on $X$, and define Sobolev spaces relative to $L^2(X):=L^2(X;\mu)$. We moreover define
\begin{equation}
\label{EqPhat}
\begin{alignedat}{2}
  \hat x &:= \tfrac{x}{h}, &\qquad
  \hat g &:= \dd\hat x^2 + \hat x^2 k(0,y,\dd y), \\
  \hat h &:= \hat x^{-1} = \tfrac{h}{x}, &\qquad
  \hat\mu &:= |\dd\hat g|=\hat x^{n-1}|\dd\hat x\,\dd k(0)|.
\end{alignedat}
\end{equation}

\subsection{Admissible operators}
\label{SsPAdm}

The class of operators of interest to us is the following.

\begin{definition}[Admissible operators]
\label{DefPOp}
  We call an $h$-dependent differential operator $P_{h,z}$ on $X^\circ$ \emph{admissible} if it is of the form
  \begin{equation}
  \label{EqPOp}
    P_{h,z} = h^2\Delta_g - z + h^2 x^{-2}Q_{1,z} + h x^{-1}q_{0,z},
  \end{equation}
  where $Q_{1,z}\in\Diffb^1(X)$ and $q_{0,z}\in\CI(X)$ depend smoothly on $z\in\C$, $|z-1|<C h$.
\end{definition}

We shall henceforth take $z=z(h)$ to be a smooth function of $h\in[0,1)$ with $z(0)=1$.

\begin{rmk}[Vector bundles]
\label{RmkPBundle0}
  Our analysis applies also to operators acting on sections of a vector bundle $E\to X$; we explain the necessary (largely notational) changes in Remark~\ref{RmkPBundle}.
\end{rmk}

Using local coordinates $y\in\R^{n-1}$ on $\pa X$, let us write $Q_{1,z}=q_{1,z}(x,y,x D_x,D_y)$ and $q_{0,z}=q_{0,z}(x,y)$. The normal operator of $P_{h,z}$ is
\begin{equation}
\label{EqPN}
\begin{split}
  N(P) &:= \Delta_{\hat g} - 1 + \hat x^{-2}q_{1,1}(0,y,\hat x D_{\hat x},D_y) + \hat x^{-1}q_{0,1}(0,y) \\
    &= D_{\hat x}^2 - i(n-1)D_{\hat x} + \hat x^{-2}\Delta_{k(0)} - 1 + \hat x^{-2}q_{1,1}(0,y,\hat x D_{\hat x},D_y) + \hat x^{-1}q_{0,1}(0,y)
\end{split}
\end{equation}
on $\tface=[0,\infty]_{\hat x}\times\pa X$.\footnote{This can be defined more invariantly as an operator on the inward pointing normal bundle ${}^+N\pa X$, which is the natural place for the b-normal operators $q_{1,1}(0,y,x D_x,D_y)$ and $q_{0,1}(0,y)$ to live; see \cite[\S4.15]{MelroseAPS} and \cite[\S3]{HintzConicPowers} for details.}

\begin{lemma}[Structural properties]
\label{LemmaPStruct}
  We have $P_{h,z}\in(\tfrac{x}{x+h})^{-2}\Diffch^2(X_\chop)$ and $N(P)\in(\tfrac{\hat x}{\hat x+1})^{-2}\Diff_{\bop,\scop}^2(\tface)$. Furthermore, $P_{h,z}-N(P)\in(x+h)\bigl(\tfrac{x}{x+h}\bigr)^{-2}\Diffch^2(X_\chop)$.
\end{lemma}
\begin{proof}
  In local coordinates $y^1,\ldots,y^{n-1}$ on $Y$, the metric $k(x,y,\dd y)$ is given by an $(n-1)\times(n-1)$ matrix $(k_{i j})$ with determinant $|k|>0$ and inverse $(k^{i j})$, and we have
  \begin{align*}
    \Delta_g &= |k|^{-\frac12} x^{-n+1} D_x\bigl(|k|^{\frac12} x^{n-1}D_x\bigr) + x^{-2}\Delta_{k(x)} \\
      &= D_x^2 - i(n-1+x\gamma)x^{-1}D_x + \sum_{i,j=1}^{n-1} x^{-2}|k|^{-\frac12}D_{y^i}\bigl(|k|^{\frac12}k^{i j}D_{y^j}\bigr),
  \end{align*}
  where $\gamma=\half\pa_x\log|k|\in\CI$. Since
  \begin{equation}
  \label{EqPStructDiff}
  \begin{alignedat}{2}
    h D_x&=\tfrac{x+h}{x}\cdot\tfrac{h}{h+x}x D_x&&\in\tfrac{x+h}{x}\Vch(X_\chop), \\
    h x^{-1}D_{y^i}&=\tfrac{x+h}{x}\cdot \tfrac{h}{h+x}D_{y^i}&&\in\tfrac{x+h}{x}\Vch(X_\chop), \\
    h x^{-1}&=\tfrac{x+h}{x}\cdot\tfrac{h}{h+x} &&\in \tfrac{x+h}{x}\Diffch^1(X_\chop),
  \end{alignedat}
  \end{equation}
  we find $h^2\Delta_g\in(\frac{x}{x+h})^{-2}\Diffch^2(X_\chop)$, with normal operator $D_{\hat x}^2-i(n-1)\hat x^{-1}D_{\hat x}+\hat x^{-2}\Delta_{k(0)}=\Delta_{\hat g}$. The remaining terms in~\eqref{EqPOp} are analyzed similarly.
\end{proof}

\subsection{Characteristic set, Hamilton flow}
\label{SsPChar}

Using the fiber-linear coordinates $(\xi_\chop,\eta_\chop)$ on $\Tch^*X_\chop$ from~\eqref{EqCVCoord}, we can read off the principal symbol from~\eqref{EqPStructDiff} to be
\[
  p := \bigl(\tfrac{x}{x+h}\bigr)^2\cdot\sigmach(P_{h,z}) = \xi_\chop^2 + |\eta_\chop|_{k^{-1}}^2 - 1.
\]
(Here, we use that $z=1+\cO(h)$, hence the principal symbol of $z$ is $1$.) This is elliptic at fiber infinity $\Sch^*X_\chop$, but has a nonempty characteristic set at finite frequencies. Near $\sface$, it is more convenient to use the fiber coordinates $(\xi,\eta)$ from~\eqref{EqCVCoordTilde}, and $\hat h=\frac{h}{x},x,y$ as coordinates on the base, so that
\begin{equation}
\label{EqPSymb}
  p = \xi^2 + k^{i j}(x,y)\eta_i\eta_j - 1, \qquad
  \Sigma = p^{-1}(0) \cap \Tch^*_\sface X_\chop = \{ \hat h=0,\ \xi^2+|\eta|_{k^{-1}}^2 = 1\}.
\end{equation}
Using~\eqref{EqCVHam} and writing $|\eta|^2=k^{i j}\eta_i\eta_j$, we then compute
\begin{equation}
\label{EqPSymbHam}
\begin{split}
  \sfH := \hat h^{-1}H_p &= 2\xi(x\pa_x-\hat h\pa_{\hat h}-\eta\pa_\eta) + \bigl(2|\eta|^2 - x\pa_x k^{i j}\eta_i\eta_j\bigr)\pa_\xi \\
    &\qquad + 2 k^{i j}\eta_i\pa_{y^j} - (\pa_{y^k}k^{i j})\eta_i\eta_j\pa_{\eta_k}.
\end{split}
\end{equation}
Restricted to $x=0$ as a b-vector field on $\ol{\Tch^*}X_\chop$, this is
\begin{equation}
\label{EqPHamResc}
  \sfH|_{x=0} = 2\xi(x\pa_x-\hat h\pa_{\hat h}-\eta\pa_\eta) + 2|\eta|^2\pa_\xi + \bigl( 2 k^{i j}\eta_i\pa_{y^j} - (\pa_{y^k}k^{i j})\eta_i\eta_j\pa_{\eta_k} \bigr).
\end{equation}
This vanishes as a standard vector field on $\hat h=x=0$ if and only if $\eta=0$. The intersection of $\eta^{-1}(0)$ with $\Sigma\cap x^{-1}(0)$ has two components: the \emph{incoming} and \emph{outgoing radial sets} $\cR_{\rm in/out}\subset\Tch^*_\sface X_\chop$,
\begin{equation}
\label{EqPRadSets}
\begin{split}
  \cR_{\rm in} &:= \{ (\hat h,x,y,\xi,\eta) \colon \hat h=0,\ x=0,\ y\in\pa X,\ \xi=-1,\ \eta=0 \}, \\
  \cR_{\rm out} &:= \{ (\hat h,x,y,\xi,\eta) \colon \hat h=0,\ x=0,\ y\in\pa X,\ \xi=+1,\ \eta=0 \}.
\end{split}
\end{equation}
These are saddle points for the rescaled Hamilton vector field $\sfH$ since
\begin{equation}
\label{EqPRadSetsLin}
  x^{-1}\sfH x=\mp 2,\quad
  \hat h^{-1}\sfH\hat h=\pm 2,\quad
  |\eta|^{-2}\sfH|\eta|^2=\pm 4\qquad\text{at}\ \cR_{\rm in/out}.
\end{equation}
(The top sign is for `in', the bottom sign for `out'.) See Figure~\ref{FigPFlow}.

\begin{figure}[!ht]
\centering
\includegraphics{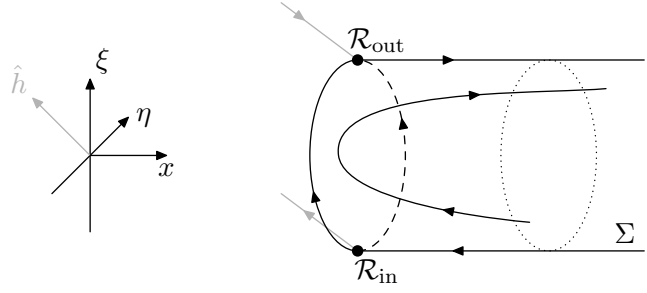}
\caption{Illustration of the flow along the rescaled Hamilton vector field $\sfH$, see~\eqref{EqPSymbHam}, through the radial sets $\cR_{\rm in}$ and $\cR_{\rm out}$. Shown is the characteristic set, the fibers of which over $\sface$ are spheres (here 1-spheres); one fiber is drawn as a dotted circle. Also indicated is (in gray) the linearization of $\sfH$ at $\cR_{\rm in/out}$ over $\tface$.}
\label{FigPFlow}
\end{figure}

Over $\Tch^*_\tface X_\chop$, the set $\cR_{\rm in}$ is a radial \emph{source} (though this really only makes sense infinitesimally at $\tface\cap\sface$ since the $\chop$-calculus is not symbolic over $\tface\setminus\sface$), and $\cR_{\rm out}$ is a radial \emph{sink}. This matches precisely the familiar situation of scattering theory on the asymptotically conic space $(\tface,\hat g)$, see~\cite{MelroseEuclideanSpectralTheory}, which we discuss in detail in~\S\ref{SsPSc}.

In $x>0$, the flow of $\sfH$ is a reparameterization of the flow of $h^{-1}H_p=x^{-1}\sfH$. Integral curves of $\sfH$ starting over a point in $X^\circ$ never reach $\pa X$ in finite time. Instead, we consider
\begin{equation}
\label{EqPHsf}
\begin{split}
  \sfH_\sface := h^{-1}H_p|_{\hat h=0} &= 2\xi(\pa_x - x^{-1}\eta\pa_\eta) + \bigl(2 x^{-1}|\eta|^2-\pa_x k^{i j}\eta_i\eta_j\bigr)\pa_\xi \\
    &\qquad + x^{-1}\bigl( 2 k^{i j}\eta_i\pa_{y^j} - (\pa_{y^k}k^{i j})\eta_i\eta_j\pa_{\eta_k} \bigr).
\end{split}
\end{equation}
Given $y_0\in\pa X$, the curves
\begin{equation}
\label{EqPInOutCurves}
\begin{alignedat}{2}
  \gamma_{I,y_0}(s) &:= (-2 s,y_0,-1,0), &\quad s&\in(-x_0,0), \\
  \gamma_{O,y_0}(s) &:= (2 s,y_0,1,0),&\quad s&\in(0,x_0)
\end{alignedat}
\end{equation}
are integral curves of $\sfH_\sface$. Here, $\gamma_{I,y_0}$ strikes $\pa X$ at $s=0$ at the incoming radial set over point $y_0$, whereas $\gamma_{O,y_0}$ emanates from the outgoing radial set over $y_0$ at $s=0$.

\begin{lemma}[Incoming/outgoing null-bicharacteristics]
\label{LemmaPInOut}
  Let $0<s_0<x_0$, and suppose that $\gamma\colon(0,s_0)\to\Sigma\cap\Tch^*_{\sface\setminus\tface}X_\chop$, is an integral curve of $\sfH_\sface$ tending to $\pa X$ as $s\searrow 0$ in the weak sense that $\liminf_{s\searrow 0}x(\gamma(s))=0$. Then in the coordinates $(x,y,\xi,\eta)$, $\gamma$ is necessarily of the form $\gamma(s)=\gamma_{O,y_0}(s)$ for some $y_0\in\pa X$. Similarly, if $\gamma\colon(-s_0,0)\to\Sigma\cap\Tch^*_{\sface\setminus\tface}X_\chop$ is an integral curve of $\sfH_\sface$ with $\liminf_{s\nearrow 0} x(\gamma(s))=0$, then $\gamma(s)=\gamma_{I,y_0}(s)$ for some $y_0\in\pa X$.
\end{lemma}
\begin{proof}
  The stable/unstable manifold theorem applied to the rescaled vector field $\sfH$ applies at $\cR_{\rm out}$ and produces a stable manifold (namely, $\Sigma\cap\Tch^*_{\sface\cap\tface}X_\chop\setminus\cR_{\rm in}$) and an unstable manifold whose tangent space at a point $\zeta\in\cR_{\rm out}$ is the sum of $T_\zeta(\cR_{\rm out})$ and to the 1-dimensional space spanned by the unique eigenvector $\pa_x$ of the linearization of $\sfH$ with positive eigenvalue. Since the manifold $\cR_{\rm out}\cup\{\gamma_{O,y_0}(s)\colon y_0\in\pa X,\ s\in(0,s_0)\}$ is $\sfH$-invariant with the same tangent space, it must be equal to this unstable manifold. The first part of the lemma follows from this observation; the second part is completely analogous.
\end{proof}

\begin{definition}[Generalized broken bicharacteristics]
\label{DefPBroken}
  Denote by $\dot\Sigma$ the topological space defined as the quotient $\Sigma/\pa\Sigma$. Let $I\subset\R$ denote an open interval. We then say that a continuous curve $\gamma\colon I\to\dot\Sigma$ is a \emph{generalized broken bicharacteristic} (GBB) if either $\gamma(I)\subset\Sigma\setminus\pa\Sigma$ and $\gamma$ is an integral curve of $\sfH_\sface$, or there exist $s_0\in I$ and $y_I,y_O\in\pa X$ so that $\gamma(s_0+t)=\gamma_{O,y_O}(t)$ for $t>0$, $s_0+t\in I$ and $\gamma(s_0+t)=\gamma_{I,y_I}(t)$ for $t<0$, $s_0+t\in I$.\footnote{In light of Lemma~\ref{LemmaPInOut}, this is equivalent to the condition that $\gamma$ is an $\sfH_\sface$-integral curve outside of $\pa\Sigma$, but may enter and exit the characteristic set over $\pa X$ at different points.} If $y_O$ is at distance $\pi$ from $y_I$ with respect to the metric $k(0)$ on $\pa X$, we say that $\gamma$ is a \emph{geometric GBB}, otherwise $\gamma$ is a \emph{strictly diffractive GBB}.
\end{definition}

See Figure~\ref{FigPBroken}. We remark without proof that geometric GBB are uniform limits of $\sfH_\sface$-integral curves just barely missing $\pa X$ (see also \cite[Lemma~1.5]{MelroseWunschConic}).

\begin{figure}[!ht]
\centering
\includegraphics{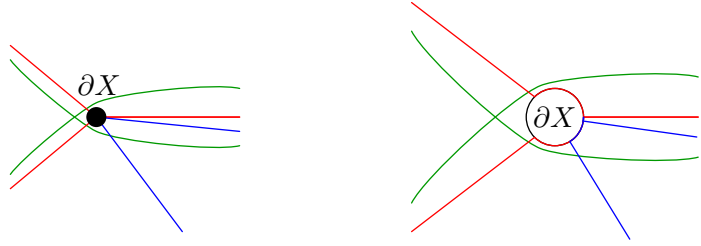}
\caption{The projection of strictly diffractive (blue) and geometric (red) GBBs to the base $X$, as well as geodesics (green) just barely missing the cone tip $\pa X$. \textit{On the left:} the geometric picture, where $\pa X$ is collapsed to a point. \textit{On the right:} the resolved picture.}
\label{FigPBroken}
\end{figure}

\subsection{Scattering theory for the normal operator}
\label{SsPSc}

Propagation through the `cone point' $\pa X$ will require \emph{global} control of the normal operator, namely the absence of purely outgoing or purely incoming solutions (depending on the direction in which one wants to propagate estimates). Let us define fiber-linear coordinates on the scattering cotangent bundle $\Tsc^*(\tface\setminus\cface)$ via
\[
  \xi_\scop\frac{\dd\hat h}{\hat h^2} + \sum_{j=1}^{n-1}(\eta_\scop)_j\frac{\dd y^j}{\hat h}.
\]
Via the identification~\eqref{EqCPNBundles}, the radial sets $\cR_{\rm in/out}$ defined in~\eqref{EqPRadSets} are then equal to the sets ${}^\scop\cR_{\rm in/out}\subset\Tsc^*(\tface\setminus\cface)$, where
\begin{align*}
  {}^\scop\cR_{\rm in} &:= \{ (\hat h,y,\xi_\scop,\eta_\scop) \colon \hat h=0,\ y\in\pa X,\ \xi_\scop=+1,\ \eta_\scop=0 \}, \\
  {}^\scop\cR_{\rm out} &:= \{ (\hat h,y,\xi_\scop,\eta_\scop) \colon \hat h=0,\ y\in\pa X,\ \xi_\scop=-1,\ \eta_\scop=0 \}.
\end{align*}
Invariantly, $\cR_{\rm in}={}^\scop\cR_{\rm in}$ is the graph of $-\frac{x}{h}\frac{\dd x}{x}=-\dd(\hat h^{-1})=\frac{\dd\hat h}{\hat h^2}$, likewise for $\cR_{\rm out}={}^\scop\cR_{\rm out}$ but with an overall sign switch. 

\begin{definition}[Conditions on the normal operator]
\label{DefPNormOp}
  Let $l,l'\in\R$, and recall~\eqref{EqPhat}.
  \begin{enumerate}
  \item\label{ItPNormOpFw} We say that $N(P)$ is \emph{injective at weight $l$ on outgoing functions} if the only solution $u$ to the equation $N(P)u=0$ satisfying $u\in\bigcup_{N\in\R} H_{\bop,\scop}^{\infty,l,-N}(\tface;\hat\mu)$ and $\WFsc(u)\subset{}^\scop\cR_{\rm out}$ is trivial: $u\equiv 0$.
  \item\label{ItPNormOpBw} We say that $N(P)^*$ (the formal adjoint with respect to $L^2(\tface;\hat\mu)$) is \emph{injective at weight $l'$ on incoming functions} if the only solution $v$ to the equation $N(P)^*v=0$ satisfying $v\in\bigcup_{N\in\R} H_{\bop,\scop}^{\infty,l',-N}(\tface;\hat\mu)$ and $\WFsc(v)\subset{}^\scop\cR_{\rm in}$ is trivial: $v\equiv 0$.
  \item\label{ItPNormOpInv} If condition~\eqref{ItPNormOpFw} and condition~\eqref{ItPNormOpBw} with $l'=-l+2$ hold, we say that $N(P)$ is \emph{invertible at weight $l$}.
  \end{enumerate}
\end{definition}

The wave front set assumptions here are the microlocal formulations of outgoing/in\-com\-ing radiation conditions. In the special case that $N(P)=\Delta_{\hat g}-1$, these assumptions are indeed equivalent to the standard Sommerfeld radiation condition. Our goal is to elevate the qualitative conditions of Definition~\ref{DefPNormOp} to quantitative estimates, see Lemma~\ref{LemmaPNEst}.

Changing variables in the expression~\eqref{EqPN} for $N(P)$ to $(\hat h,y)$ gives
\[
  N(P) = (\hat h^2 D_{\hat h})^2 + i(n-1)\hat h^2 D_{\hat h} + \hat h^2\Delta_{k(0)} - 1 + \hat h^2 q_{1,1}(0,y,-\hat h D_{\hat h},D_y) + \hat h q_{0,1}(0,y),
\]
with scattering principal symbol at $\Tsc^*_{\tface\cap\sface}\tface$ given by
\begin{subequations}
\begin{equation}
\label{EqPtfSymb}
  p_\tface=\xi_\scop^2+|\eta_\scop|^2-1.
\end{equation}
Its Hamilton vector field is
\begin{equation}
\label{EqPtfSymbHam}
  \sfH_\tface := \hat h^{-1}H_{p_\tface} = 2\xi_\scop(\hat h\pa_{\hat h}+\eta_\scop\pa_{\eta_\scop}) - 2|\eta_\scop|^2\pa_{\xi_\scop} + \hat h^{-1}H_{|\eta_\scop|^2}
\end{equation}
\end{subequations}
by~\eqref{EqbscHamsc}, which has a source, resp.\ sink structure at ${}^\scop\cR_{\rm in}$, resp.\ ${}^\scop\cR_{\rm out}$ within the characteristic set $p_\tface^{-1}(0)$. Recall then that microlocal propagation estimates near the radial sets ${}^\scop\cR_{\rm in/out}$ require suitable orders---here the decay order---of weighted Sobolev spaces to be above or below certain threshold values, see \cite[\S9]{MelroseEuclideanSpectralTheory}, \cite[\S4.7]{VasyMinicourse}, and \cite[Appendix~E.4]{DyatlovZworskiBook}.

\begin{definition}[Threshold quantities]
\label{DefPThr}
  Define the functions
  \[
    r_1:=\Im\bigl(\sigmab_1(Q_{1,1})(-\tfrac{\dd x}{x})|_{x=0}\bigr)\in\CI(\pa X),\qquad
    r_0:=\Im\bigl(q_{0,1}|_{x=0}\bigr)\in\CI(\pa X).
  \]
  Then the threshold quantities $r_{\rm in/out}\in\R$ are defined as
  \[
    r_{\rm in} := -\half + \half\sup_{\pa X}(r_1+r_0),\quad
    r_{\rm out} := -\half + \half\inf_{\pa X}(r_1-r_0).
  \]
\end{definition}

We next recall that at the other end of $\tface$, i.e.\ the `b-end' $\tface\cap\cface$, the weights $l,l'$ in Definition~\ref{DefPNormOp} are related to the \emph{boundary spectrum} of $N(P)$. Concretely, from the expression~\eqref{EqPN}, we read off
\begin{equation}
\label{EqPNormOpResc}
  \hat x^2 N(P) \in (\hat x D_{\hat x})^2 - i(n-2)\hat x D_{\hat x} + \Delta_{k(0)} + q_{1,1}(0,y,\hat x D_{\hat x},D_y) + \hat x\Diffb(\tface\setminus\sface).
\end{equation}
Its (dilation-invariant in $\hat x$) normal operator at $\hat x=0$ is given by the sum of the first four terms, and the Mellin transformed normal operator family is defined by formally replacing $\hat x D_{\hat x}$ by multiplication with $\lambda\in\C$, giving
\begin{equation}
\label{EqPNormOpMellin}
  \hat N(P)(\lambda,y,D_y) := \lambda^2 - i(n-2)\lambda + \Delta_{k(0)} + q_{1,1}(0,y,\lambda,D_y).
\end{equation}
This is a holomorphic family in $\lambda$ taking values in elliptic elements of $\Diff^2(\pa X)$. The \emph{boundary spectrum} of $N(P)$ is then
\[
  \specb(N(P)) := \{ \lambda\in\C \colon \hat N(P)(\lambda)\colon\CI(\pa X)\to\CI(\pa X)\ \text{is not invertible} \};
\]
it is a discrete subset of $\C$, and its intersection with $|\Im\lambda|<C$ is finite for any fixed value of $C$ \cite[\S5.3]{MelroseAPS}. Let us now put
\begin{equation}
\label{EqPLambda}
  \Lambda := \{ -\Im\lambda \colon \lambda\in\specb(N(P)) \};
\end{equation}
this is a discrete subset of $\R$.

\begin{lemma}[Estimates for $N(P)$]
\label{LemmaPNEst}
  Let $s,l\in\R$ and $\sfr\in\CI(\ol{\Tsc^*_{\tface\cap\sface}}\tface)$. Suppose that $\sfr$ is constant near ${}^\scop\cR_{\rm in/out}$ and satisfies $\sfr>r_{\rm in}$ at ${}^\scop\cR_{\rm in}$, $\sfr<r_{\rm out}$ at ${}^\scop\cR_{\rm out}$. Suppose moreover that $\sfH_\tface\sfr\leq 0$ and $\sqrt{-\sfH_\tface\sfr}\in\CI$ on $\Tsc^*_{\tface\cap\sface}\tface\cap\{p_\tface=0\}$ in the notation of~\eqref{EqPtfSymb}--\eqref{EqPtfSymbHam}.
  \begin{enumerate}
  \item\label{ItPNEstFw} If $N(P)$ is injective at weight $l$ on outgoing functions and $l-\frac{n}{2}\notin\Lambda$, then
    \begin{equation}
    \label{EqPNEstFw}
      \|u\|_{H_{\bop,\scop}^{s,l,\sfr}(\tface;\hat\mu)} \leq C\|N(P)u\|_{H_{\bop,\scop}^{s-2,l-2,\sfr+1}(\tface;\hat\mu)}
    \end{equation}
    for all $u$ for which both sides are finite.
  \item\label{ItPNEstBw} If $N(P)^*$ is injective at weight $-l+2$ on incoming functions and $-l+2-\frac{n}{2}\notin\Lambda$, then
    \begin{equation}
    \label{EqPNEstBw}
      \|v\|_{H_{\bop,\scop}^{-s+2,-l+2,-\sfr-1}(\tface;\hat\mu)} \leq C\|N(P)^*v\|_{H_{\bop,\scop}^{-s,-l,-\sfr}(\tface;\hat\mu)}
    \end{equation}
    for all $v$ for which both sides are finite.
  \item\label{ItPNEstInv} If $N(P)$ is invertible at weight $l$ and $l-\frac{n}{2}\notin\Lambda$,\footnote{This condition is automatically satisfied since for $l-\frac{n}{2}\notin\Lambda$, the operator $N(P)$ is not even Fredholm, cf.\ \cite[\S6.2]{MelroseAPS}.} then the operator $N(P)$ is invertible as a map
    \[
      \bigl\{ u\in H_{\bop,\scop}^{s,l,\sfr}(\tface;\hat\mu) \colon N(P)u\in H_{\bop,\scop}^{s-2,l-2,\sfr+1}(\tface;\hat\mu) \bigr\} \to H_{\bop,\scop}^{s-2,l-2,\sfr+1}(\tface;\hat\mu).
    \]
  \end{enumerate}
\end{lemma}
\begin{proof}
  This is a standard application of elliptic b-theory at $\tface\cap\cface$ and radial point estimates at $\tface\cap\sface$ in the scattering calculus as in~\cite{MelroseEuclideanSpectralTheory} and \cite[\S4.8]{VasyMinicourse}.
  
  We first prove symbolic estimates for $N(P)$ and $N(P)^*$ which do not use the injectivity assumptions. In $\tface\setminus\sface$, the operator $N(P)$ is an elliptic weighted b-differential operator. Let $\phi_j\in\CIc(\tface\setminus\sface)$, $j=0,1,2,3$, be identically $1$ near $\tface\cap\cface$, with $\phi_{j+1}\equiv 1$ on $\supp\phi_j$. Then, only recording the b-regularity and the weight at $\cface$, we have
  \begin{equation}
  \label{EqPNEst0}
    \|\phi_1 u\|_{\Hb^{s,l}} \leq C\bigl(\|\phi_2 N(P)u\|_{\Hb^{s-2,l-2}} + \|\phi_2 u\|_{\Hb^{-N,l}}\bigr)
  \end{equation}
  for any fixed $N$. Now, recalling~\eqref{EqPhat}, we have
  \[
    \Hb^{s,l}([0,\infty)_{\hat x}\times\pa X;\hat\mu)=\Hb^{s,l-\frac{n}{2}}\bigl([0,\infty)\times\pa X;\bigl|\tfrac{\dd\hat x}{\hat x}\dd k(0)\bigr|\bigr).
  \]
  Using now that $l-\frac{n}{2}\notin\Lambda$, we can estimate
  \[
    \|\phi_2 u\|_{\Hb^{-N,l}}\leq C\|\hat x^{-2}\hat N(P)(\hat x D_{\hat x},y,D_y)(\phi_2 u)\|_{\Hb^{-N-2,l-2}}
  \]
  by passing to the Mellin transform. Since $N(P)-\hat x^{-2}\hat N(P)(\hat x D_x,y,D_y)\in\hat x^{-1}\Diffb$ by~\eqref{EqPNormOpResc}, this can be plugged into~\eqref{EqPNEst0} and yields (putting back the scattering decay orders, which at this point are still arbitrary due to the localizers)
  \begin{equation}
  \label{EqPNEst1}
    \|\phi_1 u\|_{H_{\bop,\scop}^{s,l,\sfr}} \leq C\bigl(\|\phi_3 N(P)u\|_{H_{\bop,\scop}^{s-2,l-2,\sfr+1}} + \|\phi_3 u\|_{H_{\bop,\scop}^{-N,l-1,-N}}\bigr).
  \end{equation}

  Turning to the scattering end, and with $\psi_j=1-\phi_j$, we claim that (now with the b-decay orders being arbitrary)
  \begin{equation}
  \label{EqPNEst2}
    \|\psi_1 u\|_{H_{\bop,\scop}^{s,l,\sfr}} \leq C\bigl(\|\psi_0 N(P)u\|_{H_{\bop,\scop}^{s-2,l-2,\sfr+1}} + \|\psi_0 u\|_{H_{\bop,\scop}^{-N,-N,-N}}\bigr).
  \end{equation}
  This is proved by means of the scattering calculus by a combination of elliptic estimates (controlling $\psi_1 u$ away from $\Sigma_\tface:=p_\tface^{-1}(0)$), radial point estimates at ${}^\scop\cR_{\rm in/out}$, and microlocal real principal type estimates on $\Sigma_\tface\setminus({}^\scop\cR_{\rm in}\cup{}^\scop\cR_{\rm out})$. We only sketch the argument for the radial points in order to explain the emergence of the threshold condition on $\sfr$; details can be found e.g.\ in \cite[\S4.7]{VasyMinicourse}.
  
  We work in $[0,1)_{\hat h}\times\pa X\subset\tface$, and consider estimates near ${}^\scop\cR_{\rm in}$. Fixing a cutoff function $\chi\in\CIc([0,\half))$, identically $1$ near $0$ and with $\chi'\leq 0$ on $[0,\half)$, we consider a commutant
  \[
    a := \hat h^{-2\sfr-1}\chi(\hat h/\delta)\chi(|\eta_\scop|^2)\chi((\xi_\scop-1)^2),\quad
    A = \half(\Op_\scop(a)+\Op_\scop(a)^*) \in \Psisc^{-\infty,2\sfr+1},
  \]
  where $\delta>0$ controls the localization near $\hat h=0$. We compute the commutator
  \[
    2\Im\la N(P)u,A u\ra = \big\la \bigl(i[N(P),A]+2\tfrac{N(P)-N(P)*}{2 i}A\bigr)u,u\big\ra.
  \]
  (This holds directly for sufficiently decaying $u$, and for $u$ as in the statement of the Lemma can be justified using a regularization argument.) The principal symbol of $i[N(P),A]$ is equal to $\hat h\sfH_\tface a$. When $\sfH_\tface$ falls on the cutoff in $\hat h$, the result is supported in the elliptic set of $N(P)$, hence easily controlled. When $\sfH_\tface$ falls on either of the second or third cutoff functions, the result is $\leq 0$ on $\supp a$ in view of the source character of ${}^\scop\cR_{\rm in}$ (or directly using~\eqref{EqPtfSymbHam}), provided $\delta>0$ is sufficiently small; at ${}^\scop\cR_{\rm in}$ then, the principal symbol of $i[N(P),A]+2\tfrac{N(P)-N(P)^*}{2 i}A$ has a matching definite sign, i.e.\ is a \emph{negative} multiple of $\hat h^{-2\sfr}$, provided that
  \begin{equation}
  \label{EqPNRin}
    2\cdot(+1)\cdot(-2\sfr-1) + 2\cdot\sigmasc\bigl(\tfrac{Q-Q^*}{2 i}\bigr) < 0,\qquad
    Q := \hat h q_{1,1}(0,y,-\hat h D_{\hat h},D_y) + q_{0,1}(0,y),
  \end{equation}
  at ${}^\scop\cR_{\rm in}$. But ${}^\scop\cR_{\rm in}$ is the graph of the 1-form $\frac{\dd\hat h}{\hat h^2}$, hence $\sigmasc\bigl(\tfrac{Q-Q^*}{2 i}\bigr)=\Im\sigmasc(Q)$ at ${}^\scop\cR_{\rm in}$ is equal to
  \[
    \Im\sigmab(q_{1,1}(0,y,-\hat h D_{\hat h},D_y))\bigl(\tfrac{\dd\hat h}{\hat h}\bigr)+\Im q_{0,1} = r_1+r_0
  \]
  in the notation of Definition~\ref{DefPThr}. The condition~\eqref{EqPNRin} thus becomes $-2\sfr-1+(r_1+r_0)<0$, which is satisfied on all of ${}^\scop\cR_{\rm in}$ provided that $\sfr>-\half+\half\sup(r_1+r_0)=r_{\rm in}$ there. Under this assumption, one thus obtains control on $u$ microlocally near ${}^\scop\cR_{\rm in}$ in the space $\Hsc^{s,\sfr}$ by $N(P)u$ measured in $\Hsc^{s-2,\sfr+1}$.

  The analysis at ${}^\scop\cR_{\rm out}$ is similar, now using the commutant $\hat h^{-2\sfr-1}\chi(\hat h/\delta)\chi(|\eta_\scop|^2)\chi((\xi_\scop+1)^2)$. The derivatives of the latter two cutoffs along $\sfH_\tface$ are now positive due to the sink character of ${}^\scop\cR_{\rm out}$, and the principal symbol of the commutator at the radial set is a negative multiple of $\hat h^{-2\sfr}$ (thus allowing us to propagate control from a punctured neighborhood of the radial set into the radial set itself) provided that
  \begin{equation}
  \label{EqPNEstThr}
    2\cdot(-1)\cdot(-2\sfr-1) + 2\cdot\sigmasc\bigl(\tfrac{Q-Q^*}{2 i}\bigr) < 0\quad\text{at}\ {}^\scop\cR_{\rm out}.
  \end{equation}
  In view of ${}^\scop\cR_{\rm out}$ being the graph of $-\frac{\dd\hat h}{\hat h^2}$ and the calculation
  \[
    \Im\sigmasc(Q)|_{{}^\scop\cR_{\rm out}} = \Im\sigmab(q_{1,1}(0,y,-\hat h D_{\hat h},D_y))\bigl(\tfrac{-\dd\hat h}{\hat h}\bigr)+\Im q_{0,1} = -r_1+r_0,
  \]
  the condition~\eqref{EqPNEstThr} reads $2\sfr+1-r_1+r_0<0$, so $\sfr<-\half+\half\inf(r_1-r_0)=r_{\rm out}$.

  Putting~\eqref{EqPNEst1} and \eqref{EqPNEst2} together, we obtain the estimate
  \begin{equation}
  \label{EqPNEst3}
    \|u\|_{H_{\bop,\scop}^{s,l,\sfr}} \leq C\bigl(\|N(P)u\|_{H_{\bop,\scop}^{s-2,l-2,\sfr+1}} + \|u\|_{H_{\bop,\scop}^{-N,l-1,-N}}\bigr)
  \end{equation}
  for any $N$; we choose $N$ to satisfy $-N<s$ and $-N<\min\sfr$.
  
  The estimate~\eqref{EqPNEst3} implies that $N(P)$, acting on $H_{\bop,\scop}^{s,l,\sfr}$, has finite-dimensional kernel; any element $u$ in the kernel automatically lies in $H_{\bop,\scop}^{\infty,l,\sfr'}$ for any variable order function $\sfr'$ satisfying $\sfr'<r_{\rm out}$ at ${}^\scop\cR_{\rm out}$. Thus, $\WFsc(u)\subset{}^\scop\cR_{\rm out}$. Under the injectivity assumption on $N(P)$, we thus conclude that $u=0$. A standard functional analytic argument then allows one to drop the error term in~\eqref{EqPNEst3}, which gives the estimate~\eqref{EqPNEstFw}.

  The proof of part~\eqref{ItPNEstBw} is analogous; the direction of propagation in the characteristic set is now reversed, which is precisely matched by the sign switches in the orders in the estimate~\eqref{EqPNEstBw}. Part~\eqref{ItPNEstInv} is an immediate consequence of the first two parts.
\end{proof}

\begin{rmk}[Flexibility in the choice of $l$]
\label{RmkPNWeight}
  If the assumptions of part~\eqref{ItPNEstFw} of the Lemma are satisfied for some value of $l$, then they continue to hold for all values $\tilde l$ with $\tilde l-\frac{n}{2}\notin\Lambda$ for which either $\tilde l>l$, or $\tilde l\leq l$ but $\tilde l-\frac{n}{2}$ and $l-\frac{n}{2}$ lie in the same connected component $(a,b)$ of $\R\setminus\Lambda$. (Indeed, the claim for $\tilde l\leq l$ follows from the fact---proved using the Mellin transform upon localizing near $\tface\cap\cface$---that any element in the kernel of $N(P)$ on $H_{\bop,\scop}^{s,\tilde l,\sfr}$ automatically lies in $H_{\bop,\scop}^{s,b+\frac{n}{2}-\eps,\sfr}$ for any $\eps>0$.) A similar statement holds for part~\eqref{ItPNEstBw}: we may increase $-l+2$ (or stay in the same connected component of $(\R\setminus\Lambda)+\frac{n}{2}$), i.e.\ decrease $l$. Altogether then, there typically only exists an interval of finite length (possibly empty) of weights $l$ so that the invertibility condition of part~\eqref{ItPNEstInv} is satisfied.
\end{rmk}

\subsection{Statement and proof of the microlocal propagation estimate}
\label{SsPP}

We are now ready to state the main result of the paper:

\begin{thm}[Microlocal propagation through the cone point]
\label{ThmP}
  Let $P_{h,z}$ denote an admissible operator in the sense of Definition~\usref{DefPOp}, and define the threshold quantities $r_{\rm in},r_{\rm out}$ as in Definition~\usref{DefPThr}. Let $\Sigma\subset\Tch^*_\sface X_\chop$ denote the characteristic set of $P_{h,z}$ 9see~\eqref{EqPSymb}). Denote by $\sfH=\frac{x}{h}H_p\in\cV(\Tch^*_\sface X_\chop)$ the rescaled Hamilton vector field (see~\eqref{EqPSymbHam}). Let $s,l,\alpha\in\R$, $\sfb\in\CI(\ol{\Tch^*_\sface}X_\chop)$. Assume that $\sfb$ is constant near the radial sets $\cR_{\rm in/out}$ (see~\eqref{EqPRadSets}) and satisfies $\sfb-\alpha>r_{\rm in}$ at $\cR_{\rm in}$ and $\sfb-\alpha<r_{\rm out}$ at $\cR_{\rm out}$; assume moreover that $\sfH\sfb\leq 0$ and $\sqrt{-\sfH\sfb}\in\CI$ on $\Sigma$. Let $\chi,\tilde\chi\in\CIc(X)$ be cutoffs, identically $1$ near $\pa X$, and with $\tilde\chi\equiv 1$ on $\supp\chi$. Let $E\in\Psich^0(X)$, with Schwartz kernel supported in $[0,1)_h\times(\tilde\chi^{-1}(1)\times\tilde\chi^{-1}(1))$.
  \begin{enumerate}
  \item\label{ItPThmFw}{\rm (Forward propagation.)} Suppose $N(P)$ is injective at weight $l$ on outgoing functions (see Definition~\usref{DefPNormOp}\eqref{ItPNormOpFw}). Suppose that (the preimage in $\Sigma$ of) all backward GBBs (see Definition~\usref{DefPBroken}) starting in $\Sigma\cap\supp\chi$ reach $\Ellch(E)$ in finite time while remaining inside $\tilde\chi^{-1}(1)$. Then for some small $\delta>0$, we have
    \begin{equation}
    \label{EqPThmFw}
    \begin{split}
      \| \chi u \|_{\Hch^{s,l,\alpha,\sfb}(X)} &\leq C\Bigl(\|\tilde\chi P_{h,z}u\|_{\Hch^{s-2,l-2,\alpha,\sfb+1}(X)} \\
        &\quad\qquad + \|E u\|_{\Hch^{s,l,\alpha,\sfb}(X)} + h^\delta\| \tilde\chi u \|_{\Hch^{-N,l,\alpha,\sfb}(X)} \Bigr).
    \end{split}
    \end{equation}
  \item\label{ItPThmBw}{\rm (Backward propagation.)} Suppose $N(P)^*$ is injective at weight $-l+2$ on incoming functions, see Definition~\usref{DefPNormOp}\eqref{ItPNormOpBw}. Suppose that (the preimage in $\Sigma$ of) all forward GBBs starting in $\Sigma\cap\supp\chi$ reach $\Ellch(E)$ in finite time while remaining inside $\tilde\chi^{-1}(1)$. Then for some small $\delta>0$, we have
    \begin{equation}
    \label{EqPThmBw}
    \begin{split}
      &\| \chi u \|_{\Hch^{-s+2,-l+2,-\alpha,-\sfb-1}(X)} \\
      &\quad \leq C\Bigl(\|\tilde\chi P_{h,z}^*u\|_{\Hch^{-s,-l,-\alpha,-\sfb}(X)} \\
      &\quad\quad\qquad + \|E u\|_{\Hch^{-s+2,-l+2,-\alpha,-\sfb-1}(X)} + h^\delta\| \tilde\chi u \|_{\Hch^{-N,-l+2,-\alpha,-\sfb-1}(X)} \Bigr).
    \end{split}
    \end{equation}
  \end{enumerate}
\end{thm}

Since by Lemma~\ref{LemmaPStruct} and the calculations in~\S\ref{SsPChar}, the operator $P_{h,z}\in\Psich^{2,2,0,0}(X)$ is elliptic at fiber infinity, and is of real principal type (except at the radial points) at $\sface$, the estimates~\eqref{EqPThmFw} and \eqref{EqPThmBw} are sharp as far as the relative orders in the norms on $u$ on the left and $P_{h,z}^{(*)}u$ on the right are concerned. Indeed, it has the well-known real principal type loss of one order at $\sface$ and is an elliptic estimate in the $\chop$-differentiability sense.

The improvement of the final (error) terms on the right hand sides in~\eqref{EqPThmFw} and~\eqref{EqPThmBw} relative to the space on the left hand sides is accomplished at $\sface$ by microlocal symbolic means, and at $\tface$ using global normal operator estimates. The overall improvement by a positive power of $h$ between error term and left hand side allows for the inversion of $P_{h,z}$ for small $h>0$ under suitable assumptions on the \emph{global} behavior of the null-bicharacteristic flow; see~\S\S\ref{SsPA} and \ref{SH} for examples.

\begin{rmk}[Operators on vector bundles]
\label{RmkPBundle}
  Let $E\to X$ denote a smooth vector bundle. Theorem~\ref{ThmP} then holds (with the same proof) also for operators $P_{h,z}$ acting on sections of $E$, provided $P_{h,z}$ is \emph{admissible} in the sense that
  \[
    P_{h,z} = h^2 x^{-2}Q_{2,z} + h x^{-1}q_{0,z} - z,\qquad Q_{2,z}\in\Diffb^2(X;E),\quad q_{0,z}\in\CI(X;\End(E)),
  \]
  where $x^{-2}Q_{2,z}$ (replacing the combination $h^2\Delta_g+h^2 x^{-2}Q_{1,z}$ in Definition~\ref{DefPOp}) has scalar principal symbol $\sigmab(x^{-2}Q_{2,z})=\sigmab(x^{-2}\Delta_g)$. That is, $\sigmab(x^{-2}Q_{2,z})(\zeta)=|\zeta|_{g^{-1}}^2$ for $\zeta\in\Tb^*X$, with $g$ the conic metric~\eqref{EqPProdMet}. The normal operator is of class
  \[
    N(P)\in(\tfrac{\hat x}{\hat x+1})^{-2}\Diff_{\bop,\scop}^2(\tface;\pi^*E_{\pa X}),
  \]
  where $\pi\colon\tface=[0,\infty]_{\hat x}\times\pa X\to\pa X$ denotes the projection map. The injectivity conditions of Definition~\ref{DefPNormOp} are unchanged. The definition of the threshold quantities $r_{\rm in/out}$ in Definition~\ref{DefPThr} requires a minor change; to wit, with respect to a choice of a positive definite fiber inner product on $E_{\pa X}$, we set (top sign for `$\rm in$', bottom sign for `$\rm out$')
  \begin{align*}
    r_{\rm in/out} &:= -\half \pm \half\sup_{\pa X} \sigmasc\Bigl(\frac{x^{-2}Q_{2,1}-(x^{-2}Q_{2,1})^* + q_{0,1}-q_{0,1}^*}{2 i}\Bigr)\Big|_{\mp\frac{\dd x}{x^2}},
  \end{align*}
  where the $\sup$ is defined to be the supremum of the largest eigenvalue of the scattering symbol (which takes values in self-adjoint endomorphisms of $E$). One may choose different fiber inner products in the calculation of $r_{\rm in}$ and $r_{\rm out}$, respectively. A (near-)optimal choice of fiber inner products, resulting in (almost) the smallest possible $r_{\rm in}$ and largest possible $r_{\rm out}$, is typically easy to read off in concrete situations.
\end{rmk}

\begin{rmk}[Technical assumptions on the variable order]
\label{RmkPVar}
  One can replace the assumptions that $\sfb$ be locally constant near $\cR_{\rm in/out}$ and satisfy $\sqrt{-\sfH\sfb}\in\CI$ on $\Sigma$ by the simpler assumption that $\sfH\sfb\leq 0$ on $\Sigma$. This would require the use of the sharp G\aa{}rding inequality for the $\chop$-calculus, which however we do not prove here.
\end{rmk}

\begin{proof}[Proof of Theorem~\usref{ThmP}]
  We give details for the proof of part~\eqref{ItPThmFw}; the proof of part~\eqref{ItPThmBw} is completely analogous. If backward GBBs starting in $\WFch'(B)$ never pass through $\pa\Sigma\subset\Tch^*_{\sface\cap\tface}X_\chop$, the orders $l$ and $a$ are irrelevant, and the estimate~\eqref{EqPThmFw} follows from standard elliptic regularity and real principal type propagation in the (variable order) semiclassical calculus on $X^\circ$. We shall thus work in a small neighborhood of $x=0$. 

  \pfstep{Step 1: symbolic positive commutator estimate.} We first work near the incoming radial set $\cR_{\rm in}$ defined in~\eqref{EqPRadSets}; we shall use the coordinates $(\hat h,x,y,\xi,\eta)$ near $\Tch^*_{\sface\cap\tface}X_\chop$ defined by \eqref{EqCVCoordTilde} and \eqref{EqPhat}. Fix cutoffs $\chi_\pa,\chi_\sface,\chi_\cR\in\CIc([0,1))$, identically $1$ near $0$ and satisfying $\chi_\bullet'\leq 0$ and $\sqrt{-\chi_\bullet\chi_\bullet'}\in\CI([0,1))$. Denote a smooth extension of $\sfb$ to $\ol{\Tch^*}X_\chop$ by the same symbol. For small $\delta>0$, fixed momentarily, we then consider a commutant
  \begin{equation}
  \label{EqPThmComm}
    \check a = \hat h^{-\sfb-\frac12}x^{-\alpha}\chi_\pa\bigl(\tfrac{x}{\delta}\bigr)\chi_\sface\bigl(\tfrac{\hat h}{\delta}\bigr)\chi_\cR\bigl(\tfrac{\omega}{\delta}\bigr),\qquad\omega := \sqrt{|\eta|^2+|\xi+1|^2}.
  \end{equation}
  Thus, $\supp a$ is contained in any fixed open neighborhood of $\cR_{\rm in}$ when $\delta>0$ is sufficiently small. We have $\check a\in S^{-\infty,-\infty,\alpha,\sfb+\frac12}(\ol{\Tch^*}X_\chop)$. Let
  \[
    \check A \in \Psich^{-\infty,-\infty,\alpha,\sfb+\frac12}(X),\quad \sigmach(\check A)=\check a,\qquad
    A := \check A^*\check A.
  \]
  Using the $L^2(X;\mu)$ inner product, we then evaluate the commutator
  \begin{equation}
  \label{EqPThmPC}
    2\Im\la \check A P_{h,z}u,\check A u\ra = \la \sC u,u\ra,\qquad \sC=i[P_{h,z},A]+2\tfrac{P_{h,z}-P_{h,z}^*}{2 i}A \in \Psich^{-\infty,-\infty,2 a,2\sfb}(X).
  \end{equation}
  The principal symbol of $\sC$ is
  \begin{equation}
  \label{EqPThmPCSymb}
    2\hat h\check a\sfH\check a + 2\cdot\sigmach\Bigl(\tfrac{P_{h,z}-P_{h,z}^*}{2 i\hat h}\Bigr)\hat h\check a^2.
  \end{equation}
  When $\sfH$ hits $\chi_\pa$, we obtain a nonnegative contribution (in fact, the square $e^2$ of a smooth function $e$), while differentiation of $\chi_\cR$ gives a nonpositive contribution (in fact, a negative squares $-b_\cR^2$), consistently with the saddle point structure of $\sfH$ at $\cR_{\rm in}$. Differentiation of $\chi_\sface$ produces a symbol with semiclassical order $-\infty$.
  
  The main term of $\hat h\check a \sfH\check a$ near $\cR_{\rm in}$ arises from differentiation of the weight $\hat h^{-\sfb-\frac12}x^{-\alpha}$; since $H_p=\hat h\sfH$ is, modulo $\hat h x\Vb$, given by the expression~\eqref{EqPHamResc}, we can compute this modulo $x S^{-\infty,-\infty,2 a,2\sfb}$ by substituting the expression~\eqref{EqPHamResc} of $\sfH|_{x=0}$ for $\sfH$. Thus, the main term is
  \[
    \hat h^{-2\sfb}x^{-2\alpha}\bigl(2\xi\bigl(-2\alpha+2\sfb+1\bigr)+\cO(x)\bigr)\chi_\pa^2\chi_\sface^2\chi_\cR^2.
  \]
  A further contribution arises from the skew-adjoint part of $P_{h,z}$ at $\cR_{\rm in}$, which is the same as the skew-adjoint part of $N(P)$ at ${}^\scop\cR_{\rm in}$ upon making the identification~\eqref{EqCPNBundles}; this was already computed in the proof of Lemma~\ref{LemmaPNEst}. Overall then, we can write
  \begin{equation}
  \label{EqPThmPCSymbSq}
    \sigmach(\sC) = e^2 - b_\cR^2 - \eps\hat h\check a^2 - f^2\hat h\check a^2,
  \end{equation}
  where $f=\sqrt{-\bigl[2(-2(\sfb-\alpha)-1)+2(r_1+r_0)\bigr]-\eps}$ is positive (and smooth) at $\cR_{\rm in}$ for small $\eps>0$. Denoting $\chop$-quantizations of the lower case symbols by the corresponding upper case letters, we thus have
  \[
    \sC = E^*E - B_\cR^*B_\cR - \eps\|\hat h^{\frac12}\check A u\|^2 - \|\hat h^{\frac12}F\check A u\|^2 + R,
  \]
  where $R\in\Psich^{-\infty,-\infty,2\alpha,2\sfb-1}(X)$ has $\WFch'(R)\subset\supp\check a$ and arises as the remainder term not controlled by the previous symbolic considerations. We will plug this into the right hand side of~\eqref{EqPThmPC}; the left hand side is bounded from below by
  \begin{align*}
    &-\eps\|\hat h^{\frac12}\check A u\|^2 - \eps^{-1}\|\hat h^{-\frac12}\check A P_{h,z}u\|^2 \\
    &\qquad \geq -\eps\|\hat h^{\frac12}\check A u\|^2 - C\eps^{-1}\|G P_{h,z}u\|_{\Hch^{-N,-N,\alpha,\sfb+1}} - C\eps^{-1}\|\tilde\chi u\|_{\Hch^{-N,-N,\alpha,-N}}^2,
  \end{align*}
  where $G\in\Psich^0(X)$ is elliptic on $\supp\check a$; here $N\in\R$ is arbitrary. Putting $B_0:=\hat h^{\frac12}F\check A\in\Psich^{-\infty,-\infty,\alpha,\sfb}(X)$ and dropping the contribution of $B_\cR$, we thus obtain the estimate
  \begin{equation}
  \label{EqPThmRin}
    \|B_0 u\|^2 \lesssim \| G P_{h,z}u \|_{\Hch^{-N,-N,\alpha,\sfb+1}}^2 + \|E u\|^2 + |\la R u,u\ra| + \|\tilde\chi u\|_{\Hch^{-N,-N,\alpha,-N}}^2,
  \end{equation}
  which provides $\Hch^{-N,-N,\alpha,\sfb}$-control of $u$ microlocally near $\cR_{\rm in}$ provided one has microlocal $\Hch^{-N,-N,-N,\sfb}$-control of $u$ on $\WFch'(E)\subset\{0<x<\delta,\ |\xi+1|<\delta\}$, and provided $|\la R u,u\ra|$ is finite; since $G$ is elliptic near $\WFch'(R)$, we can insert the estimate
  \[
    |\la R u,u\ra|\lesssim\|G u\|_{\Hch^{-N,-N,\alpha,\sfb-\frac12}}^2+\|\tilde\chi u\|_{\Hch^{-N,-N,\alpha,-N}}^2
  \]
  into the right hand side of~\eqref{EqPThmRin}.

  Concatenating this radial point estimate with the propagation of regularity from a punctured neighborhood of $\cR_{\rm in}$ to a punctured neighborhood of $\cR_{\rm out}$ and then a radial point estimate at $\cR_{\rm out}$---proved by the same method, with the commutant $\check a$ again given by~\eqref{EqPThmComm} but now with $\omega=(|\eta|^2+|\xi-1|^2)^{1/2}$ and using that $2(\sfb-\alpha)+1-r_1+r_0<0$---we obtain the propagation estimate
  \begin{equation}
  \label{EqPThmSymb}
    \|B_0 u\|_{\Hch^{s,l,\alpha,\sfb}} \lesssim \|G P_{h,z}u\|_{\Hch^{s-2,l-2,\alpha,\sfb+1}} + \|E u\|_{\Hch^{s,-N,-N,\sfb}} + \|\tilde\chi u\|_{\Hch^{-N,l,\alpha,\sfb-\frac12}}.
  \end{equation}
  The orders at $\cface$ are unconstrained at this point, but chosen for compatibility with the normal operator argument below.
  
  Away from $\Sigma$, we have elliptic estimates. Fixing a variable order function $\sfb'$ so that
  \begin{equation}
  \label{EqPThmbp}
    \sfb'-\alpha>r_{\rm in}\quad\text{at $\cR_{\rm in}$},
  \end{equation}
  we then have the estimate
  \begin{equation}
  \label{EqPThmSymb2}
    \|\chi u\|_{\Hch^{s,l,\alpha,\sfb}} \lesssim \|\tilde\chi P_{h,z}u\|_{\Hch^{s-2,l-2,\alpha,\sfb+1}} + \|E u\|_{\Hch^{s,-N,-N,\sfb}} + \|\tilde\chi u\|_{\Hch^{-N,l,\alpha,\sfb'}}
  \end{equation}
  under the assumptions on $E,\chi,\tilde\chi$ stated in the Theorem. For $\sfb'\geq\sfb-\half$, this follows directly from~\eqref{EqPThmSymb} (upon replacing the microlocal cutoff $G$ by the less precise cutoff $\tilde\chi$). For general $\sfb'$, note that as long as~\eqref{EqPThmbp} is satisfied, one can apply this estimate inductively to the error term $\tilde\chi u$ provided $\supp\tilde\chi$ is sufficiently close to $\supp\chi$ (so that the same operator $E u$ satisfies the geometric control assumption for $\tilde\chi$ in place of $\chi$), increasing supports of the involved cutoff functions by an arbitrarily small but positive amount and gaining half a semiclassical order at each step. Thus, away from $\cR_{\rm in}$, one can ultimately take $\sfb'$ to be arbitrarily negative, while at $\cR_{\rm in}$, one always needs to have~\eqref{EqPThmbp}.

  \pfstep{Step 2: normal operator estimate.} We now work on the error term $\tilde\chi u$ in~\eqref{EqPThmSymb2}.
  
  We first prove the desired estimate~\eqref{EqPThmFw} under the stronger condition that $\sfb-\alpha>r_{\rm in}+1$ at $\cR_{\rm in}$. We split $\tilde\chi u=\chi^\flat u+(1-\chi^\flat)\chi\tilde u$, where $\chi^\flat\in\CI(X)$ is identically $1$ near $\pa X$ and supported in a very small neighborhood of $\pa X$; the part $(1-\chi^\flat)\tilde\chi u$ is supported away from $\cface\cup\tface$, hence $\|(1-\chi^\flat)\tilde\chi u\|_{\Hch^{-N,l,\alpha,\sfb^\flat}}\lesssim\|(1-\chi^\flat)\tilde\chi u\|_{\Hch^{-N,-N,-N,\sfb^\flat}}$ for any $N$. To estimate $\chi^\flat u$, we use the injectivity assumption on $N(P)$ and the resulting estimate~\eqref{EqPNEstFw} together with Corollary~\ref{CorCPNSob}\eqref{ItCPNSobVar} (with $\alpha_\mu=n$). For $0<\delta<1$ with $\sfb^\flat+2\delta<\sfb$, and choosing $\supp\chi^\flat$ sufficiently small, we obtain
  \begin{align}
  \label{EqPThmNorm0}
    \|\chi^\flat u\|_{\Hch^{-N,l,\alpha,\sfb^\flat}} &\lesssim h^{\frac{n}{2}-\alpha}\|\pi^*(\chi^\flat u)\|_{H_{\bop,\scop}^{-N,l,\sfb^\flat|_{\sface\cap\tface}-\alpha+\delta}} \\
      &\lesssim h^{\frac{n}{2}-\alpha}\|N(P)(\pi^*(\chi^\flat u))\|_{H_{\bop,\scop}^{-N-2,l-2,\sfb^\flat|_{\sface\cap\tface}-\alpha+\delta+1}} \nonumber\\
      &\lesssim \|N(P)(\chi^\flat u)\|_{\Hch^{-N-2,l-2,\alpha,\sfb^\flat+1+2\delta}}. \nonumber
  \end{align}
  Put $\sfb^\sharp:=\sfb^\flat+1+2\delta$. Using Lemma~\ref{LemmaPStruct}, which gives $N(P)-P_{h,z}\in\Psich^{2,2,-1,0}(X)$, we further estimate
  \begin{align}
    &\|N(P)(\chi^\flat u)\|_{\Hch^{-N-2,l-2,\alpha,\sfb^\sharp}} \nonumber\\
  \label{EqPThmNormInter}
  \begin{split}
    &\quad \leq \|\chi^\flat P_{h,z}u\|_{\Hch^{-N-2,l-2,\alpha,\sfb^\sharp}} + \|\chi^\flat(P_{h,z}-N(P))u\|_{\Hch^{-N-2,l-2,\alpha,\sfb^\sharp}} \\
    &\quad\qquad + \|[N(P),\chi^\flat]u\|_{\Hch^{-N-2,l-2,\alpha,\sfb^\sharp}}
  \end{split} \\
  \label{EqPThmNorm}
    &\quad \lesssim \|\chi^\flat P_{h,z}u\|_{\Hch^{-N-2,l-2,\alpha,\sfb^\sharp}} + \|\chi^\flat u\|_{\Hch^{-N,l,\alpha-1,\sfb^\sharp}} + \|\chi^\sharp u\|_{\Hch^{-N-1,-N,-N,\sfb^\sharp-1}},
  \end{align}
  where $\chi^\sharp\equiv 1$ on $\supp\chi^\flat$. Under the present condition that $\sfb-\alpha>r_{\rm in}+1$ at $\cR_{\rm in}$, we can choose $\sfb^\flat$ as in~\eqref{EqPThmbp} so that $\sfb^\sharp<\sfb$ still. Plugging this into~\eqref{EqPThmSymb} finishes the proof of part~\eqref{ItPThmFw} under this condition.

  In order to prove the Theorem as stated, thus only assuming $\sfb-\alpha>r_{\rm in}$ at $\cR_{\rm in}$, we note that the norm on second term on the right in~\eqref{EqPThmNorm} is $1$ order weaker at $\tface$ than the left hand side of~\eqref{EqPThmNorm0}, but only $\sfb^\sharp-\sfb=\sfb^\flat+1+2\delta-\sfb<1$ orders stronger at $\sface$. This suggests revisiting the estimates~\eqref{EqPThmNorm0}--\eqref{EqPThmNorm} using a more precise cutoff which distinguishes between the regimes $\hat h\lesssim x$ and $\hat h\gtrsim x$. To wit, consider $\chi^\flat=\psi^\flat(\frac{\hat h}{x})$, where $\psi^\flat\equiv 0$ on $[0,1]$ and $\psi^\flat\equiv 1$ on $[2,\infty)$. This is a smooth function on $[X_\chop;\sface\cap\tface]$, and thus conormal on $X_\chop$; in fact, we have
  \begin{equation}
  \label{EqPThmChiFlat}
    1-\chi^\flat \in \cA^{0,\zeta,-\zeta}(X_\chop) = (x+h)^\zeta\bigl(\tfrac{h}{h+x}\bigr)^{-\zeta}\cA^0(X_\chop) \subset \Psich^{0,0,-\zeta,\zeta}(X)
  \end{equation}
  for any $\zeta\geq 0$, since on $\supp(1-\chi^\flat)$ we have $x\lesssim\hat h$, thus $x+h\lesssim\frac{h}{h+x}$. Taking $\zeta=\delta$, we can therefore estimate
  \[
    \|(1-\chi^\flat)\tilde\chi u\|_{\Hch^{-N,l,\alpha,\sfb^\flat}} \lesssim \|\tilde\chi u\|_{\Hch^{-N,l,\alpha-\delta,\sfb^\flat+\delta}} \leq h^\delta\|\tilde\chi u\|_{\Hch^{-N,l,\alpha,\sfb}}
  \]
  Next, the estimate~\eqref{EqPThmNorm0} holds without change. (Note that Corollary~\ref{CorCPNSob} applies for merely conormal cutoffs.) Finally, we need to estimate~\eqref{EqPThmNormInter} more carefully. Note that
  \[
    \chi^\flat \in \cA^{0,-\zeta,\zeta}(X_\chop) \subset \Psich^{0,0,\zeta,-\zeta}(X)
  \]
  for any $\zeta\geq 0$. Taking $\zeta=1-\delta$, this gives $\chi^\flat(P_{h,z}-N(P))\in\Psich^{2,2,-\delta,-1+\delta}(X)$, hence
  \[
    \|\chi^\flat(P_{h,z}-N(P))u\|_{\Hch^{-N-2,l-2,\alpha,\sfb^\sharp}} \lesssim \|\tilde\chi u\|_{\Hch^{-N,l,\alpha-\delta,\sfb^\sharp-1+\delta}} \leq h^\delta\|\tilde\chi u\|_{\Hch^{-N,l,\alpha,\sfb}}.
  \]
  For the final, commutator, term in~\eqref{EqPThmNormInter}, we note that we can replace $\chi^\flat$ by $1-\chi^\flat$ and use~\eqref{EqPThmChiFlat} with $\zeta=\delta$, so $[N(P),\chi^\flat]\in\Psich^{1,-\infty,-\delta,-1+\delta}(X)$, which gives
  \[
    \|[N(P),\chi^\flat]u\|_{\Hch^{-N-2,l-2,\alpha,\sfb^\sharp}} \lesssim \|\tilde\chi u\|_{\Hch^{-N-1,-N,\alpha-\delta,\sfb^\sharp-1+\delta}} \leq h^\delta\|\tilde\chi u\|_{\Hch^{-N-1,-N,\alpha,\sfb}}.
  \]
  Altogether, we have shown
  \[
    \|\tilde\chi u\|_{\Hch^{-N,l,\alpha,\sfb^\flat}} \lesssim \|\tilde\chi P_{h,z}u\|_{\Hch^{-N-2,l-2,\alpha,\sfb^\sharp}} + h^\delta\|\tilde\chi u\|_{\Hch^{-N,l,\alpha,\sfb}}.
  \]
  Plugged into~\eqref{EqPThmSymb2}, we have now established the desired estimate~\eqref{EqPThmFw}. This finishes the proof of the Theorem.
\end{proof}

We can sharpen Theorem~\ref{ThmP} by working with the resolved Sobolev spaces defined in~\eqref{EqCbSob}. This is straightforward since admissible operators
\[
  P_{h,z} \in \Psich^{2,2,0,0}(X) \subset \Psi_{\cop\bop\semi}^{2,2,2,0,0}(X)
\]
are elliptic at the front face $\fbface$ of $\ol{{}^{\cop\bop\semi}T^*}X_\chop$; indeed, this follows from the ellipticity at fiber infinity $\Sch^*_\tface X_\chop\subset\ol{\Tch^*}X_\chop$ and the classical nature of the principal symbol of $P_{h,z}$. Therefore:

\begin{thm}[Propagation estimates with relative b-regularity]
\label{ThmCb}
  In the notation of Theorem~\usref{ThmP}, and for any $s'\in\R$, the forward propagation estimate~\eqref{EqPThmFw} generalizes to an estimate on $\cop\bop\semi$-Sobolev spaces,
    \begin{align*}
      \| \chi u \|_{H_{\cop\bop,h}^{s,s',l,\alpha,\sfb}(X)} &\leq C\bigl(\|\tilde\chi P_{h,z}u\|_{H_{\cop\bop,h}^{s-2,s'-2,l-2,\alpha,\sfb+1}(X)} \\
        &\quad\qquad + \|E u\|_{H_{\cop\bop,h}^{s,s',l,\alpha,\sfb}(X)} + h^\delta\| \tilde\chi u \|_{H_{\cop\bop,h}^{-N,-N,l,\alpha,\sfb}(X)} \bigr).
    \end{align*}
    The backward propagation estimate~\eqref{EqPThmBw} generalizes similarly.
\end{thm}

\subsection{Global estimates with complex absorption}
\label{SsPA}

We upgrade the microlocal estimate proved above into a quantitative invertibility statement for an operator which effectively localizes the interesting nonelliptic phenomena near the cone point into a small neighborhood of $\pa X$ via complex absorption.

Concretely, with $X=[0,2 x_0)\times Y$ and $g$ as in~\eqref{EqPMfd}--\eqref{EqPProdMet}, consider a compact $n$-dimensional manifold $X'\supset X$ with boundary $\pa X'=\pa X$, equipped with a smooth metric $g'$ which is equal to $g$ on $X^\flat:=[0,x_0]\times Y$. Given an admissible operator $P_{h,z}$ on $X$, let $P'_{h,z}\in(\frac{x}{x+h})^{-2}\Diffch^2(X'_\chop)$ denote an extension of $P_{h,z}$ from $[0,1)\times X^\flat$ to $[0,1)\times X'$ with principal part equal to $h^2\Delta_{g'}$. For $c\in(0,1)$, denote
\[
  K_c:=X'\setminus([0,c x_0]\times Y).
\]
In order to implement complex absorption, let us take $c\in(0,\half)$ small and fix an operator
\[
  Q \in \Psi_h^{-\infty}((X')^\circ)
\]
whose Schwartz kernel is supported in $K_c\times K_c$, and so that $Q$ is elliptic on $T^*K_{2 c}$ with nonnegative principal symbol. We then consider
\begin{equation}
\label{EqPAcP}
  \cP_{h,z} := P'_{h,z} - i Q,
\end{equation}
and assume that
\begin{equation}
\label{EqPAcPCtrl}
  \text{all backward GBBs of $P'_{h,z}$ enter $\Ellh(Q)$ in finite time.}
\end{equation}

By construction, $\cP_{h,z}$ is a semiclassically elliptic second order semiclassical ps.d.o.\ on $(X')^\circ$ which is elliptic over $K_{2 c}$. Moreover, due to the sign condition on the principal symbol of $Q$, one can propagate semiclassical regularity for solutions of $\cP_{h,z}u=f$ along \emph{forward} null-bicharacteristics of $P'_{h,z}$, see \cite[\S2.5]{VasyMicroKerrdS} and \cite[\S5.6.3]{DyatlovZworskiBook}. For our fixed metric $g$ on $[0,2 x_0)\times Y$, the control condition~\eqref{EqPAcPCtrl} is satisfied if we choose $c>0$ sufficiently small. Indeed, from the expression~\eqref{EqPHsf}, one finds that if $\sfH_\sface x=2\xi=0$ on the characteristic set, then $|\eta|^2=1$ and thus $\sfH_\sface^2 x=2\sfH\xi\geq 2 x^{-1}|\eta|^2-C|\eta|^2=2 x^{-1}-C>c^{-1}-C>0$ in $x<2 c$ when $c$ is sufficiently small; hence the level sets of $x$ are geodesically convex in $x<2 c$, which implies the claim.

\begin{rmk}[Relaxed conditions on $Q$]
\label{RmkPAQdelta}
  One can more generally allow $Q$ to be a second order operator with real principal symbol; a concrete choice is then $Q=\psi\cdot(h^2\Delta_{g'}+1)$ where $\psi\in\CIc(K_c^\circ)$ is identically $1$ on $K_{2 c}$.
\end{rmk}

We then have:
\begin{prop}[Global estimates with complex absorption]
\label{PropPA}
  Let $s,l,\alpha,\sfb$ be as in the statement of Theorem~\usref{ThmP} (for the operator $P_{h,z}$). Fix the volume density on $X'$ to be the metric density $|\dd g'|$. Then for small $h>0$, the operator $\cP_{h,z}$ defined by~\eqref{EqPAcP} is invertible as a map $\Hb^{s,l}(X')\to\Hb^{s-2,l-2}(X')$, and it satisfies the uniform estimate
  \begin{equation}
  \label{EqPAEst}
    \|u\|_{\Hch^{s,l,\alpha,\sfb}(X')} \leq C\|\cP_{h,z}u\|_{\Hch^{s-2,l-2,\alpha,\sfb+1}(X')} = C h^{-1}\|(x+h)\cP_{h,z}u\|_{\Hch^{s-2,l-2,\alpha,\sfb}}.
  \end{equation}
  More generally, for any $s'\in\R$, we have $\|u\|_{H_{\cop\bop,h}^{s,s',l,\alpha,\sfb}(X')} \leq C\|\cP_{h,z}u\|_{H_{\cop\bop,h}^{s-2,s'-2,l-2,\alpha,\sfb+1}}$.
\end{prop}
\begin{proof}
  By our assumptions on the complex absorbing potential $Q$, we can apply Theorem~\ref{ThmP}\eqref{ItPThmFw} with $E$ and $\chi$ supported in $X'\setminus K_c$. We thus have
  \[
    \|\chi u\|_{\Hch^{s,l,\alpha,\sfb}} \leq C\bigl(\|\cP_{h,z}u\|_{\Hch^{s-2,l-2,\alpha,\sfb+1}} + \| E u \|_{\Hch^{s,-N,-N,\sfb}} + h^\delta\|\tilde\chi u\|_{\Hch^{-N,l,\alpha-1,\sfb^\flat}}\bigr).
  \]
  On the other hand, we can control $E u$ and $(1-\chi)u$ in $\Hch^{s,l,\alpha,\sfb}$ (or simply $h^\sfb H_h^s$ if we take $E$ to be localized away from $x=0$, as we may arrange) by $\cP_{h,z}u$ in $\Hch^{s-2,l-2,\alpha,\sfb+1}$ using a combination of elliptic estimates and real principal type propagation estimates (with complex absorption), starting either from $\Ellh(Q)$ or $\{\chi=1\}$. Altogether, we obtain
  \begin{equation}
  \label{EqPAEstPf}
    \|u\|_{\Hch^{s,l,\alpha,\sfb}} \leq C\bigl(\|\cP_{h,z}u\|_{\Hch^{s-2,l-2,\alpha,\sfb+1}} + h^\delta\|u\|_{\Hch^{-N,l,\alpha,\sfb}}\bigr).
  \end{equation}
  For $h_0>0$ with $C h_0^\delta<\half$, we can now drop the error term in~\eqref{EqPAEstPf} for $0<h<h_0$. This proves the injectivity of $\cP_{h,z}$ (with a quantitative estimate). Analogous arguments prove the dual estimate
  \[
    \|v\|_{\Hch^{-s+2,-l+2,-\alpha,-\sfb-1}} \leq C\|\cP_{h,z}^*v\|_{\Hch^{-s,-l,-\alpha,-\sfb}},
  \]
  which implies the surjectivity of $\cP_{h,z}$. The proof is complete.
\end{proof}

\subsection{Propagation of Lagrangian regularity; diffractive improvement}
\label{SsPD}

By adapting arguments from \cite{MelroseWunschConic,MelroseVasyWunschEdge}, we improve upon Theorem~\ref{ThmP} by demonstrating that, under a non-focusing condition, strong singularities can only propagate along geometric GBBs. The key technical result concerns the propagation of Lagrangian regularity with respect to the incoming and outgoing Lagrangian submanifolds, localized near geometric continuations of a GBB striking the cone point. Using the coordinates $(\hat h,x,y,\xi,\eta)$ and the notation of~\eqref{EqPInOutCurves}, the incoming and outgoing Lagrangians are given by
\[
  \cF_\bullet := \bigcup_{y_0\in\pa X} \cF_{\bullet,y_0}, \qquad \bullet=I,O,
\]
where
\begin{equation}
\label{EqPDFInOut}
\begin{alignedat}{2}
  \cF_{I,y_0} &:= \gamma_{I,y_0}((-x_0,0)) &&= \{ (0,x,y_0,-1,0) \colon x<2 x_0 \}, \\
  \cF_{O,y_0} &:= \gamma_{O,y_0}((-x_0,0)) &&= \{ (0,x,y_0,1,0) \colon x<2 x_0 \}.
\end{alignedat}
\end{equation}
(We are making the $\hat h$-coordinate, which was set to $0$ in~\eqref{EqPInOutCurves}, explicit here.)

We shall first show that one can control the Lagrangian regularity of a solution $u$ of $P_{h,z}u=f$, with sufficiently regular forcing $f$, near $\cF_{O,y_0}$ by propagating Lagrangian regularity from the union of all $\cF_{I,y'}^\circ$, with $y'$ at distance $\pi$ from $y_0$, into $\pa X$ and then within $\pa X$ to $\cF_{O,y_0}\cap x^{-1}(0)$. Localization within the radial sets $\cR_{\rm in/out}$ requires a more careful choice of commutants compared to the symbolic part of the proof of Theorem~\ref{ThmP}, and the extra Lagrangian regularity is captured using test modules, as introduced in \cite{HassellMelroseVasySymbolicOrderZero} and used for this purpose in \cite{MelroseVasyWunschEdge,MelroseVasyWunschDiffraction}; see also \cite{HaberVasyPropagation}. (Test modules also feature prominently in \cite{BaskinVasyWunschRadMink,BaskinVasyWunschRadMink2,GellRedmanHassellShapiroZhangHelmholtz}.) Fix $x_0<x_1<x_2<x_3<2 x_0$ and cutoffs
\[
  \chi_j \in \CIc([0,x_j)),\ \chi_j\equiv 1\ \text{on}\ [0,x_{j-1}],\quad j=1,2,3.
\]
Mirroring \cite[Definition~4.2]{MelroseVasyWunschEdge}, we then introduce:

\begin{definition}[Test module]
\label{DefPDModule}
  Let $\cF=\cF_I\cup\cF_O$. Define the $\Psich^0(X)$-module\footnote{Recall that $K_A$ denotes the Schwartz kernel of $A$.}
  \[
    \cM := \bigl\{ A\in\Psich^{0,0,0,1}(X) \colon \supp K_A \subset [0,1)_h\times(\supp\chi_2)^2,\ \sigmach_0(\tfrac{h}{h+x}A)|_\cF=0 \bigr\}.
  \]
  Denote by $\cM^k\subset\Psich^{0,0,0,k}(X)$ the set of finite linear combinations of up to $k$-fold products of elements of $\cM$. If $\cX$ is a function space on which $\Psich^0(X)$ acts continuously, we say that $u$ has \emph{Lagrangian regularity of order $k$ relative to $\cX$} if $\cM^k u\subset\cX$. We say that elements of the space $\cM^k\cX$ satisfy the \emph{nonfocusing condition of degree $k$ relative to $\cX$}.
\end{definition}

Since $\Psich^{0,0,0,1}(X)=\bigl(\tfrac{h}{h+x}\bigr)^{-1}\Psich^0(X)$, regularity with respect to elements of $\cM$ means that the semiclassical order improves upon differentiation along suitable elements of $\Psich^0(X)$. A concrete example of an element of $\cM$ in local coordinates is $\tfrac{h+x}{h}(\tfrac{h}{h+x}D_{y^j})=D_{y^j}$.

\begin{lemma}[Properties of $\cM$]
\label{LemmaPDModule}
  (Cf.\ \cite[Lemma~4.4]{MelroseVasyWunschEdge}.) The set $\cM$ is closed under commutators. Moreover, $\cM$ is finitely generated in the sense that there exist $A_1,\ldots,A_N\in\Psich^{0,0,0,1}(X)$ with $\supp K_{A_j}\in[0,1)_h\times(\supp\chi_3)^2$ so that with $A_0:=I$, we have
  \[
    \cM = \Biggl\{ A \in \Psich^{0,0,0,1}(X) \colon \exists\,Q_j\in\Psich^0(X),\ A=\sum_{j=0}^N Q_j A_j \Biggr\}.
  \]
  Concretely, one can take $A_N$ to have principal symbol $\tfrac{h+x}{h}\cdot\sigmach((\tfrac{x}{x+h})^2 P_{h,z})$; and one may take $A_j$, $1\leq j\leq N-1$, to have principal symbol $\tfrac{h+x}{h}a_j$, where $a_j\in\CI(\ol{\Tch^*}X_\chop)$ vanishes on $\cF$ and has differential $\dd a_j$ which at a point $\zeta\in\cR_{\rm in}$, resp.\ $\zeta\in\cR_{\rm out}$ lies in the unstable, resp.\ stable eigenspace of the linearization of $\sfH$ (as a vector field on $\Tch^*_\sface X_\chop$) at $\zeta$.
\end{lemma}
\begin{proof}
  Let $B=\tfrac{h+x}{h}B_0$, $C=\tfrac{h+x}{h}C_0\in\cM$. Denote the principal symbols of $B_0,C_0\in\Psich^0(X)$ by $b_0,c_0$. We then have $[B,C]\in\Psich^{-1,0,0,1}(X)$, and
  \[
    d := \sigmach_0(\tfrac{h}{h+x}i[B,C])=\tfrac{h}{h+x} H_b(\tfrac{h+x}{h}c_0) = \tfrac{h}{h+x}H_b(\tfrac{h+x}{h})c_0 + H_b c_0.
  \]
  But by~\eqref{EqCVHam}, $H_b|_{\hat h=0}$ is a smooth b-vector field for $b\in S^{0,0,0,1}$, thus $d\in S^0(\ol{\Tch^*}X_\chop)$. Moreover, since $\cF$ is a Lagrangian submanifold, $H_b$ is tangent to $\cF$; therefore, $H_b c_0=0$ on $\cF$ since $c_0|_\cF=0$, and thus $d|_\cF=0$ as well. This proves $[B,C]\in\cM$.

  Let us now work in local coordinates $(\hat h,x,y,\xi,\eta)$ in which the rescaled Hamilton vector field $\sfH=\hat h^{-1}H_p$ of $P_{h,z}$ takes the form~\eqref{EqPSymbHam}. The linearization of $\sfH$ at $\cR_{\rm out/in}$ as a vector field on $\Tch^*X_\chop$ is (top sign for `in', bottom sign for `out')
  \begin{equation}
  \label{EqPDCommHLin}
    {\mp}2(x\pa_x-\hat h\pa_{\hat h}-\eta\pa_\eta) + 2 k^{i j}\eta_i\pa_{y^j},
  \end{equation}
  which thus has eigenvalue $\mp 2$ (with eigenvector $\dd x$), $\pm 2$ (with eigenspace spanned by $\dd\hat h$ and $\dd\eta_j$), and $0$ (with eigenspace spanned by $\dd\xi$ and $\dd y^j\pm k^{i j}\dd\eta_i$). Upon restriction to $\hat h=0$, the same statements remain true except there is no contribution from $\dd\hat h$ anymore. Since $\cF$ is locally the joint zero set of $\eta^1,\ldots,\eta^{n-1}$, and $p$, which have linearly independent differentials, every smooth function vanishing on $\cF$ can be written as a linear combination (with smooth coefficients) of $p$ and $\eta_j$. Thus, we may take quantizations of $\hat h^{-1}\eta_j$ for the operators $A_j$ in local coordinates. The full collection of $A_j$ can be defined using a partition of unity.
\end{proof}

The fact that $\cM$ is a $\Psich^0(X)$-module and a Lie algebra implies that
\begin{equation}
\label{EqPDModulek}
  \cM^k = \Biggl\{ \sum_{|\alpha|\leq k} Q_\alpha A^\alpha \colon Q_\alpha\in\Psich^0(X) \Biggr\},\qquad A^\alpha := \prod_{i=1}^N A_i^{\alpha_i},
\end{equation}
where $\alpha=(\alpha_1,\ldots,\alpha_N)\in\N_0^N$. Since modulo $\Psich^0(X)$, the operator $A_N$ is a multiple of $P_{h,z}$, regularity of solutions $u$ of $P_{h,z}u=f$ under application of an element $Q_\alpha\prod_{i=1}^N A_i^{\alpha_i}\in\cM^k$ with $\alpha_N>0$ is automatic once Lagrangian regularity of order $k-1$ has been established. In order to prove regularity of solutions of $P_{h,z}u=f$ under application of $A_j$, $1\leq j\leq N-1$, we need to control the commutators of $P_{h,z}$ with the $A_j$ chosen in Lemma~\ref{LemmaPDModule}:

\begin{lemma}[Commutators]
\label{LemmaPDComm}
  (Cf.\ \cite[Lemma~4.5]{MelroseVasyWunschEdge}.) With the $A_j$ chosen as in Lemma~\usref{LemmaPDModule}, we have, for $j=1,\ldots,N-1$,
  \[
    i[P_{h,z},A_j] = \sum_{k=0}^N C_{j k}A_k,\qquad C_{j k}\in\Psich^{1,2,0,-1}(X),
  \]
  and $\sigmach_{1,2,0,-1}(C_{j k})|_{\cF\cap x^{-1}(0)}=0$ for $k\neq 0$.
\end{lemma}
\begin{proof}
  Denote by $a_j$ the principal symbol of $\hat h A_j$ for $j=1,\ldots,N-1$, so $\sigmach(A_j)=\hat h^{-1}a_j$. Since $P_{h,z}\in\Psich^{2,2,0,0}(X)$, we have $i[P_{h,z},A_j]\in\Psich^{1,2,0,0}(X)$, with principal symbol at $\sface$ given by $\hat h\sfH(\hat h^{-1}a_j)$ in the notation used in~\eqref{EqPSymbHam}. It thus suffices to prove the existence of $c_{j k}\in S^{1,2,0,0}(\ol{\Tch^*}X_\chop)$ such that near $\hat h=0$,
  \begin{equation}
  \label{EqPDCommSymb}
    \hat h\sfH(\hat h^{-1}a_j)=\sum_{k=1}^N \hat h c_{j k} \hat h^{-1}a_k,\qquad c_{j k}|_{\cF\cap x^{-1}(0)}=0\ \ (k\neq 0);
  \end{equation}
  indeed, if $C_{j k}\in\Psich^{1,2,0,-1}(X)$ is a quantization of $\hat h c_{j k}$ times a cutoff to a neighborhood of $\sface$, then \eqref{EqPDCommSymb} implies that $C_{j 0}:=i[P_{h,z},A_j]-\sum_{k=1}^N C_{j k}A_k\in\Psich^{0,2,0,-1}(X)$. In order to verify~\eqref{EqPDCommSymb}, we note that the left hand side equals $\hat h^{-1}(\hat h\sfH a_j - a_j\sfH\hat h)$; but since at $\cF\cap x^{-1}(0)$, the differentials $\dd a_j$  and $\dd\hat h$ are eigenvectors of the linearization of $\sfH$ with the same eigenvalue, as discussed after~\eqref{EqPDCommHLin}, this vanishes quadratically at $\cF\cap x^{-1}(0)$, completing the proof.
\end{proof}

We are now ready to propagate Lagrangian regularity through the radial sets. For $s,l,\alpha,b\in\R$ and $k\in\N_0$, and using the notation~\eqref{EqPDModulek}, denote
\[
  \Hch^{s,l,\alpha,b;k}(X) := \bigl\{ u\in\Hch^{s,l,\alpha,b}(X) \colon A^\alpha u\in\Hch^{s,l,\alpha,b}(X)\ \forall\,|\alpha|\leq k \bigr\}.
\]
We recall that we will only encounter distributions on $X$ with compact support, justifying the convenient, albeit slightly imprecise, notation here.

\begin{prop}[Microlocalized propagation near the radial sets]
\label{PropPDMicro}
  Let $s,l,\alpha,b\in\R$. Let $B,E,G\in\Psich^0(X)$ denote operators with Schwartz kernels supported in $[0,1)_h\times(\supp\chi_1)^2$. Recall the quantities $r_{\rm in/out}$ from Definition~\usref{DefPThr}.
  \begin{enumerate}
  \item\label{ItPDMicroIn} {\rm (Propagation into $\cR_{\rm in}$.)} Suppose that all backward integral curves of $\sfH$ starting in $\Sigma\cap\WFch'(B)$ either tend to a subset $\cS\subset\cR_{\rm in}$ or enter $\Ellch'(E)$ in finite time while remaining inside $\Ellch(G)$. Suppose moreover that for all incoming null-bicharacteristics $\gamma_{I,y_0}\colon(-x_0,0)\to\Sigma$ with $\gamma_{I,y_0}(0)\in\cS$, there exists $s\in(-x_0,0)$ (depending on $y_0$) such that $\gamma_{I,y_0}((s,0])\subset\Ellch(G)$ and $\gamma_{I,y_0}(s)\in\Ellch(E)$. Under the condition $b-\alpha>r_{\rm in}$, we then have
    \begin{equation}
    \label{EqPDMicroIn}
    \begin{split}
      \|B u\|_{\Hch^{s,l,\alpha,b;k}(X)} &\leq C\Bigl(\|G P_{h,z}u\|_{\Hch^{s-2,l-2,\alpha,b+1;k}(X)} \\
        &\quad\qquad+ \|E u\|_{\Hch^{-N,l,\alpha,b;k}(X)} + \|\chi_2 u\|_{\Hch^{-N,l,\alpha,b-\frac12}(X)}\Bigr).
    \end{split}
    \end{equation}
  \item\label{ItPDMicroOut} {\rm (Propagation out of $\cR_{\rm out}$.)}  Suppose that all backward integral curves of $\sfH$ starting in $\Sigma\cap\WFch'(B)$ either tend to a subset $\cS\subset\cR_{\rm out}$ or enter $\WFch'(E)$ in finite time while remaining inside $\Ellch(G)$. Suppose moreover that $\cS\subset\Ellch(G)$, and that for every integral curve $\gamma\subset\Sigma\cap x^{-1}(0)\setminus\cR_{\rm out}$ of $\sfH$ with $\lim_{s\to\infty}\gamma(s)\in\cS$, there exists $s$ so that $\gamma((s,\infty))\subset\Ellch(G)$ and $\gamma(s)\in\Ellch(E)$. Then the estimate~\eqref{EqPDMicroIn} holds under the condition $b-\alpha<r_{\rm out}$.
  \end{enumerate}
\end{prop}
\begin{proof}
  We begin with the proof of part~\eqref{ItPDMicroIn}. By compactness of $\cR_{\rm in/out}$ and since $\Ellch$ is open, it suffices to prove microlocal estimates near a single point $\zeta_0\in\cR_{\rm in}$, which in the coordinate system $(\hat h,x,y,\xi,\eta)$ used in~\eqref{EqPDFInOut} has coordinates $\zeta_0=(0,0,y_0,-1,0)$.

  Now, restricted to $x=\hat h=0$ and writing $k=k(y,\eta)$ for the dual metric function of the metric $k(0)$ on $\pa X$ in local coordinates, we have
  \[
    \sfH = -2\xi\eta\pa_\eta + 2|\eta|^2\pa_\xi + (\pa_\eta k)\pa_y - (\pa_y k)\pa_\eta.
  \]
  Following \cite[Lemma~2]{MelroseZworskiFIO}, introducing $|\eta|,\hat\eta=\frac{\eta}{|\eta|}$, one has $\pa_s y=(\pa_\eta k)(y,\hat\eta)|\eta|$ and $\pa_s\hat\eta=-(\pa_y k)(y,\hat\eta)|\eta|$ along $\sfH$-integral curves, hence reparameterizing to $t=t(s)$ satisfying $t'=2|\eta|$, one obtains
  \[
    \pa_t y=\half(\pa_\eta k)(y,\hat\eta),\quad
    \pa_t\hat\eta=-\half(\pa_y k)(y,\hat\eta),\quad
    \pa_t|\eta|=-\xi,\quad
    \pa_t\xi=|\eta|.
  \]
  Thus, $\xi(t)=a\cos(t+\varphi_0)$ and $|\eta(t)|=a\sin(t+\varphi_0)$ where $a=\sqrt{\xi^2+|\eta|^2}$ is constant, and $\varphi_0\in[0,\pi]$. Therefore, the function $\Ups$ assigning to $(y,\xi,\eta)$ near $(y_0,-1,0)$ the limiting point along the backward $\sfH$-integral curve is given by evaluation at $t=-\varphi_0$, so
  \[
    \Ups(y,\xi,\eta) = \Bigl(\exp_y\Bigl((-\arccos\tfrac{\xi}{\sqrt{\xi^2+|\eta|^2}}\bigr)\tfrac{\eta}{|\eta|}\Bigr),-1,0\Bigr).
  \]
  in particular, $\Ups$ is smooth, and $\sfH\Ups=0$ at $\hat h=x=0$. Extending $\Ups$ to a smooth function in a neighborhood of $x=\hat h=0$, with values in $\R^{n-1}\times\R\times\R^{n-1}$, we thus have $\sfH\Ups=\cO(x)$ at $\hat h=0$. Since $x^{-1}\sfH x=-2$ at $\zeta_0$, we can choose $C$ so that in any sufficiently small neighborhood $\cV$ of $\zeta_0$,
  \begin{equation}
  \label{EqPDMicroSign}
    \sfH\bigl(|\Ups-\zeta_0|^2-C x\bigr) \geq x > 0\ \text{in}\ \cV.
  \end{equation}

  Fix now cutoffs $\chi_\cS,\chi_\pa,\chi_\sface,\chi_\cF,\chi_\Sigma\in\CIc([0,1))$, identically $1$ near $0$, with nonpositive derivative and with $\sqrt{-\chi_\bullet\chi_\bullet'}\in\CI$, and consider the commutant
  \[
    \check a = \hat h^{-b-\frac12}x^{-\alpha}\chi_\pa\bigl(\tfrac{x}{\delta}\bigr)\chi_\sface\bigl(\tfrac{\hat h}{\delta}\bigr)\chi_\cF\Biggl(\delta^{-1}\sum_{j=1}^{N-1} a_j^2\Biggr)\chi_\Sigma\bigl(\tfrac{p^2}{\delta}\bigr) \chi_\cS\bigl(\delta^{-1}(|\Ups-\zeta_0|^2-C x)\bigr),
  \]
  where $\delta>0$ controls the size of $\supp\check a$. We now proceed as in the first step of the proof of Theorem~\ref{ThmP}. Thus, in the symbol~\eqref{EqPThmPCSymb} of the commutator appearing in~\eqref{EqPThmPC}, and specifically in the term $2\hat h\check a\sfH\check a$, the main contribution near $\zeta_0$ arises from differentiation of the weights (and then the subprincipal symbol of $P_{h,z}$ enters in the threshold condition on $b-\alpha$ as there), giving a negative multiple of $\hat h^{-2 b}x^{-2\alpha}$. Differentiation of $\chi_\cF$ gives a term of the same sign, namely a negative square, since $\sum a_j^2$ is a local quadratic defining function of $\cR_{\rm in}$ inside of $\Sigma\cap x^{-1}(0)$. In view of~\eqref{EqPDMicroSign}, differentiation of $\chi_\cS$ produces $x$ times the negative of a square, thus another term with sign matching that of the main term. Derivatives falling on $\chi_\pa$ produce a nonnegative square, corresponding to the a priori control required along $\gamma_{I,y}$, for $y$ near $y_0$, at $x\sim\delta$. Finally, differentiation of $\chi_\Sigma$ produces a term vanishing near $\Sigma$ which thus can be controlled by elliptic regularity, and differentiation of $\chi_\sface$ produces a semiclassically trivial (namely, vanishing near $\hat h=0$) term. We can then proceed as in~\eqref{EqPThmPCSymbSq}, obtaining the desired propagation estimate.

  For $k\geq 1$, we argue as in the proof of \cite[Proposition~4.4]{BaskinVasyWunschRadMink}: rather than using $\check A=\Op_{\cop,h}(\check a)$ as the commutant, we use (in the notation~\eqref{EqPDModulek}) the vector of ps.d.o.s $(\check A A^\alpha)_{\alpha\in\cI}$ where $\cI\subset\N_0^N$ consists of all $\alpha\in\N_0^N$ with $|\alpha|=k$ and $\alpha_0=\alpha_N=0$. The main term of the commutator arises from $\check A$ as before; the new contributions, from commutators of $P_{h,z}$ with a factor $A_j$, can be expanded as in Lemma~\ref{LemmaPDModule}, and those which have the maximal number of module factors $A_l$, $1\leq l\leq N-1$, can be absorbed into this main term due to the vanishing property of the $C_{j k}$ in Lemma~\ref{LemmaPDModule}. Thus, one can control $k$ module derivatives of $u$ in a neighborhood of $\zeta_0$ provided one has control of $k-1$ module derivatives in a slightly bigger neighborhood. Thus, one obtains the estimate~\eqref{EqPDMicroIn} inductively.

  The proof of part~\eqref{ItPDMicroOut} is completely analogous; one now takes $\Ups$ at $x=\hat h=0$ to be the limiting point along \emph{forward} $\sfH$-integral curves.
\end{proof}

Note that for any $\zeta\in\ol{\Tch^*_\sface}X_\chop\setminus\cF$, there exists an element $A\in\cM$ which is elliptic at $\zeta$; hence microlocally near such $\zeta$, membership in $\Hch^{s,l,\alpha,b;k}$ is equivalent to membership in $\Hch^{s,l,\alpha,b+k}$.\footnote{That is, for $B,\tilde B\in\Psich^0(X)$ with $\WFch'(B)\subset\Ellch(\tilde B)\setminus\cF$, one has $\|B u\|_{\Hch^{s,l,\alpha,b+k}}\lesssim\|\tilde B u\|_{\Hch^{s,l,\alpha,b;k}}+\|u\|_{\Hch^{-N,l,\alpha,-N}}$.} In particular, in $\Sigma\cap x^{-1}(0)$ but away from the radial sets, the propagation of $\Hch^{s,l,\alpha,b;k}$ regularity is equivalent to the standard (real principal type) propagation of $\Hch^{s,l,\alpha,b+k}$ regularity. One can thus concatenate the radial point estimates of Proposition~\ref{PropPDMicro} with such real principal type estimates. To state this succinctly, we introduce:

\begin{definition}[Integral curves connecting the radial sets]
\label{DefPDCurves}
  \begin{enumerate}
  \item For $y\in\pa X$, denote by $\Gamma^\rightarrow(y)\subset\cC^0([0,\pi];\Sigma)$ the set of integral curves of $\sfH$ inside $\Sigma\cap x^{-1}(0)$, smoothly reparameterized to uniformly continuous curves $\gamma\colon(0,\pi)\to\Sigma\cap x^{-1}(0)$, which satisfy $\gamma(\pi)=(0,0,y,1,0)\in\cR_{\rm out}$ and $\gamma(0)\in\cR_{\rm in}$. Denoting by $\Pi\colon\Sigma\cap x^{-1}(0)\to\pa X$ the projection to the base, define the set of starting points of such curves by
  \[
    \cY^\rightarrow(y)=\{\Pi(\gamma(0))\colon\gamma\in\Gamma^\rightarrow(y)\}.
  \]
  \item We call a continuous curve $\gamma\colon I\to\Sigma$ a \emph{resolved GBB} if it is either an integral curve of $h^{-1}H_p$ disjoint from $x^{-1}(0)$, or otherwise if for some $y\in\pa X$ and $y_0\in\cY^\rightarrow(y)$, the curve $\gamma$ is the concatenation of $\gamma_{I,y_0}$, an element $\gamma$ of $\Gamma^\rightarrow(y)$ with $\Pi(\gamma(0))=y_0$, $\Pi(\gamma(\pi))=y$, and the curve $\gamma_{O,y}$.
  \end{enumerate}
\end{definition}

See Figure~\ref{FigPDResolved}.

\begin{figure}[!ht]
\centering
\includegraphics{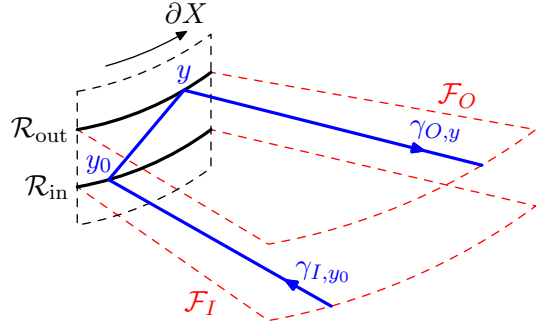}
\caption{Illustration of a part of the characteristic set, with the Lagrangians $\cF_I$ and $\cF_O$ in red, and a resolved GBB in blue.}
\label{FigPDResolved}
\end{figure}

\begin{cor}[Microlocalized propagation of Lagrangian regularity]
\label{CorPD}
  Let $s,l,\alpha\in\R$, $\sfb\in\CI(\ol{\Tch^*_\sface}X_\chop)$ and $\chi\in\CIc(X)$ with $\supp\chi\subset\supp\chi_1$ be as in Theorem~\usref{ThmP}, with $\sfb-\alpha$ satisfying the monotonicity and threshold conditions stated there. Let $k\in\N_0$. Let $B,E,G\in\Psich^0(X)$, with Schwartz kernels supported in $[0,1)_h\times(\chi^{-1}(1)\times\chi^{-1}(1))$. Suppose that all backward resolved GBBs starting at a point in $\WFch'(B)$ reach $\Ellch(E)$ in finite time while remaining in $\Ellch(G)$. Then the estimate~\eqref{EqPDMicroIn} holds.
\end{cor}

A dualization argument gives the propagation of the nonfocusing condition through $\pa X$. The simplest setting uses the modification of $P_{h,z}$ via extension to an operator $P'_{h,z}$ on compact manifold $X'\supset X$ and the inclusion of a complex absorbing term $Q\in\Psih^{-\infty}((X')^\circ)$ as in~\S\ref{SsPA}, resulting in the operator
\[
  \cP_{h,z}=P'_{h,z}-i Q
\]
in~\eqref{EqPAcP}. Recall that the Schwartz kernel of $Q$ has empty intersection with $x^{-1}([0,c x_0])\times x^{-1}([0,c x_0])$ where $0<c\ll\half$. We shall use the notation of Proposition~\ref{PropPA}.

\begin{thm}[Diffractive improvement]
\label{ThmPD}
  Let $s,l,\alpha,\sfb$ be as in the statement of Theorem~\usref{ThmP} (for the operator $P_{h,z}$). Let $E,G\in\Psich^0(X')$ be such that all forward resolved GBBs starting at a point in $\WFch'(E)\subset\Ellch(G)$ remain in $\Ellch(G)$ until they enter $\Ell_h(Q)$. Let $f_+\in\Hch^{s-2,l-2,\alpha,\sfb+1}(X')$, $f_-\in\cM^k\Hch^{s-2,l-2,\alpha,\sfb+1}(X')$ be such that $\supp f_\pm\subset x^{-1}([0,c x_0])$. Then the solution $u$ of
  \[
    \cP_{h,z}u = f := f_+ + E f_-
  \]
  can be written in the form
  \[
    u = u_+ + G u_-,\qquad u_+\in\Hch^{s,l,\alpha,\sfb}(X'),\quad u_-\in\cM^k\Hch^{s,l,\alpha,\sfb}(X').
  \]
\end{thm}

Note that on the scale of semiclassical cone Sobolev spaces, we have $E f_-\in\Hch^{s-2,l-2,\alpha,\sfb'+1}$ with $\sfb'=\sfb-k$, but typically $E f_-$ is no better than this. Thus, Theorem~\ref{ThmPD} (for $f_+=0$ for concreteness) implies that the strong semiclassical singularities of $u$ resulting from the forcing term $E f_-$ only propagate along geometric GBBs (resulting in the term $G u_-$), whereas microlocally away from these, $u$ has $\Hch^{s,l,\alpha,\sfb}$-regularity.

In a simple case, a formulation of Theorem~\ref{ThmPD} which highlights regularity rather than singularities reads as follows: fix $y_0\in\pa X$, and define the set
\[
  K := \gamma_{O,y_0} \cup \bigcup_{\gamma\in\Gamma^\rightarrow(y_0)} \gamma([0,\pi]) \cup \bigcup_{y\in\cY^\rightarrow(y_0)} \gamma_{I,y}.
\]
Thus, the quotient $K/(K\cap x^{-1}(0))$ contains the image of all backward geometric GBB continuing $\gamma_{O,y_0}$, and $K$ in addition contains all curves inside of $\Sigma\cap x^{-1}(0)$ which connect an incoming base point $y$ (at distance $\pi$ from $y_0$) with the outgoing base point $y_0$ of geometric GBBs. Fixing any $E\in\Psich^0(X')$ with $\WFch'(E)\cap K=\emptyset$, there then exists $G\in\Psich^0(X')$ with $\WFch'(G)\cap K=\emptyset$ which satisfies the conditions of Theorem~\ref{ThmPD}. Thus, if $f$ satisfies the nonfocusing condition (of some degree $k$) relative to $\Hch^{s-2,l-2,\alpha,\sfb+1}$, and with $f$ microlocally near $K$ lying in $\Hch^{s-2,l-2,\alpha,\sfb+1}$ (thus $f$ in particular does not have strong singularities along the incoming directions $\gamma_{I,y}$), then the semiclassical wave $u$ forced by $f$ lies in $\Hch^{s,l,\alpha,\sfb}$ microlocally near $K$ (thus $u$ in particular does not have a strong singularity along $\gamma_{O,y_0}$).

\begin{proof}[Proof of Theorem~\usref{ThmPD}]
  As follows from Proposition~\ref{PropPA} by taking adjoints (or directly from the proof of Proposition~\ref{PropPA}), the adjoint $\cP_{h,z}^*$ is invertible, and
  \[
    (\cP_{h,z}^*)^{-1} \colon \Hch^{-s,-l,-\alpha,-\sfb}\to\Hch^{-s+2,-l+2,-\alpha,-\sfb-1}
  \]
  is uniformly bounded. We now apply a backward propagation version of Corollary~\ref{CorPD} to $P_{h,z}^*$: for $E^*,G^*$ the adjoints of the operators $E,G$ in the statement of the Theorem, and for $B^*\in\Psich^0(X')$ so that all forward resolved GBBs starting at a point in $\WFch'(E^*)$ remain in $\Ellch(G^*)$ until they enter $\Ellch(B^*)$, we have
  \begin{align*}
    &\|E^*v\|_{\Hch^{-s+2,-l+2,-\alpha,-\sfb-1;k}} \\
    &\qquad \leq C\Bigl(\|G^*P_{h,z}^*v\|_{\Hch^{-s,-l,-\alpha,-\sfb;k}} + \|B^*v\|_{\Hch^{-s+2,-l+2,-\alpha,-\sfb-1;k}} + \|\chi u\|_{\Hch^{-N,-l+2,-\alpha,-\sfb-\frac32}}\Bigr)
  \end{align*}
  for any $k\in\N_0$. In particular, we may take $B^*$ so that all forward null-bicharacteristics of $P_{h,z}$ starting in $\WFch'(B^*)$ miss the cone point and enter $\Ellch(Q^*)$ in finite time. The term $B^*u$ is then automatically controlled for solutions of $\cP_{h,z}^*v=w$ when $G^*w\in\Hch^{-s,-l,-\alpha,-b;k}$ by elliptic regularity (on $\Ellch(Q)$) and real principal type propagation (along the backward null-bicharacteristic flow) with complex absorption. We conclude that
  \begin{align*}
    (\cP_{h,z}^*)^{-1} &\colon \bigl\{ w\in\Hch^{-s,-l,-\alpha,-\sfb} \colon G^*w\in\Hch^{-s,-l,-\alpha,-\sfb;k} \bigr\} \\
      &\qquad \to \bigl\{ v\in\Hch^{-s+2,-l+2,-\alpha,-\sfb-1} \colon E^*v\in\Hch^{-s+2,-l+2,-\alpha,-\sfb-1;k} \bigr\}.
  \end{align*}
  Upon taking adjoints (see also \cite[Appendix~A]{MelroseVasyWunschDiffraction}), this implies that
  \[
    \cP_{h,z}^{-1} \colon \Hch^{s-2,l-2,\alpha,\sfb+1} + E(\cM^k\Hch^{s-2,l-2,\alpha,\sfb+1}) \to \Hch^{s,l,\alpha,\sfb} + G(\cM^k\Hch^{s,l,\alpha,\sfb})
  \]
  is a bounded map. This completes the proof.
\end{proof}

\begin{rmk}[Second microlocalization at $\cF$]
  A sharper approach would be to second microlocalize at $\cF_I$ and $\cF_O$, thus cleanly decoupling the semiclassical orders at $\cF_I$ and $\cF_O$ (subject to threshold conditions at the radial sets) and the semiclassical order away from $\cF$; this would allow for a unified treatment of Lagrangian and nonfocusing spaces and thus for a direct proof of Theorem~\ref{ThmPD}. We leave such refinements for future work. We note that second microlocalization in the semiclassical setting was studied by Sj\"ostrand--Zworski \cite{SjostrandZworskiFractalUpper} and Vasy--Wunsch \cite{VasyWunsch2nd} following Bony's work \cite{Bony2nd}; a second microlocal refinement (at the outgoing radial set) for the scattering theory of the corresponding normal operator was recently obtained by Vasy \cite{VasyLowEnergyLag}.
\end{rmk}

\section{Applications}
\label{SH}

We now present applications of the propagation estimates proved in~\S\ref{SP}. First, we discuss the familiar geometric case of $h^2\Delta_g-1$ in~\S\ref{SsHL}, where we can moreover prove a result sharpening both Theorem~\ref{ThmP} and the propagation results of \cite{BaskinMarzuolaCone}. We discuss high frequency scattering by inverse square potentials on Euclidean space in~\S\ref{SsHV}, and high frequency scattering for the Dirac--Coulomb equation in~\S\ref{SsHDC}.

\subsection{Propagation estimates for conic Laplacians}
\label{SsHL}

For a conic metric $g$ as in~\eqref{EqPProdMet} on the manifold $X=[0,2 x_0)_x\times Y$ of dimension $n=\dim X\geq 3$, we consider
\[
  P_{h,z} = h^2\Delta_g - z,\quad |z-1|<C h.
\]
We fix the volume density $\mu=|\dd g|$ on $X$.

\begin{lemma}[Admissibility, thresholds, invertibility]
\label{LemmaHL}
  The operator $P_{h,z}$ is admissible in the sense of Definition~\usref{DefPOp}, with threshold quantities $r_{\rm in}=-\half$ and $r_{\rm out}=-\half$ (see Definition~\usref{DefPThr}). Moreover, the normal operator $N(P)=\Delta_{\hat g}-1$, with $\hat g$ given in~\eqref{EqPhat}, is invertible at weight $l$ in the sense of Definition~\usref{DefPNormOp}\eqref{ItPNormOpInv} for all
  \begin{equation}
  \label{EqHLRange}
    l\in\bigl(1-\tfrac{n-2}{2},1+\tfrac{n-2}{2}\bigr).
  \end{equation}
\end{lemma}
\begin{proof}
  Only the final statement is nontrivial. In the notation~\eqref{EqPNormOpMellin}, and passing to a spectral decomposition of $\Delta_{k(0)}$ whose eigenvalues we denote by $0\leq\lambda_j^2$, $j=0,1,2,\ldots$, one finds that $\lambda\in\specb(N(P))$ if and only if there exists $j$ with $\lambda^2-i(n-2)\lambda+\lambda_j^2=0$, so
  \[
    \specb(N(P)) = \left\{ i\left(\frac{n-2}{2}\pm\sqrt{\Bigl(\frac{n-2}{2}\Bigr)^2+\lambda_j^2}\,\right) \colon j=0,1,2,\ldots \right\}.
  \]
  Therefore, the complement of the set $\Lambda$ defined in~\eqref{EqPLambda} contains $(-n+2,0)$. As noted in Remark~\ref{RmkPNWeight}, the invertibility of $N(P)$ at weight $l$ is independent of the choice of $l$ inside the shifted interval $\frac{n}{2}+(-n+2,0)=(1-\frac{n-2}{2},1+\frac{n-2}{2})$.
  
  The choice $l=1$ is particularly natural, as the space $H_{\bop,\scop}^{1,1,0}(\tface;\hat\mu)$ is the quadratic form domain of $\Delta_{\hat g}$ (as follows from Hardy's inequality). The invertibility of $N(P)$ at weight $l=1$ is then equivalent to the limiting absorption principle for the exact conic metric $\hat g$, the proof of which is a standard application of a boundary pairing argument \cite[\S2.3]{MelroseGeometricScattering} and unique continuation at infinity. See Lemma~\ref{LemmaHVScalar} below for a proof is a more general setting.
\end{proof}

As a consequence, we may apply Theorem~\ref{ThmP} for $l$ in the range~\eqref{EqHLRange}, any value of $s\in\R$, $\alpha=0$, and variable orders $\sfb$ satisfying in particular $\sfb>-\half$ at $\cR_{\rm in}$, $\sfb<-\half$ at $\cR_{\rm out}$, and we may arrange that $|\sfb-(-\half)|<\eps$ for any fixed $\eps>0$. Packaged in the form of Proposition~\ref{PropPA} using complex absorption, we thus have, using the volume density $|\dd g|$ near $\pa X$,
\begin{equation}
\label{EqHLLossless}
  \|u\|_{\Hch^{s,l,0,\sfb}} \leq C\|\cP_{h,z}u\|_{\Hch^{s-2,l-2,0,\sfb+1}};
\end{equation}
this estimate is sharp in the sense explained after the statement of Theorem~\ref{ThmP}. Lossy estimates on constant order spaces are given by
\begin{align*}
  &\bigl\|(x+h)^{-\frac12-\eps}u\bigr\|_{\Hch^{s,l,0,0}} \sim h^{-\frac12-\eps}\|u\|_{\Hch^{s,l,0,-\frac12-\eps}} \\
  &\qquad \lesssim h^{-\frac12-\eps}\|\cP_{h,z}u\|_{\Hch^{s-2,l-2,0,\frac12+\eps}} \sim h^{-1-2\eps}\bigl\|(x+h)^{\frac12+\eps}\cP_{h,z}u\bigr\|_{\Hch^{s-2,l-2,0,0}}.
\end{align*}
In the special case $s=l$, and recalling from \cite[Theorem~6.3]{HintzConicPowers} that the domain
\[
  \cD_h^l = \cD\bigl((h^2\Delta_{g'}+1)^{l/2}\bigr)
\]
of the $(l/2)$-th power of the Friedrichs extension of the conic Laplacian $h^2\Delta_{g'}+1$ is equal to $\Hch^{l,l,0,0}$ in present notation, this gives:

\begin{prop}[Constant order estimates]
\label{PropHLLossy}
  In the above setting and with $l$ as in~\eqref{EqHLRange}, we have for all $\eps>0$ the estimates
  \begin{subequations}
  \begin{align}
  \label{EqHLLossy}
    \bigl\|(x+h)^{-\frac12-\eps}u\|_{\cD_h^l} &\leq C_\eps h^{-1-2\eps}\bigl\|(x+h)^{\frac12+\eps}\cP_{h,z}u\bigr\|_{\cD_h^{l-2}}, \\
  \label{EqHLLossier}
    \|u\|_{\cD_h^l} &\leq C_\eps h^{-1-2\eps}\|\cP_{h,z}u\|_{\cD_h^{l-2}},
  \end{align}
  \end{subequations}
  as well as more general estimates with $\cD_h^l$ and $\cD_h^{l-2}$ replaced by $\Hch^{s,l,0,0}$ and $\Hch^{s-2,l-2,0,0}$.
\end{prop}

The estimate~\eqref{EqHLLossier} is an immediate consequence of~\eqref{EqHLLossy}. We recall that in the case $l=1$, the (arbitrarily small) $2\eps$-loss in~\eqref{EqHLLossier} can be removed, as shown in the semiclassical cone setting by Baskin--Marzuola \cite{BaskinMarzuolaCone} following arguments by Melrose, Wunsch, and Vasy \cite{MelroseWunschConic,MelroseVasyWunschEdge}; in the full range of weights $l$ considered here, a lossless estimate was obtained by the author in \cite[\S6.2]{HintzConicPowers} via reduction to the case $l=1$ via conjugation by $(1+h^2\Delta_{g'})^{(l-1)/2}$ and reduction to the case $l=1$. On the other hand, the estimate~\eqref{EqHLLossier}, even for $\eps=0$, loses a full order at $\tface$ compared to the sharper estimate~\eqref{EqHLLossy}.

\begin{rmk}[Limiting absorption principle]
\label{RmkHLLAP}
  The overall $h^{-1-2\eps}$ loss in Proposition~\ref{PropHLLossy} is familiar from (and essentially arises from) the loss of slightly more than one power of $\la z\ra$ in the limiting absorption principle
  \[
    (\Delta-1\pm i 0)^{-1}\colon\la z\ra^{-\frac12-\eps}L^2(\R^n)\to\la z\ra^{\frac12+\eps}H^2(\R^n)
  \]
  on Euclidean space, which is a consequence of a sharp variable order estimate akin to~\eqref{EqHLLossless}, see Lemma~\ref{LemmaPNEst}.
\end{rmk}

A natural question is whether one can prove an estimate which removes both the $\eps$-loss of~\eqref{EqHLLossier} while retaining the lossless character of~\eqref{EqHLLossy} (or \eqref{EqHLLossless}) at $\tface$. We answer this in the affirmative:

\begin{thm}[Sharp propagation estimate]
\label{ThmHL}
  Consider a conic manifold $(X,g)$ as in~\eqref{EqPMfd}--\eqref{EqPProdMet} and with $\dim X\geq 3$. Let $P_{h,z}=h^2\Delta_g-z$, $|z-1|<C h$. Denote the characteristic set of $P_{h,z}$ by $\Sigma\subset\Tch^*_\sface X_\chop$, see~\eqref{EqPSymb}. Let $\chi,\tilde\chi\in\CI(X)$, with $\tilde\chi\equiv 1$ near $\supp\chi$, and $E\in\Psich^{-\infty}(X)$. Suppose that all backward GBB from $\Sigma\cap\supp\chi$ enter $\Ellch(E)$ in finite time while remaining inside $\supp\tilde\chi$. Then, for any $s,N\in\R$, we have an estimate
  \begin{equation}
  \label{EqHLSharpML}
  \begin{split}
    \|\chi u\|_{\Hch^{s,1,0,0}(X)} &\leq C\Bigl(\|\tilde\chi P_{h,z}u\|_{\Hch^{s-2,-1,0,1}(X)} \\
      &\quad\qquad + \|E u\|_{\Hch^{-N,-N,-N,0}(X)} + h^{\frac12}\|\tilde\chi u\|_{\Hch^{-N,-N,0,0}(X)}\Bigr).
  \end{split}
  \end{equation}
  This holds more generally with the norms in the first line replaced by $\|\chi u\|_{H_{\cop\bop,h}^{s,s',1,0,0}}$ and $\|\tilde\chi P_{h,z}u\|_{H_{\cop\bop,h}^{s-2,s'-2,-1,0,1}}$, with $s'\in\R$ arbitrary. Taking $s=1$, $N=0$ in~\eqref{EqHLSharpML} gives
  \begin{equation}
  \label{EqHLSharp}
    \|\chi u\|_{\cD_h^1} \lesssim h^{-1}\|(x+h)\tilde\chi P_{h,z}u\|_{\cD_h^{-1}} + \|E u\|_{L^2} + h^{\frac12}\|\tilde\chi u\|_{L^2},
  \end{equation}
  and upon adding complex absorption as in~\S\usref{SsPA} and equation~\eqref{EqPAcP}, we have
  \begin{equation}
  \label{EqHLSharpCAP}
    \|u\|_{\cD_h^1} = \|u\|_{\Hch^{1,1,0,0}} \lesssim h^{-1}\|(x+h)\cP_{h,z}u\|_{\cD_h^{-1}} = \|\cP_{h,z}u\|_{\Hch^{-1,-1,0,1}}.
  \end{equation}
\end{thm}

For comparison, the $h$-lossless version of~\eqref{EqHLLossier} for $l=1$ reads
\begin{equation}
\label{EqHLConseqComp}
  \|u\|_{\Hch^{1,1,0,0}} \lesssim h^{-1}\|\cP_{h,z}u\|_{\Hch^{-1,-1,0,0}} = \|\cP_{h,z}u\|_{\Hch^{-1,-1,1,1}};
\end{equation}
which is weaker than Theorem~\ref{ThmHL} in that the required control on $\cP_{h,z}$ at $\tface$ is one order stronger than in the Theorem.

As discussed after Theorem~\ref{ThmP}, the estimate~\eqref{EqHLSharpML} is sharp in the sense that the relative orders on $u$ on the left and $P_{h,z}u$ on the right cannot be improved; but here the semiclassical order remains fixed upon propagation through the cone point.

The proof of Theorem~\ref{ThmHL} uses the global positivity (as an operator) of a commutator on $\tface$, reminiscent of proofs of similar lossless results in $N$-body scattering \cite{VasyThreeBody,VasyManyBody}, as well as a splitting of $u$, using the functional calculus for $h^2\Delta_g$, into a part localized near the characteristic set and a part where $h^2\Delta_g-1$ is elliptic and can be inverted by spectral theory.

\begin{rmk}[Dimension]
  We only study the case $\dim X\geq 3$ here. The methods used in \cite{MelroseVasyWunschEdge,BaskinMarzuolaCone} based on quadratic forms, and also \cite{MelroseWunschConic}, work in the case $\dim X=2$ as well. However, the identification of the quadratic form domain with a semiclassical cone Sobolev space fails in this case (see \cite[Equation~(3.11)]{MelroseWunschConic} for $h=1$), which is why we do not consider it here.
\end{rmk}

\begin{proof}[Proof of Theorem~\usref{ThmHL}]
  We present the proof in the case that $P_{h,z}$ agrees with its normal operator, equivalently $P_{h,z}=h^2\Delta_g-1$ with $g=\dd x^2+x^2 k(y,\dd y)$ an exact conic metric. In the general case, the error terms arising from $P_{h,z}-N(P)\in\Psich^{2,-2,-1,0}(X)$ are handled easily; we leave the details to the reader. (In particular, since we shall use a global commutator argument which controls $u$ at $\sface$ and $\tface$ in one fell swoop, there is no need for a delicate argument for the combination of the symbolic estimate at $\sface$ and a normal operator estimate at $\tface$ as in the end of the proof of Theorem~\ref{ThmP}.) We write $P\equiv P_{h,z}$ for brevity.

  \pfstep{Positive commutator argument.} Define the operator
  \begin{equation}
  \label{EqHLA}
    A:=\tfrac{h}{2}\bigl(x D_x+(x D_x)^*\bigr)-1=h x D_x-\tfrac{i n h}{2}-1,\qquad
    a:=\sigmach(A)=x\xi-1,
  \end{equation}
  where we use the coordinates~\eqref{EqCVCoordTilde}. This will be the main piece of the commutant in a positive commutator calculation, and it is in essence the key term both in the commutator argument of \cite{VasyPropagationCorners} as well as in the Mourre commutant in classical scattering theory \cite{MourreSingular}. Let $\chi=\chi(x)$ be identically $1$ near $\pa X=x^{-1}(0)$, with support in any pre-specified neighborhood of $\pa X$, and so that $\chi'\leq 0$, $\sqrt{-\chi\chi'}\in\CI$, and so that $a<0$ has a constant (negative) sign on $\Sigma\cap\supp\chi$; arranging the latter property is what the constant term in~\eqref{EqHLA} is for. We then consider the operator
  \[
    \tilde A := \chi A\chi,
  \]
  and estimate in two different ways the expression
  \begin{equation}
  \label{EqHLComm}
    2 h^{-1} \Im \la P u,\tilde A u\ra = \big\la \tfrac{i}{h}[P,\tilde A]u,u\big\ra.
  \end{equation}

  Consider first the commutator term. Since $h^2\Delta_g$ is homogeneous of degree $-2$, we have $\frac{i}{h}[P,A]=-[x\pa_x,P]=2 h^2\Delta_g$, which is the crucial global positive commutator. Therefore,
  \[
    \tfrac{i}{h}[P,\chi A\chi] = 2\chi h^2\Delta_g\chi + \tfrac{i}{h}\bigl(\chi A[P,\chi]+[P,\chi]A\chi\bigr).
  \]
  The contribution of the first term to the right hand side of~\eqref{EqHLComm} is
  \[
    2\|h\nabla_g(\chi u)\|^2,
  \]
  where we write $\|\cdot\|\equiv\|\cdot\|_{L^2}$. The second term on the other hand consists of operators with coefficients supported strictly away from $x=0$. It suffices to merely capture its principal symbol, which by~\eqref{EqPSymbHam} is $2\chi a h^{-1}H_p\chi=a x^{-1}\sfH(\chi^2)=4 a\xi\chi\chi'$; near incoming directions, where $\xi<0$, this is negative, whereas near outgoing directions, where $\xi>0$, this is positive and thus has a sign matching that of the above main commutator term. For a suitable microlocal cutoff $E\in\Psich^0(X)$ which is elliptic on $\Sigma$ in the region $\xi<\eps$ for some fixed small $\eps\in(0,1)$, we thus conclude from~\eqref{EqHLComm} that
  \begin{equation}
  \label{EqHLComm2}
    2\|h\nabla_g(\chi u)\|^2 \leq 2\Im\big\la\bigl(h D_x-\tfrac{i n}{2}\tfrac{h}{x}\bigr)(\chi u), h^{-1} x\chi P u\big\ra - 2 h^{-1}\Im\la \chi u,\chi P u\ra + \|E u\|^2.
  \end{equation}
  Hardy's inequality gives $\|\tfrac{h}{x}\chi u\|\leq C_n\|h D_x(\chi u)\|$, hence the first term on the right is bounded from above by
  \[
    \eps\|h D_x(\chi u)\|^2 + C_\eps\|h^{-1}x\chi P u\|^2
  \]
  For the second term in~\eqref{EqHLComm2}, we rewrite
  \[
    \la\chi u,\chi P u\ra = -\|\chi u\|^2 + \la h\nabla_g(\chi^2 u),h\nabla_g u\ra = -\|\chi u\|^2 + \|\chi h\nabla_g u\|^2 + \la h(\nabla_g\chi^2)u,h\nabla_g u\ra;
  \]
  taking the imaginary part annihilates the first two terms, while for the final term we have
  \begin{align}
    &\la h(\nabla_g\chi^2)u,h\nabla_g u\ra - \la h\nabla_g u,h(\nabla_g\chi^2)u\ra \nonumber\\
    &\qquad = \big\la h^2\nabla_g\bigl((\nabla_g\chi^2)u\bigr)-h^2(\nabla_g\chi^2)\cdot\nabla_g u,u\big\ra \nonumber\\
  \label{EqHLImag}
    &\qquad = \la h^2\Delta_g(\chi^2) u,u\ra,
  \end{align}
  with no derivatives falling on $u$ anymore. Altogether, we obtain from~\eqref{EqHLComm2} the estimate
  \begin{equation}
  \label{EqHLEst0}
    (2-C\eps)\|h\nabla_g(\chi u)\|^2 \leq C_\eps\|h^{-1}x\chi P u\|^2 + C h\|\tilde\chi u\|^2 + \|E u\|^2,
  \end{equation}
  where $\tilde\chi\equiv 1$ on $\supp\chi$, used to bound the contribution of~\eqref{EqHLImag}.

  \pfstep{Control of $\chi u$ in $\Hch^{1,1,0,0}$.} Since the principal symbol of $h\nabla_g\in\Psich^{1,1,0,0}(X;\C,\Tch^*X_\chop)$ (mapping complex-valued functions into sections of $\Tch^*X_\chop$, cf.\ Remark~\ref{RmkCPBundles}) is not injective at the zero section over $\sface$, the estimate~\eqref{EqHLEst0} does not yet give full control of $\chi u$ in $\Hch^{1,1,0,0}$: an estimate of $\|\chi u\|_{L^2}^2$ is lacking at this point. (Note that the control of $\frac{h}{x}\chi u$ via Hardy's inequality degenerates precisely at $\sface$, i.e.\ the lift of $h=0$.) The key observation is that the characteristic set of $P$ and the set where the principal symbol of $h\nabla_g$ fails to be injective are disjoint. Thus, for some $A_1\in\Psich^{-1,-1,0,0}(X;\Tch^*X_\chop,\C)$ and $A_2\in\Psich^{-2,-2,0,0}(X)$, we have
  \[
    I=A_1\circ h\nabla_g + A_2 P+R,\qquad R\in\Psich^{-\infty,0,0,-\infty}(X);
  \]
  this implies
  \begin{equation}
  \label{EqHLEll}
    \|\chi u\| \leq C\bigl(\|h\nabla_g(\chi u)\|_{\Hch^{-1,-1,0,0}(X)} + \|P(\chi u)\|_{\Hch^{-2,-2,0,0}(X)}\bigr) + \|R(\chi u)\|_{L^2}.
  \end{equation}
  Using $[P,\chi]\in\Psich^{1,-\infty,-\infty,-1}(X)$, we can estimate the second term by
  \begin{align*}
    \|P(\chi u)\|_{\Hch^{-2,-2,0,0}} &\leq \|\chi P u\|_{\Hch^{-2,-2,0,0}} + \|[P,\chi]u\|_{\Hch^{-2,-2,0,0}} \\
      &\lesssim \|h^{-1}x\chi P u\|_{\Hch^{-2,-1,0,-1}} + \|\tilde\chi u\|_{\Hch^{-1,-N,-N,-1}}
  \end{align*}
  for any $N\in\R$. The remainder term in~\eqref{EqHLEll} is simply estimated by
  \[
    \|R(\chi u)\|_{L^2} \leq C\bigl\|\tfrac{h}{h+x}\chi u\bigr\|_{L^2} \leq C\bigl\|\tfrac{h}{x}\chi u\bigr\|_{L^2}.
  \]
  Applying Hardy's inequality to this term, the estimate~\eqref{EqHLEll} then implies, a fortiori,
  \[
    \|\chi u\|_{L^2} \leq C'\bigl(\|h\nabla_g(\chi u)\|_{L^2}+\|h^{-1}x\chi P u\|_{L^2} + h\|\tilde\chi u\|_{L^2}\bigr).
  \]
  We can now add $\eta$ times this, with $\eta C'<\half$, to the estimate~\eqref{EqHLEst0} (in which we fix $\eps<C^{-1}$), in order to obtain
  \begin{equation}
  \label{EqHLEst1}
    \|\chi u\|_{\cD_h^1}^2 = \|\chi u\|^2 + \|h\nabla_g(\chi u)\|^2 \lesssim \|h^{-1}x\chi P u\|^2 + \|E u\|^2 + h\|\tilde\chi u\|^2.
  \end{equation}

  As far as weights in $h$ and $x$ are concerned, this is already the desired estimate. However, the differential order is forced to be $1$ here, and in addition the order of differentiability required on $P u$ in~\eqref{EqHLEst1} is too strong ($0$ instead of $-1$) even in this special case.
  
  \pfstep{Sharp improvement at $\tface$.} The basic idea is to apply the estimate~\eqref{EqHLEst1} to $\phi(h^2\Delta_g)u$, where $\phi\in\CIc(\R)$ is equal to $1$ on $[-4,4]$; on the remaining piece $(1-\phi(h^2\Delta_g))u$, the operator $h^2\Delta_g-1$ can be inverted directly using the functional calculus. In order to define $h^2\Delta_g$ as a self-adjoint operator, we need to pass from $X$ to a compact manifold $X'$ with boundary $\pa X'=\pa X$ and extend $g$ to a Riemannian metric on $X'$, which we continue to denote by $g$; the operator $\phi(h^2\Delta_g)$ does depend on the choice of extension, but its structural properties, as used in the following argument, do not.
  
  Concretely, $\Phi:=\phi(h^2\Delta_g)$ is given by Lemma~\ref{LemmaHLCalc} below, the notation of which we shall use here. Now, in order to remain localized near $\pa X$, we apply the estimate~\eqref{EqHLEst1} to
  \[
    u_1 := \tilde\chi\Phi\tilde\chi u.
  \]
  Using $[P,\Phi]\equiv 0$ and $\chi[P,\tilde\chi]\equiv 0$, we estimate the first term on the right in~\eqref{EqHLEst1} by
  \begin{align}
    &\|h^{-1}x\chi P\tilde\chi\Phi\tilde\chi u\| \nonumber\\
  \label{EqHLLocEst}
    &\quad \leq \bigl\|\tfrac{x}{h}\chi\Phi\tfrac{h}{h+x}\bigl(h^{-1}(x+h)\tilde\chi P u\bigr)\bigr\| + \|h^{-1}x\chi\Phi[P,\tilde\chi] u\|.
  \end{align}
  Denoting the lift of $x$ to the left, resp.\ right factor of $X_\chop^2$ by $x$, resp.\ $x'$, we note that
  \[
    \tfrac{x}{h}\tfrac{h}{x'+h} \in \cA^{1,1,0,0,0,0}(X^2_\chop) \ \implies\ 
    \Psi := \tfrac{x}{h}\chi\Phi\tfrac{h}{x'+h} \in \cA^{1-\eps,n+1-\eps,n-\eps,0,0,\infty}(X_\chop^2).
  \]
  Passing to a b-density $0<\mu_0\in\CI(X;\Omegab^1 X)$, we claim that $\Psi$ is continuous as a map
  \[
    \Hch^{-\infty,-1,0,0}(X;|\dd g|)=\Hch^{-\infty,-1-\frac{n}{2},-\frac{n}{2},0}(X;\mu_0)\to\Hch^{\infty,-\frac{n}{2},-\frac{n}{2},0}(X;\mu_0)=\Hch^{\infty,0,0,0}(X;|\dd g|);
  \]
  but since $\Psi$ is smoothing in the sense of $\chop$-differentiability, it suffices to show the boundedness on $L^2(X;\mu_0)$ of
  \begin{equation}
  \label{EqHLPsi}
    x^{\frac{n}{2}}\Psi(x')^{-\frac{n}{2}}\bigl(\tfrac{x'}{x'+h}\bigr)^{-1} \in \cA^{\frac{n}{2}+1-\eps,n-\eps,\frac{n}{2}-1-\eps,0,0,\infty}(X^2_\chop;\Omegab^{\frac12}).
  \end{equation}
  Since this kernel is bounded section of $\Omegab^{\frac12}$ (all indices being $\geq 0$), this is a consequence of Schur's lemma.

  The operator acting on $u$ in the second term on the right in~\eqref{EqHLLocEst} has Schwartz kernel supported in $x'\geq c>0$ and $|x-x'|>c>0$ (since $\supp\chi\cap\supp\dd\tilde\chi=\emptyset$), hence lies in $\cA^{1-\eps,\infty,\infty,\infty,\infty,\infty}(X_\chop^2)$; therefore, the second term in~\eqref{EqHLLocEst} can be bounded by $h^N\|\tilde\chi^\sharp u\|$ for any $N$, where $\tilde\chi^\sharp=1$ on $\supp\tilde\chi$. Altogether, forgetting the cutoff $\tilde\chi$ and renaming $\tilde\chi^\sharp$ as $\tilde\chi$, we have proved
  \begin{equation}
  \label{EqHLu1}
    \|\chi u_1\|_{\Hch^{N,1,0,0}} \lesssim h^{-1}\|(x+h)\tilde\chi P u\|_{\Hch^{-N,-1,0,0}} + \|E u\|_{\Hch^{-N,-N,-N,0}} + h^{\frac12}\|\tilde\chi u\|_{\Hch^{-N,0,0,0}}
  \end{equation}
  for any $s,N\in\R$.

  It remains to control $\chi u_2$, where
  \[
    u_2 := u-u_1 = u-\tilde\chi\Phi\tilde\chi u.
  \]
  Let $\phi^\flat\in\CIc((-3,3))$ be identically $1$ on $[-2,2]$, and let $\Phi^\flat=\phi^\flat(h^2\Delta_g)$. Then $\chi u_2$ is localized near high frequencies, in the sense that its localization to low frequencies
  \begin{equation}
  \label{EqHLu2phi}
    \Phi^\flat(\chi u_2) = \Phi^\flat(\chi u) + \Phi^\flat[\Phi,\chi]\tilde\chi u - \Phi^\flat\Phi \chi\tilde\chi u = \Phi^\flat[\Phi,\chi]\tilde\chi u
  \end{equation}
  (using $\Phi^\flat\Phi=\Phi^\flat$ and $\chi\tilde\chi=\chi$) is $\cO(h^\infty)$ (due to the presence of $[\Phi,\chi]$) near $x=0$ and vanishes to an order $h$ more than $u$ near $\supp\dd\chi\subset X^\circ$. Moreover, $\chi u_2$ satisfies the equation
  \begin{equation}
  \label{EqHLu2P}
    P(\chi u_2) = (\chi-\chi\Phi\tilde\chi)P u + \bigl([P,\chi]u - [P,\chi]\Phi\tilde\chi u - \chi\Phi[P,\tilde\chi]u\bigr),
  \end{equation}
  Altogether, if we put
  \[
    P^\sharp := P + 2\Phi^\flat,
  \]
  then we have
  \begin{equation}
  \label{EqHLPsharpEq}
    P^\sharp(\chi u_2) = f_2 := (\chi-\chi\Phi\tilde\chi)P u + [P,\chi]u - \bigl([P,\chi]\Phi\tilde\chi u + \chi\Phi[P,\tilde\chi]u - 2\Phi^\flat[\Phi,\chi]\tilde\chi u\bigr)
  \end{equation}
  We moreover have $P^\sharp=p^\sharp(h^2\Delta_g)$, where $p^\sharp(\sigma):=(\sigma-1) + 2\phi^\flat\geq\half(\sigma+1)$ for $\sigma\geq 0$; hence we can invert $P^\sharp$ using the functional calculus for $\Delta_g$ by $(P^\sharp)^{-1}=q^\sharp(h^2\Delta)$ where $q^\sharp(\sigma)=\frac{1}{p^\sharp(\sigma)}$ is equal to $(\sigma-1)^{-1}$ for large $\sigma$. One can then show, by a combination of the arguments leading to Lemma~\ref{LemmaHLCalc} and \cite[Theorem~5.2]{HintzConicPowers}, that
  \[
    (P^\sharp)^{-1} \in \bigl(\tfrac{x}{x+h}\bigr)^2\Psich^{-2}(X) + \Psich^{-\infty,\cE}(X)
  \]
  where collection $\cE$ of index sets is equal to $\cE(-1)$ in the notation of \cite[Theorem~6.1]{HintzConicPowers}. Therefore, using~\eqref{EqHLPsharpEq}, the mapping properties of $(\frac{x}{x+h})^2\Psich^{-2}(X)$, and estimating the smoothing contribution in the space $\Psich^{-\infty,\cE}(X)$ to $(P^\sharp)^{-1}$ by means of Schur's lemma, we have
  \begin{equation}
  \label{EqHLPu2}
  \begin{split}
    \|\chi u_2\|_{\Hch^{s,1,0,0}} &= \|(P^\sharp)^{-1}f_2\|_{\Hch^{s,1,0,0}} \\
      &\lesssim \|\tilde\chi P u\|_{\Hch^{s-2,-1,0,0}} + \|\tilde\chi[P,\chi]u\|_{\Hch^{s-2,-1,0,0}} + h\|\tilde\chi u\|_{\Hch^{-N,-N,0,0}}.
  \end{split}
  \end{equation}
  Here, the first term on the right comes from the first term in~\eqref{EqHLPsharpEq} and the boundedness of $\chi-\chi\Phi\tilde\chi^\sharp$ (with $\tilde\chi^\sharp\equiv 1$ on $\supp\tilde\chi$) on $\Hch^{s-2,-1,0,0}$; this boundedness follows from the boundedness of the Schwartz kernel of
  \[
    x^{\frac{n}{2}}\tfrac{x}{x+h}\Phi\bigl(\tfrac{x'}{x'+h}\bigr)^{-1}(x')^{-\frac{n}{2}}\in\cA^{\frac{n}{2}+2-\eps,n+1-\eps,\frac{n}{2}-1-\eps,0,0,\infty}(X^2_\chop;\Omegab^{\frac12})
  \]
  similarly to the discussion of $\Psi$ in~\eqref{EqHLPsi}. The final term in~\eqref{EqHLPu2} comes from the big parenthesis in~\eqref{EqHLPsharpEq}, every term of which involves the localizer $\Phi$ to low frequencies as well as a commutator with a cutoff $\chi$ or $\tilde\chi$. But $\tilde\chi[P,\chi]\in\Psich^{1,-\infty,-\infty,-1}(X)$, hence the second term on the right is bounded from above by $\|\tilde\chi u\|_{\Hch^{s-1,-N,-N,-1}}$ for any $N$. By elliptic regularity at infinite semiclassical cone frequencies, this can be bounded by $C(\|\tilde\chi^\sharp P u\|_{\Hch^{s-3,-N-2,-N,-1}}+\|\tilde\chi^\sharp u\|_{\Hch^{-N,-N,-N,-1}})$. Combining the resulting estimate with~\eqref{EqHLu1} proves the Theorem for $\Hch$-spaces. The proof of the more general statement for $H_{\cop\bop,h}$-spaces requires only notational changes which are left to the reader.
\end{proof}

\begin{lemma}[Functional calculus]
\label{LemmaHLCalc}
  Let $\phi\in\CIc(\R)$. Then for all $\eps>0$,
  \begin{equation}
  \label{EqHLCalcPhi}
    \phi(h^2\Delta_g) \in \cA^{-\eps,n-\eps,n-\eps,0,0,\infty}(X_\chop^2;\Omegab^{\frac12}),
  \end{equation}
  where the orders of the conormal space refer to $\lb_2,\ff_2,\rb_2,\tface_2,\dface_2,\sface_2\subset X_\chop^2$ in this order.
\end{lemma}
\begin{proof}
  This can be proved using the Helffer--Sj\"ostrand formula \cite{HelfferSjostrandSchrodinger} similarly to \cite[Lemma~10.1 and Proposition~10.2]{VasyThreeBody}. Choosing a compactly supported almost analytic extension $\tilde\phi\in\CIc(\C)$ of $\phi$ (that is, $\tilde\phi|_\R=\phi$ and $|\pa_{\bar z}\tilde\phi|=\cO(|\Im z|^N)$ for all $N$), we have
  \begin{equation}
  \label{EqHLCalc}
    \phi(h^2\Delta_g) = \frac{1}{2\pi i}\int \pa_{\bar z}\tilde\phi(z) (h^2\Delta_g-z)^{-1}\,\dd\bar z\wedge\dd z.
  \end{equation}
  For $z\notin\R$, \cite[Theorem~3.10 and \S6.1]{HintzConicPowers} gives $(h^2\Delta_g-z)^{-1}\in\bigl(\tfrac{x}{x+h}\bigr)^2\Psich^{-2}(X)+\Psich^{-\infty,\cE}(X)$ where $\cE=(\cE_\lb,\cE_\ff,\cE_\rb,\cE_\tface)$ with $\Re z\geq 0$ for $(z,k)\in\cE_\lb$; $\Re z\geq n$ for $(z,k)\in\cE_\rb$; $\Re z\geq 2$ for $(z,k)\in\cE_\ff$; and $\Re z\geq 0$ for $(z,k)\in\cE_\tface$. Since the principal symbol of $h^2\Delta_g$ is real-valued, any fixed seminorm of the two summands comprising $(h^2\Delta_g-z)^{-1}$ is moreover bounded by $|\Im z|^{-k}$ for some $k$. Plugging this into~\eqref{EqHLCalc} implies that $\phi(h^2\Delta_g)$ is of the same class as the resolvent. To improve the orders, let $m\in\N$ and write $\phi(\sigma)=(\sigma+C)^{-m}\phi_m(\sigma)$ with $C>-\inf\supp\phi$, then
  \[
    \phi(h^2\Delta_g) = (h^2\Delta_g+C)^{-m}\phi_m(h^2\Delta_g).
  \]
  Applying the previous discussion to $\phi_m$ and using \cite[Theorem~6.3(3)]{HintzConicPowers} (with $w=-m$) to control $(h^2\Delta_g+C)^{-m}$ implies, upon letting $m\to\infty$, that $\phi(h^2\Delta_g)\in\Psich^{-\infty,-\infty,0,0}(X)+\cA^{-\eps,n-\eps,n-\eps,0,\infty,\infty}(X^2_\chop)$, which gives~\eqref{EqHLCalcPhi}.
\end{proof}

\subsection{Scattering by potentials with inverse square singularities}
\label{SsHV}

Complex absorption is a somewhat drastic method for gaining microlocal control along incoming directions. As a more natural setting, let us thus consider scattering by potentials on $\R^n$, $n\geq 2$, which are singular at the origin $0\in\R^n$, as in Theorem~\ref{ThmIV}. (Working on more general conic manifolds requires only minor modifications.) That is, the underlying spatial manifold is
\begin{equation}
\label{EqHVMfd}
  X = [\R_x^n;\{0\}] \cong [0,\infty)_r\times\Sph^{n-1},\quad
  g=\dd x^2 = \dd r^2+r^2 g_{\Sph^{n-1}}.
\end{equation}
We write $\Delta\equiv\Delta_g=\sum D_{x^j}^2$ for the (nonnegative) Laplacian. Let $N\in\N$, and denote by $\Delta$ the Laplacian acting component-wise on $\C^N$-valued functions. We consider scattering by matrix-valued potentials
\[
  V(x)=|x|^{-2}V_0(x),\qquad V_0\in\CIc(X;\C^{N\times N}).
\]
The assumption of compact support of $V$ can of course be relaxed considerably, but since our interest lies in understanding the effect of the singularity at $r=0$, we shall not concern ourselves with more general conditions on $V$ at infinity here.

We are interested in high energy estimates for the resolvent of $\Delta+V$; concretely, we shall consider $\Delta+V-\sigma^2$, where $\Im\sigma>0$ is bounded and $|\Re\sigma|\gg 1$. Upon introducing
\begin{equation}
\label{EqHVResc}
  h=|\sigma|^{-1},\quad
  z=(h\sigma)^2=1+\cO(h),
\end{equation}
we define
\begin{equation}
\label{EqHVResc2}
  P_{h,z} := h^2(\Delta+V-\sigma^2) = h^2\Delta - z + h^2 r^{-2}V_0.
\end{equation}
This is admissible in the sense of Definition~\ref{DefPOp}, with $Q_{1,z}=V_0$ and $q_{0,z}=0$. Since $Q_{1,z}$ has differential order $0$, the threshold quantities in Definition~\ref{DefPThr} are $r_{\rm in}=r_{\rm out}=-\half$. The normal operator of $P_{h,z}$ is computed by passing to $\hat r=\frac{r}{h}$ and setting $h=0$:
\begin{equation}
\label{EqHVNormOp}
  N(P) = \Delta_{\hat g} - 1 + \hat r^{-2}V_\pa,\qquad \hat g:=\dd\hat r^2+\hat r^2 g_{\Sph^{n-1}},\quad V_\pa:=V_0|_{\pa X}\in\CI(\pa X;\C^{N\times N}).
\end{equation}
(Thus $V_\pa(\omega)=V_0(0,\omega)$ in the coordinates $(r,\omega)\in[0,\infty)\times\Sph^2$ on $X$.) 

\begin{thm}[Potential scattering]
\label{ThmHVProp}
  Assume that the operator $N(P)$ in~\eqref{EqHVNormOp} is invertible at weight $l\in\R$ in the sense of Definition~\usref{DefPNormOp}\eqref{ItPNormOpInv}. Let $C>0$, and let $\chi_0\in\CIc(X)$ be identically $1$ near $r=0$. Then there exists $C'>0$ so that for $0<\Im\sigma<C$ and $|\Re\sigma|>C'$, the operator $\Delta+V-\sigma^2$ is invertible as a map
  \begin{equation}
  \label{EqHVPropMap}
  \begin{split}
    \Delta+V-\sigma^2 \colon &\{ u \in H^2_\loc(X^\circ) \colon \chi_0 u \in r^l\Hb^2(X),\ (1-\chi_0)u\in H^2(\R^n) \} \\
      &\qquad \to \{ f \in L^2_\loc(X^\circ) \colon \chi_0 f\in r^{l-2} L^2(X),\ (1-\chi_0)f\in L^2(\R^n) \}.
  \end{split}
  \end{equation}
  Moreover, in the notation~\eqref{EqHVResc}--\eqref{EqHVResc2}, the following uniform estimate holds for all $\eps,\delta>0$, a suitable constant $C_{\eps,\delta}>0$, and all $0<\Im\sigma<C$, $|\Re\sigma|>C'$:
  \begin{equation}
  \label{EqHVPropSharp}
  \begin{split}
    &\| \chi_0 u \|_{\Hch^{s,l,\frac12+\eps,0}} + \|(1-\chi_0)u\|_{\Hsch^{s,-\frac12-\delta}} \\
    &\qquad \leq C_{\eps,\delta} h^{-1-2\eps}\Bigl(\| \chi_0 P_{h,z}u \|_{\Hch^{s-2,l-2,-\frac12-\eps,0}} + \|(1-\chi_0)P_{h,z}u \|_{\Hsch^{s-2,\frac12+\delta}}\Bigr);
  \end{split}
  \end{equation}
  here, $\Hsc^{s,\gamma}=\Hsc^{s,\gamma}(\ol{\R^n})=\la r\ra^{-\gamma}\cF^{-1}(\la h D\ra^{-s}L^2(\R^n))$ is the semiclassical scattering Sobolev space. In particular, for $l_1\leq\min(l,\half+\eps)$, $l_2\geq\max(l-2,-\half-\eps)$, and for any fixed $\chi\in\CIc(X)$ we have
  \begin{equation}
  \label{EqHVPropSimple}
    \|\chi(\Delta+V-\sigma^2)^{-1}\chi f\|_{r^{l_1} L^2} \leq C_{\eps,\chi}|\sigma|^{-1+2\eps}\|f\|_{r^{l_2}L^2}.
  \end{equation}
\end{thm}

\begin{rmk}[Meromorphic continuation]
  For $V$ with compact support as above, the resolvent $(\Delta+V-\sigma^2)^{-1}$ can be meromorphically continued to the complex plane when $n$ is odd, and the logarithmic cover of $\C^\times$ when $n$ is even; the estimate~\eqref{EqHVPropSharp} holds in strips of bounded $\Im\sigma$ for large $|\Re\sigma|$. The construction of this continuation can be accomplished along the lines of black box scattering \cite{SjostrandZworskiBlackBox} (see also \cite[\S4]{DyatlovZworskiBook}), with those estimates in the references in $\Im\sigma\gg 1$ relying on self-adjointness replaced by estimates on the off-spectrum resolvent that follow from \cite[Theorem~3.10]{HintzConicPowers}. For applications of such estimates to expansions of scattered waves for $n$ odd, we refer the reader to \cite[\S3.2.2]{DyatlovZworskiBook}. 
\end{rmk}

\begin{rmk}[Vector bundles]
  One may more generally consider potentials valued in endomorphisms of a vector bundle $E\to X$, with $\Delta$ denoting an operator acting on sections of $E$ with scalar principal symbol given by the dual metric function. The main difference to the case of a trivial bundle is that the threshold quantities $r_{\rm in},r_{\rm out}$ depend on subprincipal terms of $\Delta$ (and their calculation requires the choice of a fiber inner product on $E$, cf.\ Remark~\ref{RmkPBundle}). In the special case of tensor bundles $E$, and with $\Delta$ denoting the tensor Laplacian, the fiber inner product on $E$ induced by $g$ does give $r_{\rm in}=r_{\rm out}=-\half$.
\end{rmk}

\begin{proof}[Proof of Theorem~\usref{ThmHVProp}]
  Semiclassical propagation estimates near infinity of $\R^n$ are standard, see e.g.\ \cite[Theorem~1]{VasyZworskiScl} (following \cite{MelroseEuclideanSpectralTheory}) in a general geometric setting, and can be combined with the propagation estimates through the singularity at $r=0$ given in Theorem~\ref{ThmP} (where we shall take $\alpha=\half+\eps$, $\sfb=2\eps$ near $\cR_{\rm in}$, and $\sfb=0$ near $\cR_{\rm out}$). Altogether, upon simplifying to constant orders, we obtain, for any $\delta>0$, and for $0<h<h_0$ with $h_0>0$ sufficiently small,
  \begin{align*}
    &\| \chi_0 u \|_{\Hch^{s,l,\frac12+\eps,0}} + \|(1-\chi_0)u\|_{\Hsch^{s,-\frac12-\delta}} \\
    &\qquad \lesssim \| \chi_0 P_{h,z}u \|_{\Hch^{s-2,l-2,\frac12+\eps,1+2\eps}} + h^{-1-2\eps}\|(1-\chi_0)P_{h,z}u \|_{\Hsch^{s-2,\frac12+\delta}},
  \end{align*}
  which is the estimate~\eqref{EqHVPropSharp}. (The loss of $h^{-2\eps}$ in the second term on the right is due to the fact that the propagation through $r=0$ comes with this loss, which then gets propagated out to infinity.) This estimate also entails the injectivity of $P_{h,z}$ for small $h>0$ and $|z-1|<C h$, with surjectivity following from the analogous estimate for the adjoint; this proves the first part of the Theorem, albeit on function spaces with weights $\la r\ra^{\pm(\frac12+\delta)}$ at infinity. But for any fixed $\sigma$ with $\sigma^2\notin[0,\infty)$, $\Delta+V-\sigma^2$ is an elliptic scattering operator near infinity, hence these weights can be removed. (It is only in the high energy limit $|\Re\sigma|\to\infty$ with $|\Im\sigma/\Re\sigma|\to 0$ that one loses uniform (in $\sigma$) ellipticity.)
  
  The simplified estimate~\eqref{EqHVPropSimple} follows by setting $s=2$ in
  \begin{align*}
     h^{-1-2\eps}\|\chi P_{h,z}u\|_{\Hch^{s-2,l-2,-\frac12-\eps,0}} &\leq h^{1-2\eps}\|\chi(\Delta+V-\sigma^2)u\|_{\Hch^{s-2,l_2,l_2,0}}, \\
     \|\chi u\|_{\Hch^{s,l,\frac12+\eps,0}} & \leq \|\chi u\|_{\Hch^{s,l_1,l_1,0}}.\qedhere
  \end{align*}
\end{proof}

We describe a few scenarios in which the invertibility assumption on $N(P)$ can be verified:
\begin{enumerate}
\item The invertibility assumption on $N(P)$ is open in $V_\pa\in\CI(\Sph^{n-1};\C^{N\times N})$. In particular, it holds when $n\geq 3$, $l\in(1-\frac{n-2}{2},1+\frac{n-2}{2})$, and $V_\pa\equiv 0$ by Lemma~\ref{LemmaHL}, and therefore also when $\|V_\pa\|_{\cC^k}$ is sufficiently small (depending on $l$) for some sufficiently large $k$.
\item Consider $V_\pa$ which depends holomorphically on a parameter $w\in\Omega$, where $\Omega\subset\C$ is open and contains $0$. (For example, this is the case when $V_\pa(w)=w V_{\pa,0}$.) Let us write $N(P_w)$ for the $w$-dependent normal operator, and assume that $N(P_0)$ is invertible at weight $l_0$; assume moreover that there is a continuous function $l\colon\Omega\to\R$ with $l(0)=l_0$ so that $l(w)\notin-\Im\specb(N(P_w))$. Then there exists a discrete set $D\subset\Omega$ so that $N(P_w)$ is invertible at weight $l(w)$ for $w\in\Omega\setminus D$. This follows from analytic Fredholm theory in $w$; we leave the details to the reader.
\end{enumerate}

A very concrete third scenario is the following:

\begin{lemma}[Scalar inverse square potentials]
\label{LemmaHVScalar}
  Let $n\geq 2$, consider the scalar case $N=1$, and suppose $V_\pa=\sfZ\in\C\setminus(-\infty,-(\frac{n-2}{2})^2]$ is a constant (so $V(x)=\frac{\sfZ}{|x|^2}+\cO(|x|^{-1})$). Then $N(P)$ is invertible at weight $l$ iff $|l-1|<\Re\sqrt{(\frac{n-2}{2})^2+\sfZ}$.
\end{lemma}
\begin{proof}[Proof of Lemma~\usref{LemmaHVScalar}]
  The boundary spectrum of $N(P)$ can be computed, via expansion into spherical harmonics, as
  \[
    \specb(N(P)) = \left\{ i\Biggl(\frac{n-2}{2}\pm\sqrt{\Bigl(\frac{n-2}{2}+\ell\Bigr)^2+\sfZ}\,\Biggr) \colon \ell\in\N_0 \right\}.
  \]
  The condition on $\sfZ$ ensures that $\Re\sqrt{(\frac{n-2}{2})^2+\sfZ}>0$, and thus for $l$ as in the statement of the Lemma, one has $l-\frac{n}{2}\notin -\Im\specb(N(P))$.
  
  Expanding an outgoing solution $u$ of $N(P)u=0$ at weight $l$ into spherical harmonics, $u(\hat r,\omega)=\sum_{|m|\leq\ell} u_{\ell m}(\hat r)Y_{\ell m}(\omega)$, the coefficient $u_{\ell m}$ satisfies a Bessel ODE
  \[
    -u_{\ell m}'' - \frac{n-1}{\hat r}u_{\ell m}' + \frac{\ell(\ell+n-2)+\sfZ}{\hat r^2}u_{\ell m} = 0,
  \]
  hence is a linear combination of $\hat r^{-\frac{n-2}{2}}H^{(1)}_{\nu_\ell}(\hat r)$ and $\hat r^{-\frac{n-2}{2}}H^{(2)}_{\nu_\ell}(\hat r)$ for $\nu_\ell=\sqrt{(\frac{n-2}{2}+l)^2+\sfZ}$. The outgoing condition can only be satisfied if $u_{\ell m}$ is a multiple of $\hat r^{-\frac{n-2}{2}}H^{(1)}_{\nu_\ell}(\hat r)$. But for $0<\hat r\ll 1$, one has
  \[
    \bigl|\hat r^{-\frac{n-2}{2}}H^{(1)}_{\nu_\ell}(\hat r)\bigr| \geq c_\ell\hat r^{-\frac{n-2}{2}-\Re\nu_\ell} \geq c_\ell\hat r^{-\frac{n-2}{2}-\Re\nu_0}
  \]
  with $c_\ell>0$, which lies in $\hat r^{l'} L^2(\hat r^{n-1}\dd\hat r)$ iff $l'<1-\Re\nu_0$, which is violated for $l'=l$. Hence necessarily $u_{\ell m}=0$. This proves that $N(P)$ is injective at weight $l$ on outgoing functions; the injectivity of $N(P)^*$ at weight $-l+2$ on incoming functions is proved similarly.
\end{proof}

Theorem~\ref{ThmIV} follows from Theorem~\ref{ThmHVProp} and Lemma~\ref{LemmaHVScalar} upon taking $\sigma=\sqrt\lambda$ and $l=2$, which allows for the choice $l_1=l_2=0$. Note that for $l=2$, the target space in~\eqref{EqHVPropMap} in $L^2(X)=L^2(\R^n)$, and the domain is $H^2_0(\R^n\setminus\{0\})$ for $n\geq 5$ by Hardy's inequality.

\begin{rmk}[Multiple scatterers]
  By exploiting the diffractive improvement obtained in~\S\ref{SsPD} as in the work of Baskin--Wunsch \cite{BaskinWunschConicDecay}, it is conceivable that one can generalize (up to $\eps$-losses) Duyckaerts' high energy resolvent estimates \cite{DuyckaertsInvSq} for scattering by a finite number of real-valued inverse square potentials and analyze the complex-valued case $\sfZ_j\in\C\setminus(-\infty,-(\frac{n-2}{2})^2]$ (or more generally the case of finitely many matrix-valued inverse square potentials). However, the study of this problem exceeds the scope of this paper.
\end{rmk}

\subsection{Scattering for the Dirac--Coulomb equation}
\label{SsHDC} 

Motivated by recent work of Baskin and Wunsch \cite{BaskinWunschDiracCoulomb}, we consider the stationary scattering theory for the Dirac--Coulomb equation on Minkowski space at high energies. As discussed in~\S\ref{SI}, our framework allows us to deal directly with the associated matrix-valued Klein--Gordon operator---which has non-symmetric leading order terms at the Coulomb singularity---albeit with an arbitrarily small loss upon propagation through the singularity. Moreover, our results include a larger range of Coulomb charges $\sfZ\in\R$ than \cite{BaskinWunschDiracCoulomb} (which requires $|\sfZ|<\half$ for technical reasons); we can even allow for $\sfZ$ which $|\sfZ|>\frac{\sqrt{3}}{2}$, in which case the Dirac--Coulomb Hamiltonian is not essentially self-adjoint.

The underlying spatial manifold is again given by~\eqref{EqHVMfd}, now with $n=3$. We recall relevant notation from~\cite{BaskinWunschDiracCoulomb}. Denote the Pauli matrices by
\[
  \sigma_1=\begin{pmatrix} 0 & 1 \\ 1 & 0 \end{pmatrix},\quad
  \sigma_2=\begin{pmatrix} 0 & -i \\ i & 0 \end{pmatrix},\quad
  \sigma_3=\begin{pmatrix} 1 & 0 \\ 0 & -1 \end{pmatrix}.
\]
Put further
\[
  \beta:=\gamma^0:=\begin{pmatrix} I & 0 \\ 0 & -I \end{pmatrix},\quad
  \gamma^j:=\begin{pmatrix} 0 & \sigma_j \\ -\sigma_j & 0 \end{pmatrix},\quad
  \alpha^j:=\beta\gamma^j=\begin{pmatrix} 0 & \sigma_j \\ \sigma_j & 0 \end{pmatrix},\quad
  \alpha_r = \sum_{j=1}^3 \frac{x_j}{r}\alpha_j.
\]
The equation governing a massive Dirac field (with mass $m\in\R$) minimally coupled to an electromagnetic potential $\bfA=(A_0,A_1,A_2,A_3)$ is
\[
  (i\slpa_\bfA-m)\psi=0,\qquad
  \slpa_\bfA:=\gamma^\mu(\pa_\mu+i A_\mu),
\]
where $\psi$ takes values in $\C^4$. We now take
\begin{equation}
\label{EqHVPotential}
  A_0=\frac{\sfZ}{r}+V,\quad V\in\CI(X),\qquad
  A_j\in\CI(X),
\end{equation}
with $\sfZ\in\R$ the charge of the Coulomb field. As shown in~\cite[\S4.3]{BaskinWunschDiracCoulomb}, the operator $-(i\slpa_\bfA+m)(i\slpa_\bfA-m)$ is then of the form
\begin{equation}
\label{EqHDCKleinGordon}
  P = -\Bigl(D_t+\frac{\sfZ}{r}\Bigr)^2 + \Delta + m^2 + i\frac{\sfZ}{r^2}\alpha_r + r^{-1}\bfR, \qquad
  \bfR\in \Diffb^1(X;\C^4).
\end{equation}

Let us pass to a fixed temporal frequency $\sigma\in\C$, thus replacing $D_t$ in~\eqref{EqHDCKleinGordon} by $-\sigma$, resulting in the operator family $\hat P(\sigma)$. Introducing $h=|\sigma|^{-1}$ and $z=(h\sigma)^2$ and multiplying $\hat P(\sigma)=\hat P(h^{-1}\sqrt z)$ by $h^2$ gives
\begin{align}
\label{EqHDCOp}
  P_{h,z} &= h^2\Delta - \Bigl(\sqrt z-\frac{h\sfZ}{r}\Bigr)^2 + \frac{h^2}{r^2}\biggl(i\sfZ\alpha_r + r^2 m^2\biggr) + \frac{h^2}{r^2} r \bfR \\
    &= h^2\Delta - z + h^2 r^{-2}Q_{1,z} + h r^{-1}q_{0,z}, \nonumber\\
    &\qquad\qquad Q_{1,z}=-\sfZ^2 + i\sfZ\alpha_r + r^2 m^2 + r \bfR,\quad
    q_{0,z} = 2\sqrt{z}\sfZ. \nonumber
\end{align}
When $\Im\sigma$ is bounded while $|\Re\sigma|\to\infty$, one has $z=1+\cO(h)$; thus, $P_{h,z}$ is an admissible operator in the sense of Definition~\ref{DefPOp}. The threshold quantities in Definition~\ref{DefPThr} take the values
\[
  r_{\rm in}=r_{\rm out}=-\half
\]
since $q_{0,1}$ is real and the principal symbol of $Q_{1,z}$ (as a first order b-differential operator) vanishes at $r=0$.

The normal operator of $P_{h,z}$ is obtained by passing to $\hat r=\frac{r}{h}$ and restricting to $h=0$, giving in polar coordinates $(\hat r,\omega)\in[0,\infty]\times\Sph^2$
\begin{equation}
\label{EqHDCNormOp}
  N(P) = D_{\hat r}^2 - \frac{2 i}{\hat r}D_{\hat r} - \Bigl(1-\frac{\sfZ}{\hat r}\Bigr)^2 + \hat r^{-2}\bigl(\slDelta+i\sfZ\alpha_r(\omega)\bigr).
\end{equation}
For $\sfZ=0$, the operator $N(P)$ is equal to $(\Delta_{\hat g}-1)\otimes \Id_{\C^4}$ where $\hat g=\dd\hat r^2+\hat r^2 g_{\Sph^2}$, hence is invertible at weight $l\in(\half,\tfrac32)$ by Lemma~\ref{LemmaHL}. For fixed $l$, this will remain true for $\sfZ$ in a small neighborhood of $0$. The determination of the largest set of $\sfZ$ for which $N(P)$ is invertible at some weight requires explicit calculations:

\begin{lemma}[Invertibility of $N(P)$]
\label{LemmaHDCInv}
  Let $\sfZ\in\R$ be such that $|\sfZ|\neq\sqrt{\kappa^2-\frac14}$ for all $\kappa\in\N$. Then the operator $N(P)$ is invertible at weight $l=1$.
\end{lemma}

The conclusion of the Lemma in particular holds in the range $|\sfZ|<\frac{\sqrt 3}{2}$; this is related to the essential self-adjointness of Dirac operators, see \cite{WeidmannDirac,LeschFuchs} and \cite[\S4.1]{BaskinWunschDiracCoulomb}.

\begin{proof}[Proof of Lemma~\usref{LemmaHDCInv}]
  We begin by separating into spinor spherical harmonics following \cite[\S2.1]{BaskinWunschDiracCoulomb}: for
  \[
    \kappa\in\Z\setminus\{0\},\quad
    \mu\in\{-|\kappa|+\half,\ldots,|\kappa|-\half\},
  \]
  define the $\C^2$-valued function on $\Sph^2$
  \[
    \Omega_{\kappa,\mu}(\omega)
      =\begin{pmatrix}
         -\sgn(\kappa)\bigl(\tfrac{\kappa+\frac12-\mu}{2\kappa+1}\bigr)^{1/2}Y_{l,\mu-\frac12}(\omega) \\
         \bigl(\tfrac{\kappa+\frac12+\mu}{2\kappa+1}\bigr)^{1/2}Y_{l,\mu+\frac12}(\omega)
       \end{pmatrix},\qquad
    l=|\kappa+\half|-\half.
  \]
  Thus $\slDelta\Omega_{\kappa,\mu}=\kappa(\kappa+1)\Omega_{\kappa,\mu}$. Moreover,
  \[
    \alpha_r\begin{pmatrix} a\Omega_{\kappa,\mu} \\ b\Omega_{-\kappa,\mu'} \end{pmatrix} = \begin{pmatrix} -b\Omega_{\kappa,\mu'} \\ -a\Omega_{-\kappa,\mu} \end{pmatrix},
  \]
  as follows from~\cite[Equations~(3), (4), (9)]{BaskinWunschDiracCoulomb} or \cite[Equation~(3.1.3)]{SzmytkowskiSpinors}. Thus, the action of the spherical operator $\slDelta+i\sfZ\alpha_r\in\Diff^2(\Sph^2;\C^4)$ appearing in~\eqref{EqHDCNormOp} on the 2-dimensional space with basis $(\Omega_{\kappa,\mu},0)$ and $(0,\Omega_{-\kappa,\mu})$ is given by the matrix
  \begin{equation}
  \label{EqHDCMatrix}
    \begin{pmatrix}
      \kappa(\kappa+1) & -i\sfZ \\
      -i\sfZ & \kappa(\kappa-1)
    \end{pmatrix}.
  \end{equation}
  This can be diagonalized when $|\sfZ|\neq|\kappa|$, and it has eigenvalue $\lambda_\kappa^\pm=\kappa^2\pm\sqrt{\kappa^2-\sfZ^2}$ on the eigenspace spanned by
  \[
    \cY_{\kappa,\mu}^\pm(\omega) := \begin{pmatrix}
      i\sfZ \Omega_{\kappa,\mu}(\omega) \\
      (\kappa\mp\sqrt{\kappa^2-\sfZ^2})\Omega_{-\kappa,\mu}(\omega)
    \end{pmatrix},
    \qquad \mu\in\{-|\kappa|+\half,\ldots,|\kappa|-\half\}.
  \]
  Thus, the action of $N(P)$ on separated functions of the form $u(\hat r)\cY_{\kappa,\mu}^\pm(\omega)$ is given by the action on $u$ of the differential operator
  \[
    N_\kappa^\pm = D_{\hat r}^2-\frac{2 i}{\hat r}D_{\hat r}-\Bigl(1-\frac{\sfZ}{\hat r}\Bigr)^2+\hat r^{-2}\lambda_\kappa^\pm.
  \]
  The Mellin-transformed normal operator of $\hat r^2 N_\kappa^\pm$ at $\hat r=0$ is the polynomial
  \[
    \lambda^2 - i\lambda - \sfZ^2 + \lambda_\kappa^\pm,\qquad \lambda\in\C;
  \]
  for its roots, we have
  \[
    -(\Im\lambda) + \tfrac32 \in \Bigl\{ 1 - \bigl(\half \pm \sqrt{\kappa^2-\sfZ^2} \bigr),\ \ 1 + \bigl(\half \pm \sqrt{\kappa^2-\sfZ^2} \bigr) \Bigr\}
  \]
  Now if $\sfZ^2>\kappa^2$, then these two roots have real parts $\half$ and $\tfrac32$, whereas if $\sfZ^2<\kappa^2$, they are disjoint from an open interval $(1-\delta,1+\delta)$ around $1$ due to the assumption that $\kappa^2-\sfZ^2\neq\frac14$.

  An outgoing solution of $N(P)$ at weight $l=1$, expanded into the spherical eigenfunctions $\cY_{\kappa,\mu}^\pm$, is an outgoing solution of $N_\kappa^\pm$; one easily finds $u=u_\infty\hat r^{-1-i Z}e^{i\hat r}+\cO(\hat r^{-2})$ as $\hat r\to\infty$, where $u_\infty\in\C$, and the $\cO(\hat r^{-2})$ term is conormal at $\hat r=0$; near $\hat r=0$ on the other hand, we have $u=\cA^{-\frac12+\delta}([0,1)_{\hat r})$. A boundary pairing argument, i.e.\ the evaluation of
  \[
    0 = \lim_{\eps\to 0}\int_\eps^{1/\eps} \bigl((N_\kappa^\pm u)\bar u - u\ol{N_\kappa^\pm u}\bigr)\,\hat r^2\,\dd\hat r = 2 i\lim_{\eps\to 0}\Im(\hat r^2 u \ol{u'}|_\eps^{1/\eps}) = -2 i|u_\infty|^2,
  \]
  gives $u_\infty=0$, and thus $u\equiv 0$ by standard ODE analysis near $\hat r=\infty$. This shows that $N(P)$ is injective at weight $l=1$ on outgoing solutions. Since $(N_\kappa^\pm)^*=N_\kappa^\mp$ with respect to the $L^2(\hat r^2\,\dd\hat r)$ inner product, the injectivity of $N(P)^*$ at weight $-l+2$ on incoming functions is proved similarly. This completes the proof when $\sfZ$ is not a nonzero integer.

  When $\sfZ\in\Z\setminus\{0\}$ and $\kappa$ satisfies $|\sfZ|=|\kappa|$, then the action of $\slDelta+i\sfZ\alpha_r$ on the span of $(\Omega_{\kappa,\mu},0)$ and $(0,\Omega_{-\kappa,\mu'})$ is not diagonalizable anymore. By inspection of~\eqref{EqHDCMatrix}, it still has the eigenvalue $\kappa^2$ with eigenspace spanned by $\cY_{\kappa,\mu}^+=\cY_{\kappa,\mu}^-$. Let $\tilde\cY_{\kappa,\mu}=(\Omega_{\kappa,\mu},0)$, then an outgoing solution $u=u_1\cY_{\kappa,\mu}^+ +u_2\tilde\cY_{\kappa,\mu}$ of $N(P)$ satisfies a lower triangular ODE system, with a decoupled equation for $u_1$ which implies $u_1\equiv 0$ by the previous arguments, whence $u_2$ is now an outgoing solution to the same equation as $u_1$ and must therefore also vanish. The proof is complete.
\end{proof}

If we cut $\bfA$ off via multiplication by a cutoff $\chi\in\CIc(X)$, the operator $P_{h,z}$ is equal to $h^2\Delta-z$ near infinity and can thus be analyzed as in~\S\ref{SsHV}. In this setting, we thus obtain invertibility and quantitative estimates for $P_{h,z}$:

\begin{thm}[High energy estimates for the Dirac--Coulomb equation]
\label{ThmHDC}
  Suppose $\bfA=(A_0,A_1,A_2,A_3)$ as in~\eqref{EqHVPotential} has compact support. Let $\sfZ\in\R$ be such that $|\sfZ|\neq\sqrt{\kappa^2-\frac14}$ for all $\kappa\in\N$. Then for $l=1$ (and indeed for $l$ sufficiently close to $1$), the operator $P_{h,z}=h^2\hat P(h^{-1}z)$ defined in~\eqref{EqHDCKleinGordon}--\eqref{EqHDCOp} is invertible as a map between the spaces~\eqref{EqHVPropMap} and satisfies the uniform bound~\eqref{EqHVPropSharp} as well as the bound~\eqref{EqHVPropSimple} (with $\Delta_g+V-\sigma^2$ replaced by $\hat P(\sigma)$) for $l_1=l_2=0$.
\end{thm}

\begin{rmk}[Complex charges]
  One can also analyze the case of non-real $\sfZ\in\C$, in which case $r_{\rm in}=-\half+\Im\sfZ$ and $r_{\rm out}=-\half-\Im\sfZ$. The difference $r_{\rm out}-r_{\rm in}=-2\Im\sfZ$ results in an additional $2\Im\sfZ$ loss of powers of the semiclassical parameter $h$ when propagating through the singularity at $r=0$. Nonetheless, the invertibility of $N(P)$ automatically holds for values of $\sfZ$ close to those allowed in Theorem~\ref{ThmHDC}, as discussed prior to Lemma~\ref{LemmaHVScalar}.
\end{rmk}

\appendix
\section{A class of examples with sharp semiclassical loss}
\label{SLoss}

Note that the semiclassical order $\sfb$ in Theorem~\ref{ThmP} must decrease from $\cR_{\rm in}$ to $\cR_{\rm out}$ by more than
\begin{equation}
\label{EqLossD}
  D=\max(r_{\rm in}-r_{\rm out},0);
\end{equation}
thus, the estimate~\eqref{EqPThmFw} controls $u$ in $L^2$, say, microlocally near the flow-out of $\cR_{\rm out}$ by $h^{-D-\eps}$ (for any $\eps>0$) times the $L^2$-norm of the microlocalization of $u$ near the flow-in of $\cR_{\rm in}$. While in many natural settings, such as those discussed in~\S\ref{SH}, one has $D=0$, it is easy to construct examples where $D>0$. The following example (placed into a general context at the end of this appendix) shows that a loss of $h^{-D}$ typically \emph{does} occur, whence our estimates are sharp up to an $\eps$-loss. This $\eps$-loss may be avoidable, though we are not able to prove or disprove this here.

Consider $X=[0,2)_r$, $\mu=|\dd r|$, and
\[
  P_{h,z}=h^2(-\pa_r^2-\tfrac{2}{r}\pa_r)-z+\tfrac{2 i h}{r}q,\qquad z=1+\cO(h),
\]
where $q\in\C$ is a parameter. (The term in parentheses is the radial part of the Laplacian on $\R^3$ in polar coordinates.) The normal operator is
\[
  N(P)=-\pa_{\hat r}^2-\tfrac{2}{\hat r}\pa_{\hat r}-1+\tfrac{2 i}{\hat r}q.
\]
For $q=0$, the kernel of $N(P)$ is spanned by $\hat r^{-1}e^{\pm i\hat r}$; since $\hat r^{-1}$ barely fails to lie in $\hat r^{-\frac12}L^2([0,1)_{\hat r},|\dd\hat r|)$, it is easy to see that $N(P)$ is invertible at weight $l$ in the sense of Definition~\ref{DefPNormOp} for $l\in(-\half,\half)$; this persists for small values of $q\in\C$. (The boundary spectrum of $N(P)$ at $\hat r=0$ is independent of $q$.) In the notation of Definition~\ref{DefPThr}, we have
\begin{equation}
\label{EqLossThr}
  r_{\rm in}=\half+\Re q,\quad
  r_{\rm out}=\half-\Re q,
\end{equation}
so $D=\max(2\Re q,0)$. The quantities~\eqref{EqLossThr} correspond precisely to the $L^2$-decay rates of incoming and outgoing solutions $\hat u_{\rm in},\hat u_{\rm out}\in\ker N(P)$, which have the asymptotic behavior
\begin{equation}
\label{EqLossAsy}
  \hat u_{\rm in}\sim\hat r^{-1-\Re q}e^{-i\hat r},\quad
  \hat u_{\rm out}\sim\hat r^{-1+\Re q}e^{i\hat r},\qquad \hat r\to\infty.
\end{equation}
(We omit the explicit expressions involving confluent hypergeometric functions.)

We can now construct an element $\hat u\in\ker N(P)$ which lies in $\hat r^l L^2$, $l\in(-\half,\half)$, near $\hat r=0$ and which is uniquely specified by requiring its incoming data at $\hat r=\infty$ to be given by $\hat u_{\rm in}$. Indeed, $\hat u$ is necessarily of the form
\[
  \hat u(\hat r) = \hat u_{\rm in} + c\hat u_{\rm out},
\]
where $c\in\C$ is the `scattering matrix'; necessarily $c\neq 0$ (since $\hat u_{\rm in}$ fails to lie in $\hat r^l L^2$ near $\hat r=0$). But then
\[
  P_{h,1}u_h(r)=0,\qquad u_h(r):=\hat u(r/h).
\]
(This exact formula uses the invariance of $P_{h,1}$ under dilations in $(h,r)$.) Considering a neighborhood of $r=1$ then, the asymptotics~\eqref{EqLossAsy} for $u_{\bullet,h}(r):=u_\bullet(r/h)$, $\bullet=\rm in,out$, imply
\[
  u_h = u_{\rm in,h}+c u_{\rm out,h},\qquad u_{\rm in,h}\sim h^{1+\Re q}e^{-i r/h},\quad u_{\rm out,h}\sim h^{1-\Re q}e^{i r/h}\qquad\text{(near $r=1$)}.
\]
This demonstrates the loss of $h^{-2\Re q}$ between the amplitudes $h^{1+\Re q}$, resp.\ $h^{1-\Re q}$ of the incoming, resp.\ outgoing pieces of $u_h$. (The fact that there is in fact a gain when $\Re q<0$ is a peculiar feature of the 1-dimensional situation considered here: the characteristic set of $P_{h,z}$ has two connected components, with the incoming and outgoing radial sets lying in different components, and the monotonicity requirement in Theorem~\ref{ThmP} does not relate the two components.)

The same idea can applied to produce many more examples with sharp loss $D$. Indeed, when $N(P)$ is invertible at weight $l$, then the solution $\hat u=\hat u(\hat r,y)$ (with $y$ denoting points on $\pa X$) of $N(P)\hat u=0$, where $\hat u$ has prescribed incoming data and lies in $\hat r^l L^2$ near $\hat r=0$, gives rise to a solution $u_h(r,y)=\hat u(r/h,y)$ of $P_{h,1}u_h(r,y)=0$ where $P_{h,1}=N(P)$ (upon changing coordinates $\hat r=r/h$). The relative decay rates of incoming/outgoing solutions of $N(P)$ are then directly reflected in the relative semiclassical orders of $u_h$ near the flow-in/flow-out of the cone point. (Since the characteristic set of $P_{h,1}$ is connected when $\dim X\geq 2$, the loss is at least $0$, cf.\ \eqref{EqLossD}; after all, even away from the cone point, semiclassical regularity cannot improve under real principal type propagation.)

\bibliographystyle{alpha}

\end{document}